\documentclass[11pt,reqno]{amsart}

\numberwithin{equation}{section}

\usepackage{amsmath}
\usepackage{amsthm}
\usepackage{amsfonts}
\usepackage{amssymb}
\usepackage[margin=1.0in]{geometry}
\usepackage{graphicx}
\usepackage[unicode=true,bookmarks=false,breaklinks=false,pdfborder={0 0 1},backref=page,colorlinks=true]{hyperref}
\usepackage[mathscr]{euscript}
\usepackage{todonotes}

\usepackage{setspace}
\definecolor{skyblue}{rgb}{0.85,0.85,1} 
\newcommand{\GC}[1]{\todo[inline,color=skyblue]{GC: #1}}

\newcommand{\ran}{\operatorname{ran}}
\usepackage{stackengine}

\newtheorem{lemma}{Lemma}[section]
\newtheorem{prop}[lemma]{Proposition}

\newtheorem{theorem}[lemma]{Theorem}
\newtheorem{cor}[lemma]{Corollary}

\theoremstyle{definition}
\newtheorem{rem}[lemma]{Remark}
\newtheorem{define}[lemma]{Definition}
\newtheorem{example}[lemma]{Example}
\newtheorem{hyp}[lemma]{Hypothesis}

\DeclareMathOperator{\Sp}{span}

\DeclareMathOperator{\tr}{Tr}
\DeclareMathOperator{\dom}{dom}
\DeclareMathOperator{\sinc}{sinc}
\DeclareMathOperator{\sflow}{sf}
\DeclareMathOperator{\Mas}{Mas}

\newcommand{\bbC}{\mathbb{C}}
\newcommand{\bbN}{\mathbb{N}}
\newcommand{\bbR}{\mathbb{R}}

\newcommand{\bbT}{\mathbb{T}}
\newcommand{\bbZ}{\mathbb{Z}}

\renewcommand{\Re}{\operatorname{Re}}
\renewcommand{\Im}{\operatorname{Im}}

\newcommand{\cJ}{\mathcal{J}}

\newcommand{\cB}{\mathcal{B}}
\newcommand{\cD}{\mathcal{D}}
\newcommand{\cE}{\mathcal{E}}
\newcommand{\cF}{\mathcal{F}}
\newcommand{\cG}{\mathcal{G}}
\newcommand{\cL}{\mathcal{L}}
\newcommand{\cH}{\mathcal{H}}
\newcommand{\cK}{\mathcal{K}}
\newcommand{\pO}{\partial \Omega}

\newcommand{\Hp}{H^{1/2}(\pO)}
\newcommand{\HpR}{H^{1/2}(\pO,\bbR)}
\newcommand{\Hm}{H^{-1/2}(\pO)}
\newcommand{\HmR}{H^{-1/2}(\pO,\bbR)}

\newcommand{\p}{\partial}

\newcommand{\R}{\mathcal R}
\newcommand{\J}{\mathcal J}

\newcommand{\gW}{\mathbf W}

\newcommand{\llangle}{\langle\!\langle}
\newcommand{\rrangle}{\rangle\!\rangle}

\newcommand{\gaD}{\gamma_{{}_D}}

\newcommand{\gaN}{\gamma_{{}_N}}

 \allowdisplaybreaks

\begin{document}

\title[Fredholm determinants, Evans functions and Maslov indices]{Fredholm determinants, Evans functions and Maslov indices for partial differential equations}
\author[G. Cox]{Graham Cox}\email{gcox@mun.ca}\address{Department of Mathematics and Statistics, Memorial University of Newfoundland, St. John's, NL A1C 5S7, Canada}
\author[Y. Latushkin]{Yuri Latushkin}\email{latushkiny@missouri.edu}\address{Department of Mathematics, University of Missouri, Columbia, MO 65211, USA}
\author[A. Sukhtayev]{Alim Sukhtayev}\email{sukhtaa@miamioh.edu}\address{Department of Mathematics, Miami University, Oxford, OH 45056, USA}
\date{\today}

\begin{abstract}
The Evans function is a well known tool for locating spectra of differential operators in one spatial dimension. In this paper we construct a multidimensional analogue as the modified Fredholm determinant of a ratio of Dirichlet-to-Robin operators on the boundary. This gives a tool for studying the eigenvalue counting functions of second-order elliptic operators that need not be self-adjoint. To do this we use local representation theory for meromorphic operator-valued pencils, and relate the algebraic multiplicities of eigenvalues of elliptic operators to those of the Robin-to-Robin and Robin-to-Dirichlet operator pencils.  In the self-adjoint case we relate our construction to the Maslov index, another well known tool in the spectral theory of differential operators. This gives new insight into the Maslov index and allows us to obtain crucial monotonicity results by complex analytic methods.
\end{abstract}

\maketitle
\tableofcontents

\section{Introduction}
Let $\Omega \subset \bbR^n$ be a bounded Lipschitz domain, and consider the differential expression
\begin{equation}\label{Ldef}
	Lu = - \sum_{j,k=1}^n \p_j (a_{jk} \p_k u) + \sum_{j=1}^n b_j \p_j u - \sum_{j=1}^n\partial_j (d_j u)+ qu,
\end{equation}
where the coefficients $a_{jk}, b_j, d_j, q$ are complex-valued functions on $\Omega$. We are interested in the Dirichlet eigenvalue problem
\begin{equation}
	Lu = \lambda u, \qquad u\big|_{\pO} = 0,
\end{equation}
so we let $\cL^D$ denote \ the Dirichlet realization of $L$. Our goal is to construct a function that is analytic (expect for isolated singularities) whose zeros coincide with the eigenvalues of $\cL^D$, with the order of each zero equaling its algebraic multiplicity. This is achieved in Theorem \ref{thm:WEvans}, where we obtain such a function as the modified determinant of a certain ratio of Robin-to-Dirichlet operators. In Section~\ref{sec:background} we review the history of using  perturbation determinants for such questions, and put our results in context by comparing them to previous constructions in the literature.

We start with some standard assumptions on $L$.
\begin{hyp}
\label{hyp:L}
The coefficients of $L$ satisfy:
\begin{enumerate}
	\item $a_{jk}, b_j, d_j, q \in L^\infty(\Omega)$ for each $1 \leq j,k \leq n$;
	\item there is a constant $C>0$ such that 
\[
	\Re\sum_{j,k=1}^n a_{jk}(x)\bar{\xi}_j\xi_k\ge C\sum_{j=1}^n|\xi_j|^2 
\]
for all $(\xi_1,\dots,\xi_n)\in\bbC^n$ and $x\in\Omega$.
\end{enumerate}
\end{hyp}

We next fix a compact\footnote{The operator $\cL^\Theta$ can be defined under less restrictive assumptions on $\Theta$; see Hypothesis~\ref{hypT}. However, our main theorem requires the stronger Hypothesis~\ref{hyp:Bp}, which implies compactness of $\Theta$.} operator $\Theta \colon \Hp \to \Hm$, and let $\cL^\Theta$ denote the realization of $L$ with the generalized Robin boundary condition
\begin{equation}\label{cRc}
	\gaN^L u + \Theta \gaD u = 0,
\end{equation}
where $\gaN^L$ denotes the conormal derivative associated to $L$ and $\gaD$ is the Dirichlet trace (restriction to the boundary). 

For any complex number $\lambda$ not in the spectrum of $\cL^\Theta$, we define the Robin-to-Dirichlet operator $N_\Theta(\lambda) \colon \Hm \to \Hp$ as follows. For $g \in \Hm$ there exists a unique $u \in H^1(\Omega)$ such that $Lu = \lambda u$ weakly in $\Omega$ and $\gaN^L u +\Theta\gaD u = g$, so we define $N_\Theta(\lambda)g = \gaD u \in \Hp$. It is easy to see that $N_\Theta(\lambda)$ is bounded, and depends analytically on $\lambda$ (see Lemma \ref{lemRtoD}) in the resolvent set $\rho(\cL^\Theta)$. 

Similarly, for any $\lambda$ not in the spectrum of $\cL^D$, we define the Dirichlet-to-Robin operator $M_\Theta(\lambda) \colon \Hp \to \Hm$ by $M_\Theta(\lambda) g = \gaN^L u + \Theta \gaD u$, where $u \in H^1(\Omega)$ is the unique solution to the Dirichlet problem $Lu = \lambda u$ and $\gaD u = g$. If $\lambda \in \rho(\cL^\Theta) \cap \rho(\cL^D)$, then $N_\Theta(\lambda)$ and $M_\Theta(\lambda)$ are both defined and are mutually inverse.

We are interested in defining a functional determinant that vanishes at the eigenvalues of $\cL^D$. A natural choice would be the modified Fredholm determinant of $\J N_\Theta(\lambda)$, where $\J \colon \Hp \to \Hm$ denotes inclusion, as this operator is meromorphic in $\lambda$ and is not invertible at the eigenvalues of $\cL^D$. This is not possible, however, since $\J N_\Theta(\lambda) - I$ is not contained in the Schatten--von Neumann ideal $\cB_p\big(\Hm\big)$ for any $p>0$, so its $p$-modified Fredholm determinant is not defined. We therefore regularize $N_\Theta(\lambda)$ by dividing by the Robin-to-Dirichlet map for an auxiliary differential expression.

To that end, we let $\widehat L$ denote the differential expression
\begin{equation}\label{Lhdef}
	\widehat Lu = - \sum_{j,k=1}^n \p_j (\hat a_{jk} \p_k u) + \sum_{j=1}^n \hat b_j \p_j u - \sum_{j=1}^n\partial_j (\hat d_j u)+ \hat qu,
\end{equation}
with complex-valued coefficients $\hat a_{jk}, \hat b_j, \hat d_j, \hat q$ satisfying Hypothesis \ref{hyp:L}, and let $\widehat\Theta \colon \Hp \to \Hm$ be compact.
We thus obtain a Dirichlet-to-Robin map $\widehat M_{\widehat\Theta}(\lambda)$ that is analytic in $\rho(\widehat\cL^D)$ and fails to be invertible precisely at the eigenvalues of $\widehat\cL^\Theta$. Using this, we define the operator
\begin{equation}\label{eq:Edef}
	E(\lambda) = \widehat M_{\widehat\Theta} (\lambda) N_\Theta(\lambda)  \in \cB\big(\Hm\big).
\end{equation}
This is well defined and analytic for $\lambda \in \rho(\cL^\Theta) \cap \rho(\widehat \cL^D)$. Moreover, we have that $E(\lambda)-I$ is contained in $\cB_p\big(\Hm\big)$ for sufficiently large $p$, as long as the following assumption holds.

\begin{hyp}\label{hyp:Bp}
Assume, in addition to Hypothesis~\ref{hyp:L}, that:
\begin{enumerate}
	\item $a_{jk}$ are real-valued, Lipschitz and symmetric, i.e. $a_{jk} = a_{kj}$ for $1 \leq j,k \leq n$;
	\item $\Theta = \iota^*\tilde\Theta$, where $\tilde\Theta \in \cB\big(\Hp,L^2(\pO)\big)$ and \[\iota \colon \Hp \to L^2(\pO),\quad \iota^* \colon L^2(\pO)\to H^{-1/2}(\pO)\] are inclusions.
\end{enumerate}
\end{hyp}

A particular case is $\Theta = \iota^*\theta\iota$, where $\theta \in\cB\big(L^2(\pO)\big)$. Letting $\theta$ be the multiplication operator corresponding to a bounded function on $\pO$, we recover the classical Robin boundary condition as a special case.

Finally, we define the \emph{multiplicity} of a point $\lambda_0$ for a function with isolated zeros and singularities, following \cite[Section~4]{Howland}. Suppose $f\colon \mathcal \bbC \to\bbC$ is analytic except at isolated singularities and its zeros do not accumulate in $\bbC$. For each $\lambda_0 \in \bbC$ we define
\begin{equation}
\label{m:def}
	m(\lambda_0;f)=\frac{1}{2\pi i}\int_{\partial D(\lambda_0;\varepsilon)} \frac{f'(\lambda)}{f(\lambda)} \,d\lambda,
\end{equation}
with $\varepsilon$ chosen small enough that the punctured disk $D'(\lambda_0;\varepsilon)$ contains no zeros or singularities of $f$. In particular, if $f$ is meromorphic, then
\begin{equation}\label{m:def2}
	m(\lambda_0;f)=\begin{cases}
		k & \text{if $\lambda_0$ is a zero of order $k$},\\
		-k & \text{if $\lambda_0$ is a pole of order $k$},\\
		0 & \text{otherwise}.
	\end{cases}
\end{equation}
The generalization \eqref{m:def} of the multiplicity function in \eqref{m:def2} is relevant for our analysis, since the $p$-modified determinant of $E(\lambda)$ that we define in \eqref{def:cE} may have essential singularities; see Remark~\ref{rem:essential} and the discussion in Section~\ref{sec:prelimBp}.

We now state our main result.

\begin{theorem}\label{thm:WEvans}
Suppose $\Omega \subset \bbR^n$ is a bounded Lipschitz domain, and let $L$, $\Theta$ and $\widehat L$, $\widehat\Theta$ satisfy Hypothesis \ref{hyp:Bp}, with $d_j - \hat d_j$ Lipschitz and $a_{jk} = \hat a_{jk}$ for all $1 \leq j,k \leq n$.
\begin{enumerate}
	\item For each $\lambda \in \rho(\cL^\Theta) \cap \rho(\widehat \cL^D)$ and $p > 2(n-1)$ the operator $E(\lambda) - I$ is in the Schatten--von Neumann ideal $\cB_p\left(\Hm\right)$.
	\item If $p>2(n-1)$ is an integer, then
\begin{equation}
\label{def:cE}
	\cE(\lambda) := \operatorname{det}_p E(\lambda)
\end{equation}
defines an analytic function on $\rho(\cL^\Theta) \cap \rho(\widehat \cL^D)$.
	\item For each $\lambda_0 \in \bbC$ there exists a meromorphic function $\varphi$, defined in a neighbourhood of $\lambda_0$, such that
\begin{equation}\label{detorder2}
	\cE(\lambda)=(\lambda-\lambda_0)^d e^{\varphi(\lambda)},
\end{equation}
where
\begin{equation}
\label{detorder}
	d=m_a(\lambda_0, \cL^D) - m_a(\lambda_0, \widehat \cL^D) + m_a(\lambda_0, \widehat\cL^\mu) - m_a(\lambda_0, \cL^\mu)
\end{equation}
and $m_a(\lambda_0, \, \bullet )$ denotes the algebraic multiplicity of $\lambda_0$ for the operator $\bullet$. In particular, the multiplicity satisfies $m(\lambda_0;\cE)=d$.
\end{enumerate}
\end{theorem}

For instance, if $\lambda_0 \in \rho(\widehat\cL^D) \cap \rho(\cL^\Theta) \cap \rho(\widehat\cL^{\widehat\Theta})$, we have that $\cE(\lambda_0) = 0$ if and only if $\lambda_0$ is an eigenvalue of $\cL^D$, with the order of the zero equal to the algebraic multiplicity of $\lambda_0$. More generally, $\cE$ contains information about the spectra of all four operators appearing in \eqref{detorder}. We will explain below how to extract information from this formula by making appropriate choices of $\Theta$, $\widehat\Theta$ and $\widehat L$.

\begin{rem}
\label{rem:essential}
In Section~\ref{sec:det} we will see that the function $\lambda \mapsto E(\lambda)$ is completely meromorphic on $\bbC$. While the determinant $\cE$ is analytic wherever $E$ is, it is possible for $\cE$ to have essential singularities at the poles of $E$. This is a consequence of the definition of the $p$-modified determinant when $p>1$; see \cite[p.~333]{Howland} and Section \ref{sec:prelimBp} for further discussion.
\end{rem}

\begin{rem}
\label{rem:betterp}
If we assume additional regularity of $\pO$ and $L$, and better mapping properties of $\Theta$ and $\widehat\Theta$, the condition $p>2(n-1)$ can be weakened to $p > n-1$; see Hypothesis~\ref{hyp:smallerp} and Propositions~\ref{propBp} and~\ref{RBp}. The requirement $p>n-1$ cannot be improved without imposing further assumptions; in Section \ref{sec:disc} we give an example with analytic boundary for which $E(\lambda) - I \notin \cB_{n-1}$.
\end{rem}

\begin{rem}\label{MNNM0}
In Remark \ref{MNNM2} we will show that $E(\lambda)$ has the same $p$-modified Fredholm determinant as $N_\Theta(\lambda) \widehat M_{\widehat\Theta} (\lambda) \in \cB\big(\Hp\big)$. While the latter operator is perhaps more appealing, as it acts on functions in $\Hp$, rather than distributions in $\Hm$, we have defined $E(\lambda)$ by \eqref{eq:Edef} because this arises naturally in our discussion of the Maslov index; see Theorem~\ref{thm:WEequal}.
\end{rem}

There are two major steps in the proof of Theorem~\ref{thm:WEvans}:
\begin{enumerate}
	\item In Section \ref{sec:multiplicity} we relate the eigenvalues of $\cL^D$, $\widehat \cL^D$, $\cL^\Theta$ and $\widehat \cL^{\widehat\Theta}$ to the corresponding Dirichlet-to-Robin, Robin-to-Dirichlet and Robin-to-Robin maps. While the equality of geometric multiplicities is straightforward, the equality of algebraic multiplicities is significantly more involved, see Theorem \ref{multR}.
	\item In Section \ref{sec:Evans} we show that the $p$-modified Fredholm determinant of $E$ is well defined, and relate the order of its zeros, poles and essential singularities to the multiplicities of the corresponding eigenvalues.
\end{enumerate}

Combining Theorem~\ref{thm:WEvans} with the definition of $m$ in \eqref{m:def}, we get an eigenvalue counting formula.

\begin{cor}\label{cor:argument}
Let $K \subset \bbC$ be a compact set with rectifiable boundary. If $\p K$ is disjoint from $\sigma(\cL^D)  \cup \sigma(\widehat\cL^D) \cup \sigma(\cL^\Theta) \cup \sigma(\widehat\cL^{\widehat\Theta})$, then
\begin{equation}
\label{count}
	\frac{1}{2\pi i} \int_{\p K} \frac{\cE'(\lambda)}{\cE(\lambda)} \,d\lambda = m_a(K,\cL^D) - m_a(K,\widehat\cL^D) + m_a(K,\widehat\cL^{\widehat\Theta})- m_a(K,\cL^\Theta),
\end{equation}
where $m_a(K,\cL^\bullet)$ is the number of eigenvalues of $\cL^\bullet$ in $K$, counted with algebraic multiplicity.
\end{cor}

In general we are not interested in all four operators appearing on the right-hand side of \eqref{count}. However, we have considerable flexibility in applying this result, since $\Theta$, $\widehat\Theta$ and $\widehat L$ can be chosen as desired, subject to the hypotheses of the theorem. Depending on the application at hand, we can choose these to make one or more of the terms on the right-hand side vanish.

For instance, adding a positive constant to the coefficient $\hat q$ simply translates the spectra of $\widehat \cL^D$ and $\widehat\cL^{\widehat\Theta}$ to the right. Therefore, given a compact set $K$, we can choose $\hat q$ to ensure that $\sigma(\widehat\cL^D)$ and $\sigma(\widehat\cL^{\widehat\Theta})$ are disjoint from $K$, in which case \eqref{count} reduces to
\begin{equation}
\label{DNcounting}
	\frac{1}{2\pi i} \int_{\p K} \frac{\cE'(\lambda)}{\cE(\lambda)} \,d\lambda = m_a(K,\cL^D) - m_a(K,\cL^\Theta).
\end{equation}
When $\Theta=0$ this is simply the difference of the Dirichlet and Neumann counting functions, a quantity that has received much attention over the years; see, for instance \cite{AM12,BRS18,Fil04,F91,GM09,LR17,S08} and references therein. These papers only consider the self-adjoint case, whereas our counting formula \eqref{DNcounting} is valid for any $L$ and $\widehat L$ satisfying Hypothesis \ref{hyp:Bp}.

If we are only interested in the Dirichlet spectrum, we can gain further control by adjusting the operators $\Theta$ and $\widehat\Theta$. For a self-adjoint operator this is particularly easy, since the spectrum is real, and so to count eigenvalues it suffices to consider sets $K$ of the form $[\lambda_1,\lambda_2] \times [-\delta,\delta]$ for $0 < \delta \ll 1$.

We recall that $L$ is \emph{symmetric} if its coefficients satisfy $\bar a_{jk} = a_{kj}$, $\bar b_j = d_j$ and $\bar q = q$ for all $1 \leq j,k \leq n$. In this case the Dirichlet realization $\cL^D$ is selfadjoint, so the algebraic and geometric multiplicities coincide. We continue to denote this common multiplicity by $m_a$ to avoid confusion with the multiplicity function $m$ in \eqref{m:def}. Since we are interested in the spectrum of $\cL^D$, which is real, we will choose $\Theta$ so that the spectrum of $\cL^\Theta$ is \emph{not} on the real axis, and hence does not contribute to the right-hand side of \eqref{count}; see Lemma~\ref{resolventLTheta}.

An operator $\Theta \colon \Hp \to \Hm$ is said to be \emph{non-real} if $\Im\llangle \Theta g,g\rrangle \neq 0$ for any nonzero $g \in \Hp$, where $\llangle \cdot,\cdot\rrangle$ denotes the dual pairing of $\Hm$ and $\Hp$. Two important examples of non-real $\Theta$ are $i\R$ and $i\J$, where $\R \colon \Hp \to \Hm$ is the Riesz isomorphism and $\J=\iota^*\iota \colon \Hp \to \Hm$ is the inclusion.

\begin{theorem}\label{thm:SA}
Suppose, in addition to the hypotheses of Theorem~\ref{thm:WEvans}, that $L$ and $\widehat L$ are symmetric and $\Theta$, $\widehat\Theta$ are non-real. For any real numbers $\lambda_1 < \lambda_2$ in $\rho(\cL^D) \cap \rho(\widehat\cL^D)$, there exists $\delta>0$ such that $K = [\lambda_1,\lambda_2] \times [-\delta,\delta]$ is contained in $\rho(\cL^\Theta) \cap \rho(\widehat\cL^\Theta)$, and hence
\begin{equation}
\label{corLDsa}
	\frac{1}{2\pi i} \int_{\p K} \frac{\cE'(\lambda)}{\cE(\lambda)} \,d\lambda = m_a\big([\lambda_1, \lambda_2],\cL^D\big) - m_a\big([\lambda_1, \lambda_2],\widehat\cL^D\big).
\end{equation}
\end{theorem}

Finally, in the case that both $L$ and $\widehat L$ are symmetric, we discuss the relationship between our multi-dimensional Evans function $\cE(\lambda)$ and the Maslov index, a well-known tool for counting eigenvalues of self-adjoint operators.

The Maslov index is a topological invariant that counts intersections, with sign and multiplicity, of Lagrangian subspaces in a symplectic Hilbert space. This was first described in \cite{A67} in the finite-dimensional case, and applied to Sturm--Liouville problems in \cite{A85}. A survey of the Maslov index in infinite dimensions is given in \cite{F04} and \cite{BBZ18}. Much work has been done recently to apply this machinery to selfadjoint eigenvalue problems, starting with \cite{DJ11} and continuing in \cite {BCJLMS, CJM17,CJLS16,CJM15,CM19,HoSu20, HoSu22, LS18, LSS}.

To count eigenvalues for the Dirichlet problem, we consider the subspaces
\begin{equation}
\label{Fdef1}
	\cG(\lambda) = \big\{ (\gaD u, \gaN^L u) : u \in H^1(\Omega,\bbR), \ Lu = \lambda u \big\}
\end{equation}
and
\begin{equation}
	D = \big\{(0,g) : g \in \HmR \},
\end{equation}
so $\cG(\lambda)$ encodes solutions to $Lu = \lambda u$ with no boundary conditions imposed, and $D$ encodes the Dirichlet boundary condition. It follows that $\cG(\lambda)$ intersects $D$ nontrivially whenever $\lambda$ is a Dirichlet eigenvalue, and the multiplicity of $\lambda$ equals the dimension of $\cG(\lambda) \cap D$. 

When $\lambda$ is real these are Lagrangian subspaces in the symplectic Hilbert space $\cH = \HpR \oplus \HmR$. The Maslov index $\Mas(\cdot)$ of the path $\cG(\lambda)$ with respect to $D$ is then defined as the spectral flow $\sflow(\cdot)$ (through $-1$) of the family $W(\lambda)$ of unitary operators, called the {\em Souriau map},
\begin{equation}
\label{Souriaudef1}
	W(\lambda) = - (I-2P_{\cG(\lambda)})(I - 2P_D),
\end{equation}
where the $P_\bullet$ denote $\cH$-orthogonal projections \cite{BF98, BBFO01,F04}. Defining subspaces $\widehat \cG(\lambda)$ for $\widehat L$, analogous to \eqref{Fdef1}, we can now state our main result connecting the Evans function and the Maslov index.

\begin{theorem}
\label{thm:MorseEvans}
With all notation and assumptions as in Theorem \ref{thm:SA}, we have
\begin{equation}
\label{eq:EM}
	\textrm{Winding number of $\cE$ around $\p K$} = \Mas\left(\widehat \cG\big|_{\lambda_1}^{\lambda_2},D\right) - \Mas\left(\cG\big|_{\lambda_1}^{\lambda_2},D\right).
\end{equation}
\end{theorem}

That is, the winding of our multi-dimensional Evans function equals the difference of Maslov indices for the perturbed and unperturbed operators.  This formula is an immediate consequence of Theorem~\ref{thm:SA}, since it says the left-hand side of \eqref{eq:EM} equals the difference of counting functions, $m\big([\lambda_1, \lambda_2],\cL^D\big) - m\big([\lambda_1, \lambda_2],\widehat\cL^D\big)$, and it is known that the Maslov index satisfies
\begin{equation}
\label{eq:mMas}
	\Mas\left(\cG\big|_{\lambda_1}^{\lambda_2},D\right)
	= -m\big([\lambda_1, \lambda_2],\cL^D\big),
\end{equation}
and likewise for $\widehat\cL^D$; see \cite[Lemma~5.3]{CJLS16}, \cite[Lemma~4.1]{CJM15}, \cite[Lemma~4.7]{DJ11} or \cite[Theorem 3.3]{LS18}.

While this establishes the equality \eqref{eq:EM}, it does not given much insight into why it is true.
To better understand this connection, we prove this equality directly by deriving an explicit algebraic connection between the Evans function and the Maslov index. This is done in Sections~\ref{sec:Sou} and~\ref{sec:SA}; see in particular Theorem~\ref{thm:WEequal}, where we relate the operators $W$, $\widehat W$ and $E$.
This demonstrates a fundamental connection between the Evans function and the Maslov index, as opposed to the mere equality of indices seen in \eqref{eq:EM}. In terms of the family $W(\lambda)$ of unitary operators in \eqref{Souriaudef1}, the Maslov index is defined to be
\begin{equation}
\label{sfMaslov1}
	\Mas\left(\cG\big|_{\lambda_1}^{\lambda_2},D\right) = \sflow \left(W \big|_{\lambda_1}^{\lambda_2} , -1 \right).
\end{equation}
As a consequence of this definition and \eqref{corLDsa}, Theorem \ref{thm:MorseEvans} is equivalent to
\begin{equation}
\label{eq:MasEvans1}
	\frac{1}{2\pi i} \int_{\p K} \frac{\cE'(\lambda)}{\cE(\lambda)} \,d\lambda = 
	\sflow \left(\widehat W \big|_{\lambda_1}^{\lambda_2} , -1 \right)
	- \sflow \left(W \big|_{\lambda_1}^{\lambda_2} , -1 \right).
\end{equation}
This is the result we prove in Section \ref{sec:SA} by directly relating $W$, $\widehat{W}$ and $E$.

There are two steps to the proof. First, in Theorem \ref{thm:cont}, we give an explicit formula for $W$ in terms of the Robin-to-Dirichlet map $N_\Theta$. In addition to clarifying the relationship between these objects, this explicit formula allows us to define $W(\lambda)$ for complex values of $\lambda$. 
The second step is to prove that the spectral flow for $W(\lambda)$ is monotone, in the sense that its eigenvalues (which must lie on the unit circle whenever $\lambda$ is real) always pass though $-1$ in the same direction as $\lambda$ increases. We prove this monotonicity using the analytic continuation of $W(\lambda)$. This continuation has the property that its spectrum lies outside the unit circle whenever $\lambda$ is in the upper half plane, and inside the unit circle whenever $\lambda$ is in the lower half plane. Combined with analyticity, this spectral mapping property gives the desired monotonicity of the spectral flow and completes the proof of \eqref{eq:MasEvans1}, and hence of Theorem~\ref{thm:MorseEvans}.

We expect that the analytic continuation of $W(\lambda)$, which to the best of our knowledge has not previously appeared in the literature, will be a valuable new tool for studying the Maslov index. For instance, it gives an elegant proof of monotonicity for the spectral flow that seems completely different from the existing proofs in the literature, which use the method of crossing forms, cf. \cite{CJLS16,CJM15,DJ11,HoSu20, HoSu22, LS18, LSS}.

\subsection*{Outline of paper}
In Section~\ref{sec:background} we put our results in context by recalling relevant literature on perturbation determinants in mathematical physics.
In Section~\ref{sec:examples} we illustrate our constructions and results using simple examples where everything can be computed explicitly.  In particular, we show how our construction is related the ``standard'' Evans function in one spatial dimension. In Section~\ref{sec:prelim} we precisely define all of the relevant operators, and establish needed properties of the Dirichlet-to-Robin and Robin-to-Dirichlet maps. In Section~\ref{sec:multiplicity} we discuss the connection between 
these maps and the eigenvalues of $\cL^D$ and $\cL^\Theta$. In Section \ref{sec:Evans} we establish the $\cB_p$ properties needed to define the modified Fredholm determinant of $E(\lambda)$, then study the zeros and singularities of the resulting Evans function $\cE(\lambda)$, culminating in the proof of Theorem~\ref{thm:WEvans}. Finally, in Sections~\ref{sec:Sou} and~\ref{sec:SA} we discuss the connection to the Maslov index, proving Theorem~\ref{thm:MorseEvans}.

\subsection*{Notation and conventions}
Unless explicitly stated otherwise, all functions are complex-valued, so we abbreviate $H^1(\Omega,\bbC) = H^1(\Omega)$ etc. Sesquilinear forms, such as the inner product, are always linear in the first argument and antilinear in the second. The inner product in $L^2(\Omega)$ will always be denoted $\langle\cdot\, ,\, \cdot\rangle$.

Following the convention of \cite{K76}, we define the adjoint $X^*$ of a Banach space $X$ to be the set of bounded \emph{antilinear} functionals on $X$. In particular, we let $\Hm = \Hp^*$, and use the notation $\llangle \varphi, g \rrangle = \varphi(g)$ to denote the action of $\varphi \in \Hm$ on $g\in \Hp$. The Riesz representation theorem therefore gives a \emph{linear} map $\R \colon \Hp \to \Hm$ satisfying
\begin{equation}
	\llangle \R f, g \rrangle = \left<f,g\right>_{\Hp}.
\end{equation}
There is also a linear map from $\Hp$ to $\Hp^{**} = \Hm^*$ that sends $f \in \Hp$ to the antilinear functional $\Hm \ni \varphi \mapsto \overline{\llangle \varphi, f\rrangle}$. Therefore, given a bounded operator $\Theta \colon \Hp \to \Hm$, we can identify its adjoint with a map $\Theta^* \colon \Hp \to \Hm$ satisfying
\begin{equation}
	\llangle\Theta f,g\rrangle = \overline{\llangle \Theta^* g,f\rrangle}
\end{equation}
for all $f,g \in \Hp$.
Finally, we recall the compact embedding $\iota \colon \Hp \to L^2(\pO)$. Using the Riesz theorem to identify $L^2(\pO)$ with $L^2(\pO)^*$, we identify $\iota^*$ with the map $L^2(\pO) \to \Hm$ that sends $f \in L^2(\pO)$ to the functional $g \mapsto \left<f,\iota g\right>$ on $\Hm$. We thus obtain a compact linear operator $\J := \iota^* \iota \colon \Hp \to \Hm$, acting as
\begin{equation}
	\llangle \J f,g \rrangle = \left<\iota f, \iota g\right>_{L^2(\pO)}
\end{equation}
for all $f,g \in \Hp$.

We denote by $\sigma(\cdot)$ and $\rho(\cdot)$ the spectrum and resolvent set of an operator. The algebraic multiplicity of an eigenvalue is denoted $m_a(\lambda, \cdot)$, and the $p$-modified Fredholm determinant is $\det_p(\cdot)$.

\section{Background and motivation}
\label{sec:background}
Perturbation determinants are among the most popular tools of modern mathematical physics \cite{GoKr, K76, Yaf}. In particular, for a given perturbed operator on a Hilbert space, they provide an analytic (or meromorphic) function of the spectral parameter whose zeros (or poles) give the location of discrete eigenvalues. A typical example of this is an application of the important Birman--Schwinger principle for the Schr\"odinger operator $\cL=-\Delta+V$, whose discrete spectrum under appropriate assumptions is given by zeros of the perturbation determinant $\det\big(I+V(-\Delta-z)^{-1}\big)$; this follows from the identity $\cL-z=\big(I+V (-\Delta-z)^{-1}\big)(-\Delta-z)$ for $z$ in the resolvent set of $-\Delta$. To avoid citing quite a number of influential papers on this classical topic we refer the reader to \cite{BTG20} where one can find an extensive bibliography. 

The perturbation determinants of this type are quite ``large" as the operators involved are acting on the whole Hilbert space, say $L^2(\Omega)$ for a domain $\Omega \subset \bbR^n$. It is therefore natural to attempt to reduce the space by passing to a ``smaller" space of functions on the boundary $\pO$.
A typical example of this reduction in the ODE case is furnished by one-dimensional scattering theory \cite[Chapter~17]{ChadSab}, where the discrete spectrum of a one-dimensional 
Schr\"odinger operator is given by the zeros of the classical Jost function, an analytic function of the spectral parameter constructed by means of solutions to the spectral problem with appropriate asymptotic behavior.

More generally, in the theory of one-dimensional ODEs with asymptotically constant or periodic coefficients, the spectrum of differential operators is given by zeros of the Evans function, a powerful and popular tool in the stability theory of traveling waves \cite{AlGaJo90, KPP13, PeWe92, S02}. An important fact that goes back to the classical work of Jost and Pais, cf.\ \cite{GLMZ05, GLMZ20, MR2347795,  GMZ1}, and was established in a general Evans function setting in \cite{GLM07}, is that the Birman--Schwinger-type perturbation determinant is actually equal to the Jost and Evans functions. Unfortunately, the one-dimensional construction of the Evans function does not admit an easy multidimensional generalization, although, in a slightly different direction, there were several attempts to construct direct analogs of the Evans function for PDE situations \cite{DN06,DN08} and, in particular, for infinite cylinders \cite{LP15}. 

Amid the attempts to generalize the notions of the Jost and Evans function to the PDE context, yet another classical set of objects took a prominent role, those being the Dirichlet-to-Neumann and Robin-to-Robin maps. We refer to \cite{GLMZ05, GLMZ20, GMZ1},   with a later spectacular development of this topic in \cite{MR3831271,BTG20,MR2347795,MR2569392,MR2732078}. This is quite helpful as the Robin-to-Robin maps and their abstract versions, such as abstract Weyl functions in the theory of abstract boundary triplets \cite{Behrndt_2020,Schm12}, are acting on ``smaller" boundary spaces, for instance $L^2(\partial\Omega)$ as described above.

In this paper we construct an analog of the Evans function for multidimensional elliptic operators that need not be self-adjoint. In some sense the current paper stems from \cite{GLMZ05, GMZ1}, except instead of using the Birman--Schwinger-type perturbation determinants we use the Dirichlet-to-Neuman and Robin-to-Robin maps to directly construct meromorphic operator valued functions whose zeros and poles are the eigenvalues of the respective differential operators. We construct a function (which we call an analog of the Evans function) as a Fredholm determinant of a particular Robin-to-Robin map. Because these are infinite dimensional operators we must use a suitable regularization to be in the class of operators for which Fredholm determinants are defined.  As a result, our Evans function is analytic except at isolated singularities that could be either poles or essential singularities.

In addition, we need to relate the spectra of the Robin and Dirichlet realizations of $L$ to the spectra of the respective Dirichlet-to-Robin, Robin-to-Robin and Robin-to-Dirichlet maps. This relation is well known in the selfadjoint case, for instance when $L$ is the Laplacian \cite{GMZ1}. It was recently studied in an important paper \cite{BE19} for non-selfadjoint operators, with a particular choice of Robin boundary condition. We prove here quite general results relating the Jordan chains (and hence algebraic multiplicities of eigenvalues) of the Robin realization of $L$ and the Robin-to-Robin operator pencil. Our approach stems from \cite{BE19} but provides a unified approach to all types of boundary conditions for general (not necessarily selfadjoint) elliptic operators. Our main technical tool is the theory of local equivalency of nonlinear operator pencils that goes back to \cite{GohSig}, see also \cite[Chapter IX]{GGK90}, \cite{Howland}, \cite{BTG20}, \cite{LS10} and the vast literature therein.

Yet another interesting connection we explore is related to the Maslov index, a notion from infinite-dimensional symplectic geometry \cite{BBZ18,F04} that has been widely used in stability theory and related problems when the operator in question is selfadjoint \cite{BCJLMS, CJM17,CJLS16,CJM15,DJ11,HoSu20, HoSu22,LS18, LSS}. There one uses the Souriau map \cite{F04}, a double reflection operator whose spectrum detects the fact that two Lagrangian subspaces have a nontrivial intersection. (This detection is exactly what the standard ODE Evans function does, though in the non-selfadjoint case there is no Lagrangian structure.) The Lagrangian subspaces that we encounter happen to be the graphs of certain Robin-to-Robin maps, and we notice and exploit a relationship between the Souriau map and the Cayley transform of the Robin-to-Robin operator to establish appropriate properties of the Evans function. Even for the well-studied ODE situation we establish a new, direct relation between the standard Evans function and the Maslov index.

\section{Examples}
\label{sec:examples}

Before proving the main results we give a few examples. The first two are one-dimensional, and hence avoid the technical difficulties associated with determinants in infinite dimensions. We start with the Laplacian on an interval, for which the Dirichlet-to-Neumann can be explicitly computed. This example demonstrates the role of the boundary operator $\Theta$, and shows how it can be chosen to separate the spectrum of $\cL^D$ from that of $\cL^\Theta$.

The second example is a matrix-valued Schr\"odinger operator on an interval. This example clarifies the relationship between our function $\cE$ and the ``standard"  Evans function in one dimension, as described in \cite{AlGaJo90,KPP13, S02}, for instance. In particular, we show that $\cE$ can be written explicitly in terms of four separate Evans functions, corresponding to a perturbed and unperturbed differential operator with two different sets of boundary conditions.

Our final example is the Laplacian on the unit disc. Unlike the first two examples, the Dirichlet-to-Neumann map is now an operator acting on an infinite-dimensional space, as opposed to a matrix, and so one needs to worry about its $\cB_p$ properties. Using known asymptotic formulas for the eigenvalues of the Dirichlet-to-Neumann map, we find that $E(\lambda) - I$ is of class $\cB_p$ if and only if $p>1$, as claimed in Remark \ref{rem:betterp}. We also discuss Schr\"odinger operators with radial potentials, and relate $E(\lambda)$ to an infinite collection of one-dimensional Evans functions.

\subsection{The interval}
\label{sec:interval}
We first consider $L = -(d/dx)^2$ on the unit interval $\Omega = (0,1)$. Since $\pO = \{0,1\}$, the Robin-to-Dirichlet map $N_\Theta(\lambda)$ is a $2\times2$ matrix. It thus has a well-defined determinant, so there is no need to introduce a reference operator $\widehat L$ and the corresponding Dirichlet-to-Robin map, as in \eqref{eq:Edef}, and we can illustrate our results by simply computing $\det N_\Theta(\lambda)$.

The Dirichlet spectrum is $\sigma(\cL^D) = \{ (n\pi)^2 : n \in \bbN \}$. For $\lambda \in \rho(\cL^D)$, the Dirichlet-to-Neumann map $M(\lambda)$ is given by
\begin{equation}\label{defMLAM}
	M(\lambda) = \begin{bmatrix} a & - b \\ -b & a \end{bmatrix},
\end{equation}
where
\[
	a = \frac{\cos \sqrt\lambda}{\sinc \sqrt\lambda}, \qquad b = \frac{1}{\sinc \sqrt\lambda},
\]
and $\sinc z = z^{-1} \sin z$ is the unnormalized sinc function; see \cite[Section~3]{Daners}. Here $\cos\sqrt\lambda$ and $\sinc\sqrt\lambda$ are shorthand for the power series
\begin{equation}
\label{series}
	\cos \sqrt\lambda = \sum_{n=0}^\infty \frac{(-1)^n \lambda^n}{(2n)!} , \qquad \sinc\sqrt\lambda = \sum_{n=0}^\infty \frac{(-1)^n \lambda^n}{(2n+1)!},
\end{equation}
which define entire functions of $\lambda$ and do not depend on a choice of square root. Note that $\sinc\sqrt\lambda = 0$ if and only if $\lambda \in \sigma(\cL^D)$.

For any matrix $\Theta = \left[ \begin{smallmatrix} t_{11} & t_{12} \\ t_{21} & t_{22} \end{smallmatrix}\right]$, the Dirichlet-to-Robin map $M_\Theta(\lambda)$ is given by
\[
	M_\Theta(\lambda) = M(\lambda) + \Theta = \begin{bmatrix} a + t_{11} & - b + t_{12} \\ -b + t_{21} & a + t_{22} \end{bmatrix}.
\]
It is easy verified that $b^2 - a^2 = \lambda$. Since $N_\Theta(\lambda)$ is inverse to $M_\Theta(\lambda)$ whenever both operators are defined, we find that
\begin{equation}
\label{NdetFD}
	\det  N_\Theta(\lambda) = \big(\det M_\Theta(\lambda)\big)^{-1} = \frac{\sinc\sqrt\lambda}{(\det\Theta-\lambda) \sinc\sqrt\lambda + \operatorname{tr}\Theta \cos\sqrt\lambda + (t_{12} + t_{21})}.
\end{equation}

As noted above, the numerator vanishes on the spectrum of $\cL^D$, and it can be shown that the denominator vanishes on the spectrum of $\cL^\Theta$. Therefore, the determinant has zeros in $\sigma(\cL^D) \setminus \sigma(\cL^\Theta)$, poles in $\sigma(\cL^\Theta) \setminus \sigma(\cL^D)$, and removable singularities in $\sigma(\cL^D) \cap \sigma(\cL^\Theta)$, as expected from Theorem~\ref{thm:WEvans}. 

This example demonstrates the importance of the boundary operator $\Theta$. If we had chosen $\Theta=0$, so that $\cL^\Theta = \cL^N$ is just the Neumann Laplacian on $(0,1)$, the determinant would have no zeros, because $\sigma(\cL^D) \subset \sigma(\cL^N)$, and its only pole would be at $\sigma(\cL^N) \setminus \sigma(\cL^D) = \{0\}$. Indeed, it follows immediately from \eqref{NdetFD} with $\Theta=0$ that
\begin{equation}
	\det N(\lambda) = - \frac{1}{\lambda}.
\end{equation}

As noted in the introduction, it is advantageous to choose $\Theta$ so that $\sigma(\cL^D)$ and $\sigma(\cL^\Theta)$ are disjoint. This is easily done in the symmetric case, where $\cL^D$ is self-adjoint and hence has real eigenvalues; see Theorem \ref{thm:SA}. For instance, choosing $\Theta=i I_{2\times2}$, we find that $\lambda$ is an eigenvalue of $\cL^\Theta$ if and only if $(\lambda+1) \sinc\sqrt\lambda - 2i \cos\sqrt\lambda=0$,
which is clearly not possible for any real $\lambda$. It follows that $\det N_\Theta$ is analytic in an open neighborhood of the real axis, and in this neighborhood we have $\det N_\Theta(\lambda) = 0$ if and only if $\lambda$ is an eigenvalue of $\cL^D$.

\subsection{The Evans function for a Schr\"odinger operator}
We next consider the Schr\"odinger operator $Lu=-u''+Qu$, with continuous $\bbC^{n\times n}$-matrix valued potential $Q$.
Since in the main part of the paper we consider only bounded domains, in this section we deal with the bounded interval $\Omega=(-1,1)$; however, the generalization to the line $(-\infty, \infty)$ is not hard. 

Our main result for this problem compares our generalized Evans function $\cE(\lambda)$ to the ``standard" Evans functions for the Dirichlet and Robin problems,  as defined in \cite{AlGaJo90, KPP13, S02}, which we now recall. For the Dirichlet Evans function we let $Y_\pm(\cdot,\lambda)$ be the $\mathbb{C}^{n\times n}$-matrix valued solutions of $LY_\pm = \lambda Y_\pm$ with boundary conditions
\[
	Y_-(-1,\lambda) = Y_+(1,\lambda) = 0_{n\times n} \qquad\text{and}\qquad 
Y_-'(-1,\lambda) = Y_+'(1,\lambda)=I_{n\times n}.
\]
The Dirichlet Evans function, which we denote $\mathsf E_{_D}$, is then defined to be the determinant of the Evans matrix
\begin{equation}
\label{def:EDmatrix}
	\mathbf{E}_D(x,\lambda)= \begin{bmatrix} Y_-(x,\lambda) & Y_+(x,\lambda) \\
	Y'_-(x,\lambda) & Y'_+(x,\lambda)\end{bmatrix}.
\end{equation}
The determinant is easily shown to be independent of $x$, so we have
\begin{equation}
\label{ED}
	\mathsf E_{_D}(\lambda) = \det\begin{bmatrix} Y_-(1,\lambda) &0_{n\times n}\\ Y'_-(1,\lambda) &I_{n\times n} \end{bmatrix} = \det Y_-(1,\lambda).
\end{equation}
Similarly, we let $V_\pm(\cdot,\lambda)$ be the matrix-valued solutions of $LV_\pm = \lambda V_\pm$ with boundary conditions
\[
	V_-(-1,\lambda) = - V_+ (1,\lambda) = I_{n\times n}\qquad\text{and}\qquad 
	V_-'(-1,\lambda) = V_+'(1,\lambda) = 0_{n\times n}.
\]
For the Robin problem we fix $n\times n$ matrices $\Theta_\pm \in M_n(\bbC)$ and let $\Theta={\rm diag } \{ \Theta_+, \Theta_-\}$. Then $W_\pm := V_\pm + Y_\pm \Theta_\pm$ satisfy the eigenvalue equation $LW_\pm = \lambda W_\pm$ and the Robin 
boundary conditions
\[
	-W'(-1,\lambda)+\Theta_-W(-1,\lambda) = 
	W'(1,\lambda)+\Theta_+W(1,\lambda) = 
	0_{n\times n},
\]
so we define the Evans function for the $\Theta$-Robin eigenvalue problem as
\begin{equation}
\label{ETheta}
	\mathsf E_{\Theta}(\lambda) = \det\begin{bmatrix}W_-(x,\lambda) &
			W_+(x,\lambda) \\ W'_-(x,\lambda) & W'_+(x,\lambda)
			\end{bmatrix},
\end{equation}
which is again independent of $x$.

We now relate these Evans functions to the function $\cE$ defined by \eqref{def:cE} with $p=1$. As in the previous example, the Dirichlet-to-Robin and Robin-to-Dirichlet maps are matrices, and hence have well-defined determinants, so we can write
\begin{equation}
\label{detE1D}
    	\cE(\lambda) = \det N_\Theta(\lambda)  \det \widehat M_{\widehat\Theta}(\lambda)
\end{equation}
for all $\lambda \in \rho(\cL^\Theta) \cap \rho(\widehat \cL^D)$.

\begin{theorem}
\label{thm:fdEvans}
For any continuous matrix-valued potential $Q$ and matrices $\Theta_\pm \in M_n(\bbC)$ we have
\begin{equation}
\label{detN1D}
	\det N_\Theta(\lambda) = \frac{\mathsf E_{_D}(\lambda) }{\mathsf E_{\Theta}(\lambda)}
	\ \text{ for all } \lambda \in \rho(\cL^\Theta), \qquad
	\det M_\Theta(\lambda) = \frac{\mathsf E_{\Theta}(\lambda)}{\mathsf E_{_D}(\lambda) }
	\ \text{ for all } \lambda \in \rho(\cL^D).
\end{equation}
Given an auxiliary potential $\widehat Q$ and boundary matrices $\widehat\Theta_\pm$, we thus have
\begin{equation}
	\cE(\lambda) 
	= \frac{\mathsf E_{_D}(\lambda) \mathsf {\widehat E}_{\widehat\Theta}(\lambda)}{\mathsf E_{\Theta}(\lambda) \mathsf{\widehat E}_{_D}(\lambda)}
\end{equation}
for all $\lambda \in \rho(\cL^\Theta) \cap \rho(\widehat\cL^D)$.
\end{theorem}


Before proving the theorem, we recall the definitions of $M_\Theta$ and $N_\Theta$. For a differentiable function $u \colon [-1,1] \to \bbC^n$ we define the Dirichlet and Neumann traces
\[
	\gaD u= \begin{bmatrix} u(1)\\u(-1)\end{bmatrix}\in\bbC^{2n},\quad
	\gaN u= \begin{bmatrix} u'(1)\\-u'(-1)\end{bmatrix}\in\bbC^{2n}.
\]
Assuming that $\lambda \in \rho(\cL^D) $ and $f\in\mathbb{C}^{2n}$, we let $u$ be the solution of the boundary value problem \begin{equation}
\label{DBVP}
-u''+Qu=\lambda u, \,\,\,\gamma_{{}_D}u=f,
\end{equation}
and then define the Dirichlet-to-Neumann map $M(\lambda)$ by $M(\lambda)f=\gaN u$. The Dirichlet-to-Robin map $M_\Theta(\lambda)$ is defined by $M_\Theta(\lambda)f = \gaN u + \Theta \gaD u$, so $M_\Theta(\lambda) = M(\lambda) + \Theta$.
For $\lambda \in \rho(\cL^\Theta)$ the Robin-to-Dirichlet map $N_\Theta(\lambda)$ is defined similarly, and satisfies $N_\Theta(\lambda) = M_\Theta(\lambda)^{-1}$ for all $\lambda \in \rho(\cL^\Theta) \cap \rho(\cL^D)$.

Next, we let
\begin{equation}
	\cG(\lambda) := \big\{ (\gaD u, \gaN u)\in\mathbb{C}^{4n}: -u''+Qu=\lambda u \big\}
\end{equation}
denote the set of traces of solutions to the eigenvalue equation, with no boundary conditions imposed. Since the set $\{u : -u''+Qu=\lambda u \}$ is spanned by the columns of the matrices $Y_-(\cdot,\lambda)$ and $V_-(\cdot,\lambda)$, we find that $\cG(\lambda)$ is spanned by the columns of the  \emph{frame matrix}
\begin{equation}
\label{frame}
	\begin{bmatrix} X \\ Z \end{bmatrix} 
	:=\begin{bmatrix} \gaD Y_-& \gaD V_- \\ \gaN Y_- & \gaN V_-\end{bmatrix}
	= \begin{bmatrix} Y_-(1,\lambda)& V_-(1,\lambda)\\Y_-(-1,\lambda) & V_-(-1,\lambda)\\Y_-'(1,\lambda) &V_-'(1,\lambda)\\-Y_-'(-1,\lambda)& -V_-'(-1,\lambda) \end{bmatrix}
	=\begin{bmatrix} Y_-(1,\lambda)& V_-(1,\lambda) \\ 0_{n\times n} & I_{n\times n} \\Y_-'(1,\lambda) &V_-'(1,\lambda)\\ -I_{n\times n} & 0_{n\times n} \end{bmatrix},
\end{equation}
which has size $4n \times 2n$.

We now relate the frame matrix to the the Dirichlet-to-Neumann map.

\begin{lemma}\label{1+2}
If $Q$ is a continuous matrix-valued potential, then $\lambda\in\rho(\cL^D)$ if and only if $X$ is invertible, in which case the Dirichlet-to-Neumann map is given by $M(\lambda) = ZX^{-1}$.
\end{lemma}

Since the columns of  $\left[ \begin{smallmatrix} X \\ Z \end{smallmatrix}\right]$ and $\left[ \begin{smallmatrix} I_{2n\times2n} \\ Z X^{-1} \end{smallmatrix}\right]$ span the same subspace, the lemma implies that $\cG(\lambda)$ is the graph of the Dirichlet-to-Neumann map.

\begin{proof}
Note that $\lambda\in\sigma(\cL^D)$ if and only if there is a nontrivial solution to $-u''+Qu=\lambda u$ that satisfies Dirichlet boundary conditions at both $-1$ and $+1$. From the definition of $Y_-$ and $Y_+$, this happens if and only if the columns of the Evans matrix $\mathbf{E}_D(x,\lambda)$ in \eqref{def:EDmatrix} are linearly dependent, i.e. $\det \mathbf{E}_D(x,\lambda) = \det Y_-(1,\lambda)=0$, and it follows from \eqref{frame} that $\det Y_-(1,\lambda) = \det X$.
	
Assuming $\lambda \in \rho(\cL^D) $, we choose $f\in\mathbb{C}^{2n}$ and let $u$ be the solution of \eqref{DBVP}.  Since $\{u : -u''+Qu=\lambda u \}$ is spanned by the columns of $Y_-$ and $V_-$, there exists a vector $\mathbf{c}\in\mathbb{C}^{2n}$ such that $u(x)=[Y_-(x,\lambda)\ \, V_-(x,\lambda)] \mathbf{c}$ for all $x\in[-1,1]$. Then $\gaD u=X\mathbf{c}$ and $\gaN u=Z\mathbf{c}$ by the definition of the frame. Since $X$ is invertible, this implies $\mathbf{c}=X^{-1}f$ and hence $M(\lambda)f=\gaN u=Z\mathbf{c}=ZX^{-1}f$.
\end{proof}

We are now ready to prove the main result.

\begin{proof}[Proof of Theorem~\ref{thm:fdEvans}]

Using Lemma~\ref{1+2} we obtain $M_\Theta = Z X^{-1} + \Theta = (Z+\Theta X) X^{-1}$. To prove \eqref{detN1D}, we will show that $\det X = \mathsf E_{_D}(\lambda)$ and $\det (Z + \Theta X) = \mathsf E_{\Theta}(\lambda)$. The first claim follows immediately from \eqref{ED} and \eqref{frame}, since
\[
	\mathsf E_{_D}(\lambda) = \det Y_-(1,\lambda) = \det X.
\]
To prove the second claim, we notice that 
\begin{align}
\begin{split}\label{ER1}
	\det(Z+\Theta X)&=\det\begin{bmatrix}  Y_-'(1,\lambda)+\Theta_+ Y_-(1,\lambda) & V_-'(1,\lambda)+\Theta_+ V_-(1,\lambda)\\ -I_{n\times n} &\Theta_- \end{bmatrix}\\
	&=\det\big((Y_-'(1,\lambda)+\Theta_+ Y_-(1,\lambda))\Theta_-+V_-'(1,\lambda)+\Theta_+ V_-(1,\lambda)\big),
\end{split}
\end{align}
where the second equality follows from the fact that if $AC=CA$, then
\begin{equation*}
	\det\begin{bmatrix}A&B\\
		C&D\end{bmatrix}=\det(AD-BC).
\end{equation*}
On the other hand, letting $x=1$ in \eqref{ETheta} gives

\begin{align}
\begin{split}\label{ER2}
	\mathsf E_{\Theta}(\lambda)
	&=\det\begin{bmatrix}V_-(1,\lambda)+Y_-(1,\lambda)\Theta_-&
		-I_{n\times n} \\V'_-(1,\lambda)+Y'_-(1,\lambda)\Theta_-&\Theta_+\end{bmatrix} \\
		& =\det\big(\Theta_+(V_-(1,\lambda)+Y_-(1,\lambda)\Theta_-)+V'_-(1,\lambda)+Y'_-(1,\lambda)\Theta_-\big),
\end{split}
\end{align}
where the second equality follows from the fact that if $BD=DB$, then
\begin{equation*}
	\det\begin{bmatrix}A&B\\
		C&D\end{bmatrix}=\det(DA-BC).
\end{equation*}

Comparing \eqref{ER1} and \eqref{ER2} proves that $\det (Z + \Theta X) = \mathsf E_{\Theta}(\lambda)$, and hence $\det M_\Theta(\lambda) = \mathsf E_{\Theta}(\lambda) / \mathsf E_{_D}(\lambda)$ for all $\lambda \in \rho(\cL^D)$. If $\lambda \in \rho(\cL^\Theta) \cap \rho(\cL^D)$ we can invert this to obtain $\det N_\Theta(\lambda) = \mathsf E_{_D}(\lambda) / \mathsf E_{\Theta}(\lambda)$, and by continuity we conclude that this holds for all $\lambda \in \rho(\cL^\Theta)$.
\end{proof}

\subsection{The unit disc}
\label{sec:disc}
We finally consider $L = -\Delta$ on the unit disc. For a reference operator we take $\widehat L = -\Delta + \gamma$ for some $\gamma \in \bbR$, and we choose boundary operators $\Theta = \mu \J$ and $\widehat\Theta = \hat\mu\J$, where $\mu$ and $\hat\mu$ are complex numbers and $\J \colon \Hp \to \Hm$ is the inclusion. With these choices the eigenvalues of $E(\lambda)$ can be computed explicitly, and we obtain the following.

\begin{theorem}\label{thm3.3}
If $\lambda > \max\{\gamma,0\}$ and $\mu \neq \hat \mu$, then $E(\lambda) - I \in \cB_p$ if and only if $p > 1$. 
\end{theorem}

This verifies the claim made in Remark~\ref{rem:betterp} about the optimality of $p>n-1$.

\begin{proof}
From \cite[Section~4]{Daners} we know that the Dirichlet-to-Neumann map $M(\lambda)$ satisfies
\[
	M(\lambda) e^{ik\vartheta} = d_k(\lambda) e^{ik\vartheta}, \ \vartheta\in[0,2\pi),
\]
for all $k \in \bbZ$, where
\begin{equation}
\label{dklambda}
	d_k(\lambda) = \frac{\sqrt\lambda J'_{|k|}(\sqrt\lambda)}{J_{|k|}(\sqrt\lambda)}
\end{equation}
and $J_k$ are the Bessel functions of the first kind.
If $k$ is even, both $J_{|k|}(z)$ and $z J'_{|k|}(z)$ are entire functions, given by convergent power series only having even powers of $z$, and hence can be evaluated at $z = \sqrt\lambda$ for any $\lambda \in \bbC$, independent of the choice of square root. When $k$ is odd, the same is true of $J_{|k|}(z)/z$ and $J'_{|k|}(z)$. Therefore, as in \eqref{series}, we see that \eqref{dklambda} unambiguously defines $d_k(\lambda)$ for every $\lambda \in \bbC$.

%
%

It follows that $\widehat M(\lambda) = M(\lambda-\gamma)$ has eigenvalues $d_k(\lambda-\gamma)$, so the operator $E(\lambda) = N_\Theta(\lambda) \widehat M_{\widehat\Theta}(\lambda)$ has eigenvalues
\begin{equation}
	\frac{d_k(\lambda - \gamma) + \hat\mu}{d_k(\lambda) + \mu},
\end{equation}
and hence $E(\lambda) - I$ has eigenvalues
\[
	\frac{d_k(\lambda - \gamma) - d_k(\lambda) + \hat\mu - \mu}{d_k(\lambda) - \mu},
\]
indexed by $k \in \bbZ$.

To analyze the $\cB_p$ properties, we recall \cite[Lemma~4.2]{Daners}, which implies
\begin{equation}
	\lim_{k \to \infty} \big(d_k(\lambda) - k\big) = \lim_{k \to \infty} \big(d_k(\lambda - \gamma) - k\big) = 0.
\end{equation}
Therefore, there exists a natural number $N$ so that $|d_k(\lambda) - k| < 1$ and $|d_k(\lambda) - d_k(\lambda-\gamma)| < \tfrac12 |\mu - \hat\mu|$ for all $k \geq N$. Using this, we obtain
\begin{equation}
	\frac12 \frac{|\mu - \hat\mu|}{k + (|\mu| + 1)} \leq 
	\left|\frac{d_k(\lambda - \gamma) - d_k(\lambda) + \hat\mu - \mu}{d_k(\lambda) - \mu}\right| \leq \frac32 \frac{|\mu - \hat\mu|}{k - (|\mu| + 1)}
\end{equation}
for $k \geq N$ and the result follows.
\end{proof}

We conclude this section with a discussion of the Schr\"odinger operator $L=-\Delta+q(|x|)$ on the unit disc with continuous radial potential. We will show that the function $E$ from \eqref{eq:Edef} can be represented as an infinite diagonal matrix whose entries are one-dimensional Evans functions for Schr\"odinger operators on a half-line. However, this does not imply that $\cE = \det_p E$ is the product of one-dimensional Evans functions, on account of the regularizing exponential factors in $\det_p$ when $p>1$; see \eqref{defpdet}.

Passing to polar coordinates $(r,\vartheta)\in(0,1]\times[0,2\pi)$, defining $t\ge0$ by $r=e^{-t}$ and separating variables in the eigenvalue equation $Lu=\lambda u$ produces solutions $v_k(t,\lambda)e^{ik\vartheta}$ for $k \in \bbZ$,  where $v_k(\cdot\,, \lambda)$ is a solution to the one-dimensional Schr\"odinger equation
 \begin{equation}\label{odimsch}
 -v''+Q(t,\lambda)v+k^2v=0, \qquad Q(t,\lambda):=e^{-2t}\big(q(e^{-t})-\lambda\big)
 \end{equation}
 that is bounded on $[0,\infty)$.
With a suitable normalization we can assume that $v_k(\cdot,\lambda)$ is the Jost solution to \eqref{odimsch}, that is, the unique solution asymptotic to the plane wave $t\mapsto e^{-|k|t}$ corresponding to $Q\equiv 0$; see, e.g., \cite[Chapter 17]{ChadSab}. 
For instance, if $q\equiv 0$ as in Theorem \ref{thm3.3}, we may choose $v_k(t,\lambda)=(\sqrt{\lambda})^{-|k|}J_{|k|}\big(\sqrt{\lambda}e^{-t}\big)$.

As for the Laplacian, for general potentials we have $\sigma(\cL^D)=\{\lambda: v_k(0,\lambda)=0 \text{ for some } k=0,1,\dots\}$
and for $\lambda\notin\sigma(\cL^D)$ the Dirichlet-to-Neuman operator $M(\lambda)$ maps the function $e^{ik\vartheta}$ to the function $d_k(\lambda)e^{ik\vartheta}$, where we have introduced the notation $d_k(\lambda)=-{v'_k(0,\lambda)}/{v_k(0,\lambda)}$. (When $q\equiv0$ this reduces to \eqref{dklambda}; the minus sign appears here because we are differentiating with respect to $t = - \log r$.) We remark that the one-dimensional operator $d_k(\lambda)$ is the Dirichlet-to-Neuman map for the one-dimensional Schr\"odinger equation \eqref{odimsch}. Defining $\widehat v_k$ analogously, and letting $\Theta=\mu \cJ$ and $\widehat{\Theta}=\widehat{\mu}\cJ$ as before, we arrive at the following description of $E$ on the disc:
\begin{equation}
	E(\lambda)=\mathbf{F} {\rm diag } \left\{ \frac{\widehat d_k(\lambda) + \widehat\mu}{d_k(\lambda) + \mu} \right\}_{k\in\bbZ}\mathbf{F}^{-1},
\end{equation}
where
\begin{align*}
	{d}_k(\lambda)=-\frac{{v}'_k(0,\lambda)}{{v}_k(0,\lambda)}, \qquad
	\widehat{d}_k(\lambda)=-\frac{\widehat{v}'_k(0,\lambda)}{\widehat{v}_k(0,\lambda)}
\end{align*}
and
\[
	\mathbf{F}:(c_k)_{k\in\bbZ}\mapsto\sum_{k\in\bbZ} c_k e^{ik\vartheta}
\]
 is the discrete Fourier transform.

\section{Preliminaries}
\label{sec:prelim}
We now define all of the operators needed for the statement and proof of Theorem~\ref{thm:WEvans}. In particular, we define the Dirichlet and Robin realizations, $\cL^D$ and $\cL^\Theta$, as well as the Dirichlet-to-Robin and Robin-to-Dirichlet maps, $M_\Theta$ and $N_\Theta$. We also define the Robin-to-Robin map $R_{\Theta_1,\Theta_2}$. While this does not appear in the statement of the theorem, it is needed for the proof, which is given in Section~\ref{sec:det}. The constructions in Sections \ref{sec:domains} and \ref{sec:Robin} are standard, and are included for the sake of completeness. In Section~\ref{sec:resTheta} we prove an auxiliary result on the spectrum of $\cL^\Theta$ that will be useful in the proof of Theorem~\ref{thm:WEvans}.

\subsection{Operators and domains}
\label{sec:domains}
Recall the differential expression
\[Lu = - \sum_{j,k=1}^n \p_j (a_{jk} \p_k u) + \sum_{j=1}^n b_j \p_j u - \sum_{j=1}^n\partial_j(d_ju)+ qu,\] defined in \eqref{Ldef}. We define the associated sesquilinear form
\begin{equation}\label{defPhi}
	\Phi(u,v) = \sum_{j,k=1}^n \left< a_{jk} \p_k u, \p_j v \right> + \sum_{j=1}^n \left< b_j \p_j u, v \right> + \sum_{j=1}^n\langle d_ju,\p_jv\rangle+ \left<qu,v\right>
\end{equation}
where $\left<\cdot,\cdot \right>$ denotes the $L^2(\Omega)$ inner product, so that $\Phi(u,v) = \left< Lu,v \right>$ for all $u,v \in C^\infty_c(\Omega)$ compactly supported in $\Omega$.  We will also consider the formal adjoint differential expression
\begin{equation}\label{Ldaddef} 
L^\dag u = - \sum_{j,k=1}^n \p_k (\bar{a}_{jk}\p_j u) + \sum_{j=1}^n \bar{d}_j \p_j u - \sum_{j=1}^n\partial_j(\bar{b}_i)u+ \bar{q}u,
\end{equation}
where bar stands for complex conjugation. The adjoint form $\Phi^*$ is defined by $\Phi^*(u,v)=\overline{\Phi(v,u)}$ and corresponds to $L^\dag$ in the sense that $\Phi^*(u,v) = \left< L^\dag u,v \right>$ for all $u,v \in C^\infty_c(\Omega)$. We adhere to the standard notation $\Phi[u]=\Phi(u,u)$ for the quadratic form associated with $\Phi$.

To describe the domains of the Dirichlet, Neumann and Robin realizations of the operator $L$, we will need the following brief discussion of the Dirichlet and Neumann traces; see, for instance, \cite{McL00}. By the standard trace theorem, the linear mapping $H^1(\Omega) \cap C(\overline\Omega) \ni u \mapsto u\big|_{\pO} \in  C(\partial\Omega)$ can be extended to a bounded linear operator
\begin{equation}\label{e9}
	\gaD \colon H^1(\Omega) \longrightarrow \Hp.
\end{equation} 
To define the conormal derivative operator for $L$, we introduce the function space 
\begin{align}
& D^1_{L}(\Omega):=\{u\in H^1(\Omega): L u\in L^2(\Omega)\}, \label{cc4}
\end{align}
equipped with the natural graph norm of $L$,
\begin{equation}\label{e1}
\|u\|_{L,1}:=\left(\|u\|_{H^1({\Omega})}^2+\|L u\|_{L^2({\Omega})}^2\right)^{1/2}, 
\end{equation}
where $L u$ should be understood in the sense of distributions. We define $D^1_{L^\dag}(\Omega)$ analogously.

\begin{prop} \cite[Theorem 4.4]{McL00}\label{e10}
Assuming Hypothesis \ref{hyp:L}, there exist bounded linear operators $\gaN^{L}\in\cB\big(\cD^1_{L}(\Omega), H^{-1/2}(\partial\Omega)\big)$ and $\gaN^{L^\dag}\in\cB\big(\cD^1_{L^\dag}(\Omega), H^{-1/2}(\partial\Omega)\big)$ such that Green's first identity holds, that is,
\begin{align}
	&\Phi(u,v)=\langle L u, v\rangle_{L^2(\Omega)}+\llangle {\gaN^L}u,\gaD v\rrangle
\quad \text{ for all } u\in\cD^1_{L}(\Omega), v\in H^1{(\Omega)},\label{e14}\\
	&\Phi(u,v)=\langle  u, L^\dag v\rangle_{L^2(\Omega)}+\overline{\llangle \gaN^{L^\dag}v, \gaD u\rrangle}
\quad \text{ for all } u\in H^1{(\Omega)}, v\in\cD^1_{L^\dag}(\Omega),
\end{align}
and hence Green's second identity
\begin{align}\label{e13}
\langle L u, v\rangle_{L^2(\Omega)}- \langle u,L^\dag v\rangle_{L^2(\Omega)}=\overline{\llangle{\gaN^{L^\dag}} v ,\gaD u \rrangle}-\llangle{\gaN^L} u ,\gaD v \rrangle \quad
\end{align}
holds for all $u\in\cD^1_{L}(\Omega)$ and $v\in\cD^1_{L^\dag}(\Omega)$.
\end{prop}

We will frequently use the fact that 
\begin{equation}\label{UCP}
	\big\{u\in H^1(\Omega): L u=0, \gaD u = 0 \text{ and } \gaN^L u = 0   \big\}=\{0\},
\end{equation}
which follows from the unique continuation principle, cf. \cite[Theorem 3.2.2]{I06}.

When $u$ is sufficiently smooth, say in $H^2(\Omega)$, and the coefficients $b_j$ and $d_j$ are Lipschitz, the conormal derivatives can be computed by the usual formulas
\begin{align}
	{\gaN^{L}}u &=\sum_{j,k=1}^{n} a^{jk}\nu_j\gaD(\partial_k u)+\sum_{j=1}^n \nu_j\gaD (d_ju), \label{e2}\\
	{\gaN^{L^\dag}}u &=\sum_{j,k=1}^{n} \bar{a}^{jk}\nu_j\gaD(\partial_k u)+\sum_{j=1}^n \nu_j\gaD (\bar{b}_ju),
\end{align}
with $\nu=(\nu_1,\ldots, \nu_n)$ denoting the outward unit normal to $\partial\Omega$. We will also need the following identity for the difference of Neumann traces when the principal parts of $L$ and $\widehat L$ coincide.

\begin{lemma}
\label{lem:gaNdiff}
Let $L$ and $\widehat L$ both satisfy Hypothesis \ref{hyp:L}, with $a_{jk} = \hat a_{jk}$ for all $j,k$, and $d_j - \hat d_j$ Lipschitz. If $u \in D^1_L(\Omega)$, then $\widehat L u \in L^2(\Omega)$, and
\begin{equation}
\label{gaNdiff}
	\gaN^L u - \gaN^{\widehat L} u = \sum_{j=1}^k \nu_j \gaD \big((d_j - \hat d_j) u\big) \in L^2(\pO).
\end{equation}
\end{lemma}

If $u \in H^2(\Omega)$, this formula would follow immediately from subtracting \eqref{e2} and the corresponding expression for $\gaN^{\widehat L} u$. The advantage of Lemma~\ref{lem:gaNdiff} is that it only requires $u \in D^1_L(\Omega)$. For a Lipschitz domain we can only guarantee that the right-hand side of \eqref{gaNdiff} is in $L^2(\pO)$, even though each $\gaD \big((d_j - \hat d_j) u\big)$ is in $\Hp$, since we only know that the $\nu_j$ are bounded.

\begin{proof}
For any $v \in H^1_0(\Omega)$ we have
\begin{align*}
	\big<L u,v\big> - \big<\widehat L u,v\big> &= \Phi(u,v) - \widehat\Phi(u,v) \\
	&= \sum_{j=1}^n \left< (b_j - \hat b_j) \p_j u, v \right> + \sum_{j=1}^n\langle (d_j - \hat d_j) u,\p_jv\rangle+ \left<(q - \hat q)u,v\right>
\end{align*}
and hence
\begin{equation}
\label{Ldiff}
	\big<L u,v\big> - \big<\widehat L u,v\big> = \sum_{j=1}^n \left< (b_j - \hat b_j) \p_j u, v \right> - \sum_{j=1}^n\langle \p_j ((d_j - \hat d_j) u),v\rangle+ \left<(q - \hat q)u,v\right>,
\end{equation}
where in the last line we have integrated by parts, using the fact that $(d_j - \hat d_j) u \in H^1(\Omega)$. By density, this last equality holds for all $v \in L^2(\Omega)$, and in particular for $v \in H^1(\Omega)$. For $v \in H^1(\Omega)$ we thus compute $\Phi$ and $\widehat\Phi$ using \eqref{defPhi}, and subtract \eqref{Ldiff} to obtain
\begin{align*}
	\Phi(u,v) - \widehat\Phi(u,v) - \big(\big<L u,v\big> - \big<\widehat L u,v\big>\big) &= \sum_{j=1}^n\langle (d_j - \hat d_j) u,\p_jv\rangle + \langle \p_j (d_j - \hat d_j) u,v\rangle \\
	&=  \sum_{j=1}^k \llangle \nu_j \gaD \big((d_j - \hat d_j) u\big), \gaD v \rrangle,
\end{align*}
using the divergence theorem in the second line. On the other hand, \eqref{e14} says that
\begin{align*}
	\Phi(u,v) - \widehat\Phi(u,v) - \big(\big<L u,v\big> - \big<\widehat L u,v\big>\big)
	=  \llangle {\gaN^L}u - {\gaN^{\widehat L}}u,\gaD v\rrangle
\end{align*}
for all $v \in H^1(\Omega)$. Comparing these formulas completes the proof.
\end{proof}

We now define the Dirichlet, Robin and Neumann realizations of $L$ using the theory of sectorial forms. We first recall some standard facts, which can be found in \cite[Chapter~11]{Schm12}. Let $V\subset H^1(\Omega)$ be a closed subspace that contains $H^1_0(\Omega)$. The form $\Phi$ is said to be \emph{elliptic} on $V$ if it is bounded and satisfies the abstract G\r{a}rding inequality. That is, there is a constant $C>0$ such that $|\Phi(u,v)|\le C\|u\|_V\|v\|_V$ for all $u,v\in V$, and constants $\gamma > 0$ and $c\in\bbR$ such that 
	\begin{equation}\label{Gin}
		\Re \Phi[u] \geq \gamma \|u\|_{V}^2 + c \|u\|_{L^2(\Omega)}^2
	\end{equation}
	holds for all $u \in V$.
The form $\Phi$ is elliptic if and only if it is closed and sectorial, in which case it 
generates a sectorial operator $\cL^\Phi$ on $L^2(\Omega)$ satisfying $\Phi(u,v)=\langle\cL^\Phi u,v\rangle$ for all $u,v\in V$. The operator $\cL^\Phi$ is given by 
\begin{align}
\begin{split}\label{domLV}
\dom(\cL^\Phi)&=\big\{u\in V: \text{ there exists $w_u\in L^2(\Omega)$ such that $\Phi(u,v)=\langle w_u,v\rangle$ for all $v\in V$}\big\} \\
	\cL^\Phi u &= w_u.
\end{split}
\end{align}

We now construct the Dirichlet realization of $L$.

 \begin{prop}\label{domD}
Assuming Hypothesis \ref{hyp:L}, the form $\Phi$ defined in \eqref{defPhi} is elliptic on $V= H^1_0(\Omega)$, and the corresponding sectorial operator $\cL^D$ satisfies $\cL^D u=Lu$ for all $u$ in its domain
\begin{equation}
\label{dom:LD}
	\dom(\cL^D) = \big\{u \in H^1_0(\Omega) : Lu\in L^2(\Omega) \big\}.
\end{equation}
Moreover, $\cL^D$ has compact resolvent, so its spectrum consists entirely of isolated eigenvalues of finite algebraic multiplicity, and $\cL^D - \lambda$ is Fredholm of index zero for any $\lambda \in \bbC$.
\end{prop}

\begin{proof}
It is well known that $\Phi$ is elliptic under Hypothesis \ref{hyp:L}, see, e.g., \cite[Proposition~11.10]{Schm12}. If $u \in H^1_0(\Omega)$ and $Lu \in L^2(\Omega)$, it follows from the first Green formula \eqref{e14} that $\Phi(u,v)  = \left<Lu,v\right>$ for all $v \in H^1_0(\Omega)$, and so \eqref{domLV}  implies that $u \in \dom(\cL^D)$. On the other hand, if $u \in \dom(\cL^D)$, then \cite[Lemma~11.11]{Schm12} implies $Lu \in L^2(\Omega)$, thus proving \eqref{dom:LD}.

Since $H^1(\Omega)$ is compactly embedded in $L^2(\Omega)$,  \cite[Theorem~11.8(iii)]{Schm12} implies that $\cL^D$ has compact resolvent, and the remaining statements follow from \cite[Theorem~IX.3.1]{EE87}.
\end{proof}

Next, we fix an operator $\Theta \colon \Hp \to \Hm$ and construct the Robin realization, $\cL^\Theta$, of $L$, corresponding to the boundary condition $\gaN^L u + \Theta\gaD u = 0$. For $\Theta=0$ this is the Neumann condition $\gaN^L u = 0$, and the corresponding operator will be denoted $\cL^N$. To construct $\cL^\Theta$, we let $V=H^1(\Omega)$ and, using $\Phi$ from \eqref{defPhi}, define the form
\begin{equation}\label{dfnPhimu}
	\Phi_\Theta(u,v):=\Phi(u,v) + \llangle \Theta\gaD u,\gaD v\rrangle, \qquad \dom(\Phi_\Theta):=H^1(\Omega).
\end{equation}
To define $\cL^\Theta$, we need to know that $\Phi_\Theta$ is elliptic. This requires an extra assumption on $\Theta$.

\begin{hyp}\label{hypT}
Assume, in addition to Hypothesis~\ref{hyp:L}, that $\Theta \colon \Hp \to \Hm$ is a bounded operator and the form $\Phi_\Theta$ in \eqref{dfnPhimu} is elliptic.
\end{hyp}

The following lemma gives some easily verified conditions under which this hypothesis is satisfied. 
\begin{lemma}
Assume Hypothesis~\ref{hyp:L} and let $\Theta \colon \Hp \to \Hm$ be a bounded operator. The form $\Phi_\Theta$ in \eqref{dfnPhimu} is elliptic if either
	\begin{enumerate}
		\item $\Re \llangle \Theta \gaD u, \gaD u \rrangle \geq 0$ for all $u\in H^1(\Omega)$, or
		\item $\Theta$ is compact.
	\end{enumerate}
\end{lemma}

\begin{proof}
Boundedness of $\Phi_\Theta$ follows from the boundedness of the form $\Phi$, the operator $\Theta$, and the trace map $\gaD$, since
\[
	\big| \llangle \Theta \gaD u, \gaD v \rrangle \big| \leq C \|\Theta \gaD u\|_{\Hm} \|\gaD v\|_{\Hp}
	\leq C  \| u\|_{H^1(\Omega)} \| v\|_{H^1(\Omega)}.
\]
To establish the G\r{a}rding inequality \eqref{Gin} for $\Phi_\Theta$, it is enough to show that for any $\epsilon>0$, there exists a constant $C = C(\epsilon)$ such that
\begin{equation}
\label{thetabound}
	-\Re \llangle \Theta \gaD u, \gaD u \rrangle \leq \epsilon\| u\|^2_{H^1(\Omega)} + C(\epsilon)\|u\|^2_{L^2(\Omega)}
\end{equation}
for all $u\in H^1(\Omega)$. To see that this suffices,
we use \eqref{thetabound} and the G\r{a}rding inequality for $\Phi$ to compute
\begin{align*}
	\Re\Phi_\Theta[u] &= \Re\Phi[u] + \Re \llangle \Theta \gaD u, \gaD u \rrangle \\
	&\geq \gamma \|u\|_{H^1(\Omega)}^2 + c \|u\|_{L^2(\Omega)}^2 
	-\epsilon\| u\|^2_{H^1(\Omega)} - C(\epsilon)\|u\|^2_{L^2(\Omega)}
\end{align*}
and then choose $\epsilon < \gamma$.

The condition \eqref{thetabound} is trivially satisfied if $\Re \llangle \Theta \gaD u, \gaD u \rrangle \geq 0$ for all $u\in H^1(\Omega)$, so to complete the proof we just need to show that it also holds if $\Theta$ is compact.
Suppose there is an $\epsilon>0$ for which no such $C(\epsilon)$ exists. This means there is a sequence $\{u_j\}$ in $H^1(\Omega)$ with
\begin{equation}
\label{contradict}
	\big|\llangle \Theta \gaD u_j, \gaD u_j \rrangle\big| \geq \epsilon\| u_j \|^2_{H^1(\Omega)} + j\|u_j \|^2_{L^2(\Omega)}
\end{equation}
for each $j$, and we can assume that $\|u_j\|_{H^1(\Omega)} = 1$. Then there is a subsequence (still denoted $\{u_j\}$) and a function $u \in H^1(\Omega)$ so that $u_j \rightharpoonup u$ in $H^1(\Omega)$, and hence
\[
	u_j \to u \text{ in } L^2(\Omega), \qquad \gaD u_j \rightharpoonup \gaD u \text{ in } \Hp, 
	\qquad \Theta\gaD u_j \to \Theta\gaD u \text{ in } \Hm.
\]
It follows that $\llangle \Theta \gaD u_j, \gaD u_j \rrangle \to \llangle \Theta \gaD u, \gaD u \rrangle$. From \eqref{contradict} we get $|\llangle \Theta \gaD u, \gaD u \rrangle| \geq \epsilon$, which implies $u$ is nonzero. On the other hand, \eqref{contradict} also gives
\[
	\|u_j \|^2_{L^2(\Omega)} \leq j^{-1}\left(\big|\llangle \Theta \gaD u_j, \gaD u_j \rrangle\big| - \epsilon\right) \to 0
\]
and hence $u = 0$, a contradiction.
\end{proof}

We now construct the operator $\cL^\Theta$.

\begin{prop}\label{prop:domLmu}
Assume Hypothesis~\ref{hypT}, so the form $\Phi_\Theta$ defined in \eqref{dfnPhimu} is elliptic. The corresponding sectorial operator $\cL^\Theta$ satisfies $\cL^\Theta u=Lu$ for all $u$ in its domain
\begin{equation}\label{domLmutr}
	\dom(\cL^\Theta)=\big\{ u \in H^1(\Omega) : Lu \in L^2(\Omega) \text{ and } \gaN^Lu + \Theta\gaD u = 0\big\}.
\end{equation}
Moreover,  $\cL^\Theta$ has compact resolvent, so its spectrum consists entirely of isolated eigenvalues of finite algebraic multiplicity, and $\cL^\Theta - \lambda$ is Fredholm of index zero for any $\lambda \in \bbC$.
\end{prop}

\begin{proof}
To describe the domain of $\cL^\Theta$, we denote by $\cD$ the right-hand side of \eqref{domLmutr}. If $u\in\cD$ then $Lu\in L^2(\Omega)$ satisfies, for any $v\in H^1(\Omega)$,
\begin{align*}
\langle Lu, v\rangle&=\Phi(u,v)-\llangle\gaN^L u,\gaD v\rrangle \qquad\text{(by the first Green formula \eqref{e14})}\\
&=\Phi(u,v) + \llangle\Theta \gaD u,\gaD v\rrangle
\qquad\text{(because $u\in\cD$)}\\
&=\Phi_\Theta(u,v).
\end{align*}
Thus, $u\in\dom(\cL^\Theta)$ according to \eqref{domLV} with $\Phi$ replaced by $\Phi_\Theta$. This proves $\cD\subset\dom(\cL^\Theta)$.

To prove the inverse inclusion, we take $u\in\dom(\cL^\Theta)$ and notice that $\cL^\Theta u=Lu\in L^2(\Omega)$, by \cite[Lemma~11.11]{Schm12}. For all $v\in H^1(\Omega)$ we have the equalities
\begin{align*}
\langle Lu, v\rangle&=\Phi(u,v)-\llangle\gaN^L u,\gaD v\rrangle &\text{(by the first Green formula \eqref{e14})}\\
\langle Lu, v\rangle&=\Phi_\Theta(u,v)=\Phi(u,v) + \llangle\Theta\gaD u,\gaD v\rrangle &\text{(by \eqref{domLV} applied to $\Phi_\Theta$)}.
\end{align*}
Comparing these yields $\gaN^Lu + \Theta\gaD u = 0$, completing the proof of \eqref{domLmutr}.
The final assertions are proved as in Proposition~\ref{domD}.
\end{proof}

By the general relation between the adjoint form and adjoint operator, see \cite[Theorem 11.8]{Schm12}, we have $(\cL^\Theta)^*=\cL^{\Phi_\Theta^*}$, where $\Phi_\Theta^*(u,v)=\Phi^*(u,v) + \llangle \Theta^*\gaD u,\gaD v\rrangle$.
Analogously to Proposition~\ref{prop:domLmu}, one can then prove that
\begin{equation}\label{domadj}
	\dom \big((\cL^\Theta)^* \big)= 
\big\{u\in H^1(\Omega) : L^\dag u \in L^2(\Omega) \text{ and } \gaN^{L^\dag}u + \Theta^* \gaD u = 0\big\}.
\end{equation}
Finally, we observe that the unique continuation result \eqref{UCP} is equivalent to
\begin{equation}\label{UCP2}
	\ker \cL^\Theta \cap \ker \cL^D =\{0\}
\end{equation}
for any $\Theta$ satisfying Hypothesis~\ref{hypT}.

\subsection{The Robin-to-Dirichlet and Dirichlet-to-Robin maps}\label{sec:Robin}
We now define the Robin-to-Dirichlet and Dirichlet-to-Robin operators, $N_\Theta(\lambda)$ and $M_\Theta(\lambda)$, associated to $L$. These are the main ingredients in the definition of the operator family $E(\lambda)$ in \eqref{eq:Edef}, and the properties established in this section will be used throughout.

First, we recall some results on the generalized Robin boundary value problem. Fix an operator $\Theta \colon \Hp \to \Hm$ and consider the boundary value problem
\begin{equation}\label{inhRbvp}
	Lu-\lambda u = f \text{ in } \Omega,  \quad \gaN^Lu + \Theta\gaD u = g \text{ on } \pO,
\end{equation}
with $f\in H^1(\Omega)^*$ and $g\in H^{-1/2}(\pO)$, recalling that $H^1(\Omega)^*$ is a proper subspace of $H^{-1}(\Omega) = H^1_0(\Omega)^*$. Using \eqref{domLmutr}, we see that the homogeneous problem (with $f=0$ and $g=0$) admits a nontrivial solution if and only if $\lambda \in \sigma(\cL^\Theta)$.

\begin{prop}\cite[Theorem~4.11]{McL00}\label{TH4.11}
Assume Hypothesis~\ref{hypT}, and suppose that $\lambda\in \rho(\cL^\Theta)$.
Then, for any $f\in H^1(\Omega)^*$ and $g\in H^{-1/2}(\pO)$, the inhomogeneous problem \eqref{inhRbvp} has a unique solution $u\in H^1(\Omega)$, which satisfies the estimate
\begin{equation}\label{estinhpr}
	\|u\|_{H^1(\Omega)}\le c \left(\|f\|_{H^{-1}(\Omega)} + \|g\|_{H^{-1/2}(\pO)} \right).
\end{equation}
\end{prop}

This is a slight generalization of the result appearing in \cite{McL00}, but the proof is identical, since it only relies on the ellipticity of $\Phi_\Theta$. This will be sufficient to construct the operator $N_\Theta(\lambda)$. In Section \ref{sec:Bp}, when we establish $\cB_p$ mapping properties of $E(\lambda)$, we will need the following refinement.

\begin{prop}\cite[Theorem~4.24]{McL00}\label{TH4.24}
Assume, in addition to the hypotheses of Proposition~\eqref{TH4.11}, that Hypothesis~\ref{hyp:Bp} is satisfied. If $f \in L^2(\Omega)$ and $g \in L^2(\pO)$, then the unique solution $u$ to \eqref{inhRbvp} has $\gaD u \in H^1(\pO)$, with the estimate
\begin{equation}\label{MC:424}
	\|\gaD u\|_{H^1(\pO)}\le c\left( \|f\|_{L^2(\Omega)} + \| g \|_{L^2(\pO)} \right).
\end{equation}
\end{prop}

\begin{proof}
Since $Lu - \lambda u = f \in L^2(\Omega)$ and $\gaN^L u = g - \Theta \gaD u \in L^2(\pO)$, we can apply \cite[Theorem~4.24]{McL00} to obtain $\gaD u \in H^1(\pO)$, with
\[
	\|\gaD u\|_{H^1(\pO)} \le c\left( \|u\|_{H^1(\Omega)} +  \|f\|_{L^2(\Omega)} + \| \gaN^L u\|_{L^2(\pO)} \right).
\]
Hypothesis~\ref{hyp:Bp}  implies
$
	\| \Theta \gaD u\|_{L^2(\pO)} \leq c \| \gaD u\|_{\Hp} \leq c \|u\|_{H^1(\Omega)}
$
and hence
\[
	\| \gaN^L u\|_{L^2(\pO)} = \| g - \Theta \gaD u\|_{L^2(\pO)} \leq \| g \|_{L^2(\pO)} + c \|u\|_{H^1(\Omega)},
\]
so we get
\[
	\|\gaD u\|_{H^1(\pO)} \le c\left( \|u\|_{H^1(\Omega)} +  \|f\|_{L^2(\Omega)} + \| g \|_{L^2(\pO)} \right).
\]
Finally, Proposition \ref{TH4.11} implies $\|u\|_{H^1(\Omega)} \le c \left(\|f\|_{L^2(\Omega)} + \|g\|_{L^2(\pO)} \right)$, completing the proof.
\end{proof}

To define the Robin-to-Dirichlet map $N_\Theta(\lambda)$ for $\lambda$ outside of the spectrum of $\cL^\Theta$, we fix $g\in \Hm$, let $u\in\cD_L^1(\Omega)$ denote the unique solution to 
\begin{equation}\label{sRbvp}
	Lu=\lambda u, \quad \gaN^Lu + \Theta\gaD u=g,
\end{equation}
as guaranteed by Proposition \ref{TH4.11}, and then define
\begin{equation}\label{dfnRtoD}
	N_\Theta(\lambda)g:=\gaD u.
\end{equation} 
If $\Theta=0$, then $N_0(\lambda)$ is the 
Neumann-to-Dirichlet map, which we abbreviate as $N(\lambda)$. We will need the following standard properties of the Robin-to-Dirichlet operator, cf. \cite{Behrndt_2020,GM08} and the literature cited therein.

\begin{lemma}\label{lemRtoD} 
Assume Hypothesis~\ref{hypT}.
\begin{enumerate}
	\item[(i)] If $\lambda \in \rho(\cL^\Theta)$, then $N_\Theta(\lambda)\in\cB \big(\Hm, \Hp \big)$.
	\item[(ii)] The map $\lambda \mapsto N_\Theta(\lambda)$ is analytic in $\rho(\cL^\Theta)$.
\end{enumerate}
\end{lemma}

\begin{proof} To prove (i), we use the estimate from Proposition \ref{TH4.11} with $f=0$, to conclude that the solution $u$ of the boundary value problem  \eqref{sRbvp} satisfies $\|u\|_{H^1(\Omega)}\le c\|g\|_{H^{-1/2}(\p\Omega)}$. We then compute
\[
	\|N_\Theta(\lambda)g\|_{\Hp} = \|\gaD u\|_{H^{1/2}(\Omega)}\le 
	c \|u\|_{H^1(\Omega)}
	\leq c \|g\|_{\Hm}
\]
and the result follows.

To prove (ii), fix $\lambda_0 \in \rho(\cL^\Theta)$ and $g \in \Hm$. By Proposition \ref{TH4.11} there exists a unique solution $v \in H^1(\Omega)$ to the boundary value problem
\[
	Lv = \lambda_0 v, \quad \gaN^L v + \Theta\gaD v = g.
\]
Now for any $\lambda \in \rho(\cL^\Theta)$ we have $(\lambda - L)v = (\lambda - \lambda_0) v \in L^2(\Omega)$, and so $w(\lambda) := (\cL^\Theta - \lambda)^{-1} \big( (\lambda - L)v \big) \in \dom(\cL^\Theta)$ satisfies
\[
	(L - \lambda) w = (\lambda - L)v , \quad \gaN^L w + \Theta\gaD w = 0.
\]
Since $\cL^\Theta - \lambda$ is bounded from $\dom(\cL^\Theta)$ (equipped with the $D^1_L(\Omega)$ norm) into $L^2(\Omega)$, the resolvent $(\cL^\Theta - \lambda)^{-1} \in \cB\big( L^2(\Omega), D^1_L(\Omega)\big)$ is analytic in $\lambda$. Therefore, $\lambda \mapsto w(\lambda)$ defines an analytic function $\rho(\cL^\Theta) \to H^1(\Omega)$.
It follows that $u(\lambda) := v + w(\lambda)$ solves the boundary value problem
\[
	Lu = \lambda u, \quad \gaN^L u + \Theta \gaD u = g,
\]
and so
\[
	N_\Theta(\lambda) g = \gaD u(\lambda) \in \Hp
\]
depends analytically on $\lambda$. This is the case for each $g \in \Hm$, so $\lambda \mapsto N_\Theta(\lambda)$ is strongly analytic and hence analytic.
\end{proof}

The Dirichlet-to-Robin map $M_\Theta(\lambda)$ is defined similarly, for any 
 $\Theta\in\cB(\Hp, \Hm)$. Assuming $\lambda \in \rho(\cL^D)$, the boundary value problem
\[
	Lu = \lambda u, \quad \gaD u = g
\]
has a unique solution $u \in D^1_L(\Omega)$ for each $g \in \Hp$, so we define
\begin{equation}
	M_\Theta(\lambda) g := \gaN^L u + \Theta \gaD u.
\end{equation}
For $\Theta=0$, $M_0(\lambda)$ is the 
Dirichlet-to-Neumann map, and we abbreviate it as $M(\lambda)$.

The following lemma (and its proof) is analogous to Lemma \ref{lemRtoD}.

\begin{lemma}\label{lemDtoR} 
Assume Hypotheses \ref{hyp:L} and let $\Theta \in \cB\big(\Hp,\Hm\big)$.
\begin{enumerate}
	\item[(i)] If $\lambda \in \rho(\cL^D)$, then $M_\Theta(\lambda)\in\cB \big(\Hp, \Hm \big)$.
	\item[(ii)] The map $\lambda \mapsto M_\Theta(\lambda)$ is analytic in $\rho(\cL^D)$.
\end{enumerate}
\end{lemma}

We remark that $N_\Theta(\lambda)M_\Theta(\lambda)=I_{H^{1/2}(\pO)}$ and 
$M_\Theta(\lambda)N_\Theta(\lambda)=I_{H^{-1/2}(\pO)}$ for $\lambda\in\rho(\cL^D)\cap\rho(\cL^\Theta)$.

\begin{rem}\label{signconv} In the definitions of the Neumann-to-Dirichlet and Dirichlet-to-Neumann maps we choose to use the sign convention adapted in \cite{Behrndt_2020} rather than in \cite{GM08}. That is, we write $M_0(\lambda)g=\gaN^Lu$ rather than  $M_0(\lambda)g=-\gaN^Lu$. As a result, for  symmetric $L$ the function $\lambda\mapsto N_\Theta(\lambda)$ is a Nevanlinna function (as in \cite{Behrndt_2020}), instead of $\lambda\mapsto M_\Theta(\lambda)$ (as in \cite{GM08}), cf.\ Proposition \ref{prop8.8} below.
\end{rem}

Finally, given $\Theta_1$ satisfying Hypothesis~\ref{hypT} and any bounded $\Theta_2$, we introduce the Robin-to-Robin map $R_{\Theta_1,\Theta_2}$ (which we abbreviate as $R_{1,2}$) as follows. If $\lambda \in \rho(\cL^{\Theta_1})$, then for each $g \in \Hm$ the Robin boundary value problem 
\[
	Lu=\lambda u, \quad \gaN^Lu + \Theta_1 \gaD u=g,
\]
has a unique solution $u\in\cD_L^1(\Omega)$, and so we define
\begin{equation}
	R_{1,2}(\lambda) g := \gaN^L u +\Theta_2 \gaD u.
\end{equation}
It follows immediately that $R_{\Theta_1,\Theta_1} = I_{\Hm}$. Moreover, if $\Theta_2$ also satisfies Hypotheses~\ref{hypT} and $\lambda \in \rho(\cL^{\Theta_1}) \cap \rho(\cL^{\Theta_2})$, then $R_{1,2}(\lambda) = R_{2,1}(\lambda)^{-1}$. The following properties are easily verified, as above.

\begin{lemma}\label{lemRtoR} 
Assume that $\Theta_1$ satisfies Hypotheses~\ref{hypT} and $\Theta_2 \in \cB\big(\Hp,\Hm\big)$.
\begin{enumerate}
	\item[(i)] If $\lambda \in \rho(\cL^{\Theta_1})$, then $R_{1,2}(\lambda)\in\cB \big(\Hm \big)$.
	\item[(ii)] The map $\lambda \mapsto R_{1,2}(\lambda)$ is analytic in $\rho(\cL^{\Theta_1})$.
\end{enumerate}
\end{lemma}

The introduction of $R_{1,2}$ is motivated by the following factorization property. For convenience we write $N_j(\lambda) = N_{\Theta_j}(\lambda)$ for $j=1,2$, and similarly for $M_j(\lambda)$.

\begin{prop}\label{propRtoR} 
Assume that $\Theta_1$ satisfies Hypotheses~\ref{hypT} and $\Theta_2 \in \cB\big(\Hp,\Hm\big)$.
\begin{enumerate}
	\item[(i)]  If $\lambda \in\rho(\cL^{\Theta_1})$, then $R_{1,2}(\lambda)= I_{\Hm}+(\Theta_2 - \Theta_1) N_1(\lambda)$.
	\item[(ii)] If $\lambda \in \rho(\cL^D) \cap \rho(\cL^{\Theta_1})$, then $M_2(\lambda) = R_{1,2}(\lambda) M_1(\lambda) $.
	\end{enumerate}
	If $\Theta_2$ also satisfies Hypotheses~\ref{hypT} and $\lambda \in \rho(\cL^{\Theta_1}) \cap \rho(\cL^{\Theta_2})$, then 
	\begin{enumerate}
	\item[(iii)] $N_1(\lambda) = N_2(\lambda) R_{1,2}(\lambda)$,
	\item[(iv)] $N_1(\lambda)-N_2(\lambda)=N_2(\lambda)(\Theta_2-\Theta_1)N_1(\lambda)$,
	\item[(v)] $R_{\Theta_1, - \Theta_2}(\lambda)R_{\Theta_2,\Theta_1}(\lambda)=I_{\Hm}-2\Theta_2N_2(\lambda)$.
\end{enumerate}
\end{prop}

This is an important ingredient in the proof of Theorem \ref{thm:WEvans}, since it will allow us to separate the zeros and poles of $\widehat M_\Theta$ and $N_\Theta$, and hence deal with the case that both terms in \eqref{eq:Edef} are singular. The last assertion is used in the proof of Theorem \ref{thm:cont}.

\begin{proof}
We prove (iii);  assertions (i) and (ii) follow from similar computations. Fix $g \in \Hm$, and let $u$ denote the unique solution to
\[
	Lu=\lambda u, \quad \gaN^Lu + \Theta_1 \gaD u=g,
\]
so that $R_{1,2}(\lambda) g = \gaN^L u +\Theta_2 \gaD u$ and $N_1(\lambda) g = \gaD u$. Since $u$ also solves the boundary value problem
\[
	Lu=\lambda u, \quad \gaN^Lu + \Theta_2\gaD u = R_{1,2}(\lambda) g,
\]
we have
\[
	N_2(\lambda) \big(R_{1,2}(\lambda) g\big) = \gaD u = N_1(\lambda) g
\]
for arbitrary $g$, and so $N_2(\lambda) R_{1,2}(\lambda) = N_1(\lambda)$.

To prove (iv), we multiply (i) by $N_2(\lambda)$ and use (iii) to infer
\[N_2(\lambda)+N_2(\lambda)(\Theta_2-\Theta_1)N_1(\lambda)=N_2(\lambda)R_{1,2}(\lambda)=N_1(\lambda).\]
Assertion (v) follows from (i) and (iv),
\begin{align*}
R_{\Theta_1, - \Theta_2}&(\lambda)R_{\Theta_2,\Theta_1}(\lambda)=\big(I_{H^{-1/2}}+(-\Theta_2-\Theta_1)N_1(\lambda)\big)\big(I_{H^{-1/2}}+(\Theta_1-\Theta_2)N_2(\lambda)\big)\\ &=I_{H^{-1/2}}+(\Theta_1-\Theta_2)N_2(\lambda)-(\Theta_2+\Theta_1)N_1(\lambda)
-(\Theta_2+\Theta_1)N_1(\lambda)(\Theta_1-\Theta_2)N_2(\lambda)\\
&=I_{H^{-1/2}}-2\Theta_2N_2(\lambda),
\end{align*}
thus finishing the proof.
\end{proof}

\subsection{Controlling the resolvent set of $\cL^\Theta$}
\label{sec:resTheta}
In our proof of Theorem~\ref{thm:WEvans} we will need to factor $\widehat M_{\widehat\Theta}$ and $N_\Theta$ through the Robin-to-Robin map, making use of Proposition~\ref{propRtoR} for some auxiliary boundary operator. This is made possible by the following result.

\begin{lemma}\label{T0}
Assume Hypothesis \ref{hyp:L}. For each $\lambda_0 \in \bbC$ there exists a finite rank operator $\Theta_0 \colon \Hp \to \Hm$ such that $\lambda_0 \in \rho(\cL^{\Theta_0})$.
\end{lemma}

\begin{proof}
If $\lambda \in \rho(\cL^N)$ it suffices to choose $\Theta = 0$, so we assume for the rest of the proof that $\lambda \in \sigma(\cL^N)$. Let $u_1, \ldots, u_k \in H^1(\Omega)$ be a basis for $\ker(\cL^N - \lambda)$, and similarly let $v_1, \ldots, v_k$ be a basis for the kernel of $(\cL^N)^* - \bar\lambda$. We define $\Theta$ on $S_L:= \Sp\{\gaD u_1, \ldots, \gaD u_k\}$ by
\[
	\Theta(\gaD u_i) = \R (\gaD v_i), \quad 1 \leq i \leq k,
\]
where $\R \colon \Hp \to \Hm$ is the Riesz map, and extend $\Theta$ to the rest of $\Hp$ by defining it to be zero on $S_L^{\perp}$. Note that $\operatorname{ran}\Theta = \R(S_{L^\dag})$, where $S_{L^\dag} := \Sp\{\gaD v_1, \ldots, \gaD v_k\}$.

To see that $\lambda$ is not an eigenvalue of $\cL^\Theta$, suppose there exists $u \in H^1(\Omega)$ such that
\[
	Lu = \lambda u, \qquad \gaN^L u + \Theta \gaD u=0.
\]
From the definition of $\R$, Green's second identity \eqref{e13}, and $\gaN^{L^\dagger}v_i=0$, we have
\[
	 \left< \gaN^L u, \R(\gaD v_i) \right>_{\Hm} = \llangle \gaN^L u, \gaD v_i \rrangle = \overline{\llangle \gaN^{L^\dag} v_i, \gaD u \rrangle} = 0,
\]
for each $i$, which means $\gaN^L u$ is $\Hm$-orthogonal to the subspace $\R(S_{L^\dag})$. On the other hand, we have $-\gaN^L u = \Theta \gaD u \in \R(S_{L^\dag})$, and so $-\gaN^L u = \Theta \gaD u = 0$. The fact that $\gaN^L u = 0$ means $u = c_1 u_1 + \cdots + c_k u_k$ for some constants $c_i$, and hence the definition of $\Theta$ yields
\[
	\Theta\gaD u = c_1 \mathcal{R} (\gaD v_1) + \cdots +  c_k \mathcal{R} (\gaD v_k).
\]
Since $\Theta\gaD u = 0$, this implies $c_1 = \cdots = c_k = 0$, hence $u = 0$ and $\lambda$ is not an eigenvalue of $\cL^\Theta$.
\end{proof}

\begin{rem}
\label{rem:Rohleder}
In the special case that $L$ is symmetric, it is possible to find a real number $\mu$ such that a given $\lambda_0 \in \bbC$ is not in the spectrum of $\cL^{\Theta_\mu}$, where $\Theta_\mu := \mu \J$. If $\lambda_0 \in \rho(\cL^N)$ then we choose $\mu = 0$, as in the proof of Lemma~\ref{T0}. If $\lambda_0 \in \sigma(\cL^N)$ we can use \cite[Theorem~3.2]{MR3206692}, which says that the eigenvalues of $\cL^{\Theta_\mu}$ are strictly monotone in $\mu$, and hence guarantees $\lambda_0$ is not an eigenvalue of $\cL^{\Theta_\mu}$ for small, nonzero $\mu$. Note that $\Theta_\mu = \mu \J$ is compact (because $\J$ is) and hence satisfies Hypothesis~\ref{hypT}. This construction will be used below, in the proof of Proposition~\ref{prop:Wmult}.
\end{rem}

\section{Eigenvalues of elliptic operators and Robin-to-Robin maps}
\label{sec:multiplicity}

In this section we relate the eigenvalues of the linear operators $\cL^D$, $\cL^{\Theta_1}$ and $\cL^{\Theta_2}$ to the eigenvalues of the nonlinear operator pencils $N_{\Theta_1}$, $M_{\Theta_2}$ and $R_{1,2}$ defined in the previous section. Here and below, we say that $\lambda\in\bbC$ is an eigenvalue of a pencil $T(\cdot)$ if $\ker T(\lambda)\neq\{0\}$. This material will be used in the proof of Theorem \ref{thm:WEvans} given in the next section, but it is also of independent interest as a further development of several known results that can be found in \cite{BE19} and the bibliography therein.
While it is not hard to show that the eigenvalues and their geometric multiplicities coincide, relating the algebraic multiplicities is significantly more involved.
Indeed, even defining the algebraic multiplicity $m_a\big(\lambda, T(\cdot)\big)$ of an eigenvalue of an operator pencil requires some work; see Definition \ref{JC} below. We summarize the main results of this section as follows.

\begin{theorem}\label{multR}
	Assume Hypothesis \ref{hyp:L} and let $\Theta_1$ and $\Theta_2$ satisfy Hypothesis \ref{hypT}.
	\begin{enumerate}
		\item If $\lambda\in\rho(\cL^{\Theta_1})$, then  $\lambda\in\sigma(\cL^{\Theta_2})$ if and only if $\lambda$ is an eigenvalue of the pencil $R_{1,2}(\cdot)$; moreover, $m_a(\lambda,\cL^{\Theta_2})=m_a\big(\lambda, R_{1,2}(\cdot)\big)$.
		\item  If $\lambda\in\rho(\cL^{\Theta_1})$, then  $\lambda\in\sigma(\cL^D)$ if and only if $\lambda$ is an eigenvalue of the pencil $N_{\Theta_1}(\cdot)$; moreover, $m_a(\lambda,\cL^D)=m_a\big(\lambda, N_{\Theta_1}(\cdot)\big)$.
		\item If $\lambda\in\rho(\cL^D)$, then  $\lambda\in\sigma(\cL^{\Theta_2})$ if and only if $\lambda$ is an eigenvalue of the pencil $M_{\Theta_2}(\cdot)$; moreover, $m_a(\lambda,\cL^{\Theta_2})=m_a\big(\lambda, M_{\Theta_2}(\cdot)\big)$.
	\end{enumerate}
\end{theorem}

	The relationship between the Robin eigenvalues and the Dirichlet-to-Robin map was recently described in an important paper \cite{BE19} that was a major step in our understanding of the subject. In fact, a version of the main tool that we employ here, Lemma \ref{lemma1}, is already contained in the proofs of \cite{BE19} for the case of Dirichlet-to-Robin maps.

The proof of Theorem \ref{multR} proceeds in two steps. First, in Section \ref{sec:geo}, we show that the geometric multiplicities agree. Then, in Section \ref{sec:alg}, we give a one-to-one correspondence between Jordan chains, which implies that the algebraic multiplicities agree as well. Throughout, we assume Hypothesis \ref{hyp:L}.

\subsection{Geometric multiplicity}
\label{sec:geo}
If $\lambda_0\in\bbC$ is an eigenvalue of a nonlinear pencil $T(\cdot)$, then its \emph{geometric multiplicity} $m_g\big(\lambda_0, T(\cdot)\big)$ is defined to be $\dim \ker T(\lambda_0)$. We set $m_g\big(\lambda_0, T(\cdot)\big)=0$ if  $\lambda_0$ is not an eigenvalue. If $T_0$ is a closed linear operator then the geometric multiplicity $m_g(\lambda_0, T_0)$ of $\lambda_0$ as an eigenvalue of $T_0$ is equal to the geometric multiplicity $m_g\big(\lambda_0, T(\cdot)\big)$ of the linear pencil $T(\lambda)=T_0-\lambda$.
We begin by relating the geometric multiplicities of eigenvalues of elliptic operators and eigenvalues of Robin-to-Dirichlet, Robin-to-Robin and Dirichlet-to-Robin maps.

\begin{lemma}\label{LNFred}
	Assume that $\Theta_1$ and $\Theta_2$ satisfy Hypothesis \ref{hypT}. If $\lambda_0\in\rho(\cL^{\Theta_1})$, then $N_{\Theta_1}(\lambda_0)$ and $R_{1,2}(\lambda_0)$ are Fredholm operators of index $0$, and the maps
	\begin{align*}
		(\gaN^L+\Theta_1\gaD)\big|_{\ker(\cL^D-\lambda_0)} \colon &\ker(\cL^D-\lambda_0)\longrightarrow\ker N_{\Theta_1}(\lambda_0), \\
		(\gaN^L+\Theta_1\gaD)\big|_{\ker(\cL^{\Theta_2}-\lambda_0)} \colon &\ker(\cL^{\Theta_2}-\lambda_0)\longrightarrow\ker R_{1,2}(\lambda_0),
	\end{align*}
	are bijections. In particular,
	\[
	\dim\ker (\cL^D-\lambda_0) = \dim\ker N_{\Theta_1}(\lambda_0), \quad 
	\dim\ker (\cL^{\Theta_2}-\lambda_0) = \dim\ker R_{1,2}(\lambda_0),
	\]
	and
	\begin{align*}
		\lambda_0\in\rho(\cL^D)\cap\rho(\cL^{\Theta_1}) \ &\Longleftrightarrow \ N_{\Theta_1}(\lambda_0) \text{ is invertible,} \\
		\lambda_0\in\rho(\cL^{\Theta_2})\cap\rho(\cL^{\Theta_1}) \ &\Longleftrightarrow \ R_{1,2}(\lambda_0) \text{ is invertible.}
	\end{align*}
Similarly, if $\lambda_0\in\rho(\cL^D)$, then $M_{\Theta_2}(\lambda_0)$ is a Fredholm operator of index $0$ and the map
	\[
	\gaD\big|_{\ker(\cL^{\Theta_2}-\lambda_0)} \colon \ker(\cL^{\Theta_2}-\lambda_0)\longrightarrow\ker M_{\Theta_2}(\lambda_0)
	\]
	is a bijection, hence
	\begin{align*}
		\lambda_0\in\rho(\cL^{\Theta_2})\cap\rho(\cL^D) \ &\Longleftrightarrow \ R_{1,2}(\lambda_0) \text{ is invertible.}
	\end{align*}
%
%
\end{lemma}

\begin{proof} We will prove the items concerning $\cL^D$ and $N_{\Theta_1}$; the remaining items can be shown in the same way using Proposition~\ref{TH4.11}.
	
	By Proposition \ref{domD} we know that  the operator $\cL^D-\lambda_0$ is Fredholm of index $0$, and its spectrum consists of eigenvalues of finite algebraic multiplicity, so $p := \dim\ker(\cL^D-\lambda_0)$ is finite.
Now let $\lambda_0\in\rho(\cL^{\Theta_1})$.  If $u\in\ker(\cL^D-\lambda_0)$ and $f=(\gaN^{L}+\Theta_1\gaD)u\in H^{-1/2}(\pO)$, then $N_{\Theta_1}(\lambda_0)f=\gaD u=0$, therefore $f\in\ker N_{\Theta_1}(\lambda_0)$ 
	and the map $(\gaN^L+\Theta_1\gaD)\big|_{\ker(\cL^D-\lambda_0)}$ from the statement of the theorem is well-defined. Since $\lambda_0\in\rho(\cL^{\Theta_1})$, the map $(\gaN^L+\Theta_1\gaD)\big|_{\ker(\cL^D-\lambda_0)}$ is injective.
	On the other hand, if $f\in\ker N_{\Theta_1}(\lambda_0)$, then there exists a unique $u\in \cD_L^1(\Omega)$ so that $(L-\lambda_0)u=0$ and $(\gaN^L+\Theta_1\gaD)u=f$; moreover, $\gaD u=N_{\Theta_1}(\lambda_0)f=0$. Therefore, $u\in\ker(\cL^D-\lambda_0)$ and the map $(\gaN^L+\Theta_1\gaD)\big|_{\ker(\cL^D-\lambda_0)}$ is surjective. In particular, $\dim\ker N_{\Theta_1}(\lambda_0)=p$. 
	
	Moreover, according to \cite[Theorem 4.10]{McL00},  there are $p$ linearly independent solutions to the adjoint homogeneous Dirichlet problem, denoted here by $v_1, \ldots, v_p$. The inhomogeneous problem $(L-\lambda_0)u=0$, $\gaD u=g\in H^{1/2}(\pO)$ is solvable if only if $\llangle \gaN^{L^\dag}v_j,  g \rrangle=0$ for each $j$. By the unique continuation principle \cite{BR12} associated with the adjoint problem, $\dim \operatorname{span}\big\{\gaN^{L^\dag}v_1, \ldots, \gaN^{L^\dag}v_p\big\}=p$.  Since $g\in\ran N_{\Theta_1}(\lambda_0)$ if and only if there is a solution of the inhomogeneous problem $(L-\lambda_0)u=0$, $\gaD u=g$, the codimension of $\ran N_{\Theta_1}(\lambda_0)$ is $p$. Hence, the range of $N_{\Theta_1}(\lambda_0)$ is closed and the operator $N_{\Theta_1}(\lambda_0)$ is Fredholm of index $0$.
	
	Finally, if $\lambda_0\in\rho(\cL^D)\cap\rho(\cL^{\Theta_1})$, then $N_{\Theta_1}(\lambda_0)$ is Fredholm of index $0$ and 
	$\dim \ker N_{\Theta_1}(\lambda_0)=\dim\ker(\cL^D-\lambda_0)=0$, therefore $N_{\Theta_1}(\lambda_0)$ is invertible. 
\end{proof}

\subsection{Algebraic multiplicity}
\label{sec:alg}

We first recall the definition of algebraic multiplicity for eigenvalues of nonlinear operator pencils; cf. \cite{BTG20} and the bibliography therein.

\begin{define}\label{JC}
Let $\mathcal U$ be an open subset of $\bbC$, and $T \colon \mathcal U \to \cB(\mathcal{H},\mathcal{K})$  be an analytic family of bounded operators between Hilbert spaces $\cH$ and $\cK$.
\begin{enumerate}
	\item Given $k\in\bbN$, we say that vectors $f_0,\ldots,f_{k-1} \in \cH $ form a {\em Jordan chain of length $k$} for $T$ at $\lambda_0 \in \mathcal U$ if $f_0 \neq 0$ and
	\begin{align}\label{kchain}
		\sum_{l=0}^j\frac{1}{l!}T^{(l)}(\lambda_0)f_{j-l}&=0,\quad 0 \leq j \leq k-1.
	\end{align}
	\item The rank $r(f_0)$ of a nonzero vector $f_0\in\ker T(\lambda_0)$ is the supremum of the lengths of all Jordan chains starting at $f_0$.
	\item If $\lambda_0$ is an eigenvalue and $\{f_{0,i}\}$ is a basis for $\ker T(\lambda_0)$, the \emph{algebraic multiplicity} of $\lambda_0$ is defined to be
	\begin{equation}
		m_a\big(\lambda_0,T(\cdot)\big) = \sum_{i=1}^{m_g(\lambda_0, T(\cdot))} r(f_{0,i}).
	\end{equation}
	If $\lambda_0$ is not an eigenvalue we set $m_a\big(\lambda_0,T(\cdot)\big) = 0$.
\end{enumerate}
\end{define}

Note that the algebraic multiplicity will be infinite if $\ker T(\lambda_0)$ is infinite-dimensional or if any eigenvector admits arbitrarily long Jordan chains. In the special case when $\mathcal{H} = \mathcal{K}$ and $T(\lambda) = T_0 - \lambda I$ for some $T_0 \in \cB(\cH)$, the definition of a Jordan chain reduces to the familiar one from linear algebra:
\[
(T_0 - \lambda_0) f_0 = 0, \quad (T_0 - \lambda_0) f_1 = f_0, \quad
(T_0 - \lambda_0) f_2 = f_1, \quad \ldots, \quad (T_0 - \lambda_0) f_{k-2} = f_{k-1}.
\]
However, Jordan chains for a nonlinear pencil can behave quite differently than their linear counterparts; see Example \ref{ex1} for an illustration.

\begin{rem}
\label{rem:mult_notation}
Two different notions of algebraic multiplicity occur in this paper: 1) for an eigenvalue of a linear operator; and 2) for an eigenvalue of a (possibly nonlinear) pencil. For a \emph{linear} pencil $T(\lambda) = T_0 - \lambda I$ we always have $m_a\big(\lambda_0,T(\cdot)\big) = m_a(\lambda_0,T_0)$ but in general these two notions of algebraic multiplicity are not the same.
Note that $\lambda_0$ is an eigenvalue of the pencil $T(\cdot)$ if and only if $0$ is an eigenvalue of the linear operator $T(\lambda_0)$. However, the algebraic multiplicity of $\lambda_0$ for the pencil $T(\cdot)$ does not necessarily coincide with the algebraic multiplicity of $0$ for the linear operator $T(\lambda_0)$, as the following example demonstrates.
\end{rem}

\begin{example}\label{ex1}
	Consider the nonlinear pencil 
	$D(\lambda) = \left[\begin{smallmatrix}1&0\\0&\lambda^2\end{smallmatrix}\right]$. 
	Then $f_0 = \left[ \begin{smallmatrix} 0 \\ 1 \end{smallmatrix}\right]$ forms the basis of $\ker D(0)$. The next vector $f_1$ in the Jordan chain associated with $f_0$ at $\lambda=0$ should satisfy the equation $D(0)f_1=0$, which implies $f_1=\alpha f_0$ for arbitrary $\alpha \in \bbC$. In particular, we can choose $f_1=f_0$ or $f_1=0$, so we see that vectors in the Jordan chain do not need to be linearly independent, and can even be zero, except for the eigenvector $f_0$ at the beginning of the chain. The next vector $f_2$ in the chain at $\lambda=0$ should satisfy the equation $D(0)f_2+2f_0=0$ which is impossible because $f_0\notin\ran D(0)$. Therefore, we obtain the absence of a chain beyond the generalized eigenvector $f_1$, implying $r(f_0) = 2$ and hence $m_a\big(0,D(\cdot)\big)=2$. On the other hand, $0$ is a simple eigenvalue of $D(0)$, so for $\lambda_0 = 0$ we have $m_a\big(\lambda_0,D(\cdot)\big) \neq m_a\big(0,D(\lambda_0)\big)$.
\end{example}

We now turn to the proof of Theorem \ref{multR}. As we will see,  this theorem is an almost immediate consequence of the following Lemma \ref{lemma1}
and its analogues, Lemmas \ref{lemma2-2} and \ref{lemma3}.

\begin{lemma}\label{lemma1}
	Assume that $\Theta_1$ satisfies Hypothesis \ref{hypT}, $\Theta_2 \in \cB\big(\Hp,\Hm\big)$ and $\lambda_0\in\rho(\cL^{\Theta_1})$. If $u_{-1}=0$, $u_0,\ldots,u_{k-1}\in D_L^1(\Omega)$ and $f_0,\ldots,f_{k-1}\in H^{-1/2}(\pO)$ satisfy
	\begin{equation}\label{chain1}
		(L-\lambda_0)u_j=u_{j-1},\quad (\gaN^L+\Theta_1\gaD) u_j=f_j
	\end{equation} 
	for $0\leq j\leq k-1$, then
		 \begin{equation}\label{Ruj}
			(\gaN^L+\Theta_2\gaD) u_j=\sum_{l=0}^j\frac{1}{l!}R^{(l)}_{1,2}(\lambda_0)f_{j-l}
		\end{equation}
	for $0\leq j\leq k-1$.
\end{lemma}

In turn, the proof of  Lemma \ref{lemma1}  is based on the following lemma.

\begin{lemma}\label{lemma2}
	Under the assumptions in Lemma \ref{lemma1}, for $\lambda$ near $\lambda_0$ and any $h\in H^{1/2}(\pO)$ let $v(\lambda)$ denote the unique solution to the boundary value problem
	\begin{equation}\label{chain2}
		(L^\dag-\bar\lambda)v(\lambda)=0,\quad (\gaN^{L^\dag}+\Theta^*_1\gaD) v(\lambda)=-(\Theta^*_2-\Theta^*_1) h. 	
	\end{equation}  Then
	\begin{align}\label{LR}
		\begin{split}
			\langle (L-\lambda)u_{j},v(\lambda)\rangle_{L^2(\Omega)}&={\llangle  R_{1,2}(\lambda)f_j-(\gaN^L+\Theta_2\gaD)u_j , h \rrangle}
		\end{split}	
	\end{align}
	for $0\leq j\leq k-1$.
\end{lemma}
\begin{proof}
	Green's identity \eqref{e13}, \eqref{chain1} and \eqref{chain2} yield
	\begin{align*}
		\begin{split}
			\langle (L-\lambda)u_{j},v(\lambda)\rangle_{L^2(\Omega)}&=\langle L u_{j},v(\lambda)\rangle_{L^2(\Omega)}-\langle  u_{j}, L^\dag v(\lambda)\rangle_{L^2(\Omega)}\\
			&=\overline{\llangle \gaN^{L^\dag}v(\lambda),  \gaD u_j \rrangle}-{\llangle \gaN^{L}u_j,  \gaD v(\lambda) \rrangle}\\
			&=\overline{\llangle -\Theta^*_1 \gaD v(\lambda)-(\Theta^*_2-\Theta^*_1) h,  \gaD u_j \rrangle}
			-{\llangle -\Theta_1\gaD u_j+f_j,  \gaD v(\lambda) \rrangle}\\
			&=-\llangle (\Theta_2-\Theta_1) \gaD u_j,  h \rrangle-{\llangle f_j,  \gaD v(\lambda) \rrangle}.
		\end{split}	
	\end{align*}
	Since $\lambda_0\in\rho(\cL^{\Theta_1})$, for $\lambda$ near $\lambda_0$ we also have $\lambda\in\rho(\cL^{\Theta_1})$ and so $N_{\Theta_1}(\lambda)$ is well-defined. Moreover, \eqref{chain2} yields $(N_{\Theta_1}(\lambda))^*(\Theta^*_2-\Theta^*_1) h=-\gaD v(\lambda)$. Therefore,
	\begin{align*}
		{\llangle f_j,  \gaD v(\lambda) \rrangle}&=-{\llangle f_j,  (N_{\Theta_1}(\lambda))^*(\Theta^*_2-\Theta^*_1) h \rrangle}=-\overline{\llangle  (\Theta^*_2-\Theta^*_1) h, N_{\Theta_1}(\lambda)f_j \rrangle}\\
		&=-\llangle   (\Theta_2-\Theta_1)N_{\Theta_1}(\lambda)f_j,  h \rrangle=-\llangle   (\Theta_2-\Theta_1)\gaD w_j,  h \rrangle,
	\end{align*}
	where $w_j(\lambda)$ denotes the solution to the boundary value problem
	\begin{equation}
		(L-\lambda)w_j(\lambda)=0,\quad (\gaN^L+\Theta_1\gaD) w_j(\lambda)=f_j. 	
	\end{equation} 
	Combining these equations, and adding and subtracting $\gaN^L$-terms, we get
	\begin{align*}
		\begin{split}
			&\langle (L-\lambda)u_{j},v(\lambda)\rangle_{L^2(\Omega)}=-\llangle   (\Theta_2-\Theta_1)(\gaD u_j-\gaD w_j),  h \rrangle\\
			&=\llangle  (\gaN^L+\Theta_1 \gaD) u_j,  h \rrangle-\llangle  (\gaN^L+\Theta_2 \gaD) u_j,  h \rrangle-\llangle  (\gaN^L+\Theta_1 \gaD) w_j,  h \rrangle+\llangle  (\gaN^L+\Theta_2 \gaD) w_j,  h \rrangle\\
			&=\llangle  f_j,  h \rrangle-\llangle  (\gaN^L+\Theta_2 \gaD) u_j,  h \rrangle+\llangle R_{1,2}(\lambda)f_j,  h \rrangle-\llangle  f_j,  h \rrangle\\
			&=\llangle R_{1,2}(\lambda)f_j
			- (\gaN^L+\Theta_2 \gaD) u_j,  h \rrangle,
		\end{split}	
	\end{align*}
which gives \eqref{LR}.
\end{proof}

\begin{proof}[Proof of Lemma \ref{lemma1}]
	Adding and subtracting $\lambda_0$ and using \eqref{chain1}, we rewrite \eqref{LR} as
	\begin{align}\label{LR1}
		\begin{split}
			\langle u_{j-1},v(\lambda)\rangle_{L^2(\Omega)}+{\llangle  (\gaN^L+\Theta_2\gaD)u_j , h \rrangle}={\llangle  R_{1,2}(\lambda)f_j , h \rrangle}+(\lambda-\lambda_0)&\langle u_{j},v(\lambda)\rangle_{L^2(\Omega)}.
		\end{split}
	\end{align}
	In particular, for $j=0$ the relation $u_{-1}=0$ yields
	\begin{align}\label{LR2}
		\begin{split}
			{\llangle  (\gaN^L+\Theta_2\gaD)u_0 , h \rrangle}&={\llangle  R_{1,2}(\lambda)f_0 , h \rrangle}+(\lambda-\lambda_0)\langle u_{0},v(\lambda)\rangle_{L^2(\Omega)}.
		\end{split}
	\end{align}
	Since $h$ is arbitrary, letting $\lambda=\lambda_0$ in \eqref{LR2} proves \eqref{Ruj} for $j=0$. 
	
	On the other hand, setting $\lambda=\lambda_0$ in \eqref{LR1} gives
	\begin{align}\label{LR3}
		\begin{split}
			\langle u_{j-1},v(\lambda_0)\rangle_{L^2(\Omega)}+{\llangle  (\gaN^L+\Theta_2\gaD)u_j , h \rrangle}={\llangle  R_{1,2}(\lambda_0)f_j , h \rrangle}.
		\end{split}
	\end{align}
	Moreover, subtracting \eqref{LR3} from \eqref{LR1} yields
	\begin{align}\label{minus}
		\begin{split}
			\llangle  (R_{1,2}(\lambda)-R_{1,2}(\lambda_0))f_j , h \rrangle=\langle u_{j-1},(v(\lambda)-v(\lambda_0))\rangle_{L^2(\Omega)}-(\lambda-\lambda_0)\langle &u_{j},v(\lambda)\rangle_{L^2(\Omega)}.
		\end{split}
	\end{align}
	To begin the proof of \eqref{Ruj} for $1\leq j\leq k-1$, we decompose
	\begin{equation*}
		R_{1,2}(\lambda)=\sum_{s=0}^{l-1}\frac{1}{s!}(\lambda-\lambda_0)^sR^{(s)}_{1,2}(\lambda_0)+R_{1,2}(l;\lambda),
	\end{equation*}
	where we have introduced $R_{1,2}(l;\lambda):=\sum_{s=l}^{\infty}\frac{1}{s!}(\lambda-\lambda_0)^sR^{(s)}_{1,2}(\lambda_0)$.
	In particular,
	\begin{align}\label{R1}
		&R_{1,2}(1;\lambda)=R_{1,2}(\lambda)-R_{1,2}(\lambda_0),\\\label{limR}
		&\lim_{\lambda\to\lambda_0}(\lambda-\lambda_0)^{-l}R_{1,2}(l;\lambda)=\frac{1}{l!}R_{1,2}^{(l)}(\lambda_0),\\\label{l+1}
		&R_{1,2}(l+1;\lambda)=R_{1,2}(l;\lambda)-\frac{1}{l!}(\lambda-\lambda_0)^lR^{(l)}_{1,2}(\lambda_0).
	\end{align}
	The proof of \eqref{Ruj} for $1\leq j\leq k-1$ is based on \eqref{LR1} and the following assertion:
	\begin{equation}\label{assert}
		-\langle u_{j-1},v(\lambda)\rangle_{L^2(\Omega)}=\sum_{l=1}^{j}(\lambda-\lambda_0)^{-l}{\llangle  R_{1,2}(l;\lambda)f_{j-l} , h \rrangle},\quad1\leq j\leq k.
	\end{equation}
Indeed, assuming \eqref{assert} and letting $\lambda\to\lambda_0$, \eqref{limR} yields
	\begin{equation}\label{assert0}
		-\langle u_{j-1},v(\lambda_0)\rangle_{L^2(\Omega)}=\sum_{l=1}^{j}\frac{1}{l!}{\llangle  R_{1,2}^{(l)}(\lambda_0)f_{j-l} , h \rrangle},\quad1\leq j\leq k.
	\end{equation}
	Adding \eqref{LR3} and \eqref{assert0} and using that $h$ is arbitrary yields \eqref{Ruj} for $1\leq j\leq k-1$. 
	
	So, it remains to prove \eqref{assert}. We use induction. For $j=1$ we use \eqref{minus} and \eqref{R1} as follows:
	\begin{align*}
		(\lambda-\lambda_0)^{-1}{\llangle  R_{1,2}(1;\lambda)f_{0} , h \rrangle} &=	(\lambda-\lambda_0)^{-1}({\llangle  R_{1,2}(\lambda)f_{0} , h \rrangle}-{\llangle  R_{1,2}(\lambda_0)f_{0} , h \rrangle})\\
		&=(\lambda-\lambda_0)^{-1}(	-(\lambda-\lambda_0)\langle u_{0},v(\lambda)\rangle_{L^2(\Omega)}+0) \\
		&=-\langle u_{0},v(\lambda)\rangle_{L^2(\Omega)},
	\end{align*}
	which gives \eqref{assert} for $j=1$.
	
	Let us assume that  \eqref{assert} holds for $j=m$. Taking the limit in \eqref{assert} and using  \eqref{limR}, we obtain
	\begin{equation}\label{m}
		-\langle u_{m-1},v(\lambda_0)\rangle_{L^2(\Omega)}=\sum_{l=1}^{m}\frac{1}{l!}{\llangle  R_{1,2}^{(l)}(\lambda_0)f_{j-l} , h \rrangle}.
	\end{equation}
	
	Now we prove  \eqref{assert} for $j=m+1$. Changing the summation index and using \eqref{R1} and  \eqref{l+1} yields
	\begin{align*}
		&\sum_{l=1}^{m+1}(\lambda-\lambda_0)^{-l}  R_{1,2}(l;\lambda)f_{m+1-l}=\sum_{l=2}^{m+1}(\lambda-\lambda_0)^{-l}  R_{1,2}(l;\lambda)f_{m+1-l}+(\lambda-\lambda_0)^{-1}  R_{1,2}(1;\lambda)f_{m}\\
		&=(\lambda-\lambda_0)^{-1}\left(\sum_{l=1}^{m}(\lambda-\lambda_0)^{-l}  R_{1,2}(l+1;\lambda)f_{m-l}+ R_{1,2}(\lambda)f_{m}-R_{1,2}(\lambda_0)f_{m}\right)\\
		&=(\lambda-\lambda_0)^{-1}\left(\sum_{l=1}^{m}(\lambda-\lambda_0)^{-l}  R_{1,2}(l;\lambda)f_{m-l}-\sum_{l=1}^{m}\frac{1}{l!}  R_{1,2}^{(l)}(\lambda_0)f_{m-l}+ R_{1,2}(\lambda)f_{m}-R_{1,2}(\lambda_0)f_{m}\right).
	\end{align*}
	Using the induction assumption, \eqref{minus} and \eqref{m}, we have
	\begin{align*}
		&\sum_{l=1}^{m+1}(\lambda-\lambda_0)^{-l}{\llangle  R_{1,2}(l;\lambda)f_{m+1-l} , h \rrangle}=(\lambda-\lambda_0)^{-1}\big(-\langle u_{m-1},v(\lambda)\rangle_{L^2(\Omega)}+\langle u_{m-1},v(\lambda_0)\rangle_{L^2(\Omega)}\\
		&+\langle u_{m-1},(v(\lambda)-v(\lambda_0))\rangle_{L^2(\Omega)}-(\lambda-\lambda_0)\langle u_{m},v(\lambda)\rangle_{L^2(\Omega)}\big)=-\langle u_{m},v(\lambda)\rangle_{L^2(\Omega)}.
	\end{align*}
	This proves \eqref{assert} for $j=m+1$ and completes the proof of Lemma \ref{lemma1}.
\end{proof}

Our next result is a direct corollary of Lemma \ref{lemma1}; it shows that the Jordan chains of $\cL^{\Theta_2}$ and $R_{1,2}$ are in one-to-one correspondence.

\begin{theorem}\label{Rchain}
	Assume that $\Theta_1$ and $\Theta_2$ satisfy Hypothesis \ref{hypT}. Let $\lambda_0\in\sigma(\cL^{\Theta_2})\cap\rho(\cL^{\Theta_1})$ and consider the analytic function
		$\lambda\mapsto R_{1,2}(\lambda)$ from $\rho(\cL^{\Theta_1})$ into $\mathcal{B}\big(H^{-1/2}(\pO)\big)$.
	\begin{enumerate}
		\item Let $u_0,\ldots,u_{k-1}\in\dom(\cL^{\Theta_2})$ be a Jordan chain of length $k$ for $\cL^{\Theta_2}$ at $\lambda_0$, and define 
		\begin{equation*}
			f_j:=(\gaN^L+\Theta_1\gaD) u_j\in H^{-1/2}(\pO),\quad 0\leq j\leq k-1. 
		\end{equation*}
	Then the vectors $f_0,\ldots,f_{k-1}$ form a Jordan chain of length $k$ for $R_{1,2}$ at $\lambda_0$.
		\item Let  $f_0,\ldots,f_{k-1}\in H^{-1/2}(\pO)$ be a Jordan chain of length $k$ for $R_{1,2}$ at $\lambda_0$. Set $u_{-1}=0$ and for $0\leq j\leq k-1$ let $u_j\in\cD_L^1(\Omega)$ be the unique solution to the boundary value problem
		\begin{equation}\label{bvpu}
			(L-\lambda_0)u_j=u_{j-1},\quad (\gaN^L+\Theta_1\gaD) u_j=f_j.
		\end{equation}
		Then the vectors $u_0,\ldots,u_{k-1}$ form a Jordan chain of length $k$ for $\cL^{\Theta_2}$ at $\lambda_0$. 
	\end{enumerate}
\end{theorem}
\begin{proof}
	(1)	We need to prove that $f_0 \neq 0$ and 
	\begin{align}\label{nonR}
		\sum_{l=0}^j\frac{1}{l!}R_{1,2}^{(l)}(\lambda_0)f_{j-l}&=0
	\end{align}
for $0\leq j\leq k-1.$
Since $\lambda_0\in\rho(\cL^{\Theta_1})$, we have $f_0\neq0$ as otherwise $u_0\in\cD_L^1(\Omega)$  would have solved the boundary value problem $(L-\lambda_0)u_0=0$, $(\gaN^L+\Theta_1\gaD) u_0=0$, implying that $\lambda_0\in\sigma(\cL^{\Theta_1})$. Since \eqref{chain1} holds, assertion \eqref{nonR} follows directly from \eqref{Ruj} in Lemma \ref{lemma1} because $u_j\in\dom(\cL^{\Theta_2})$.

	(2) From the definition of a Jordan chain for $R_{1,2}$ we have \eqref{nonR}, while \eqref{bvpu} is \eqref{chain1}. The right-hand side of formula  \eqref{Ruj} vanishes, implying that  $(\gaN^{L}+\Theta_2\gaD)u_j=0$  and so $u_j\in\dom(\cL^{\Theta_2})$ for $0\leq j\leq k-1$, thus finishing the proof. 
\end{proof}

We will now relate the Jordan chains for $\cL^{D}$ and $N_{\Theta_1}$. We start with the following lemma, which is an analogue and an easy consequence of Lemma \ref{lemma1}.

\begin{lemma}\label{lemma2-2}
	Assume that $\Theta_1$ satisfies Hypothesis \ref{hypT} and $\lambda_0\in\rho(\cL^{\Theta_1})$.  If $u_{-1}=0$, $u_0,\ldots,u_{k-1}\in D_L^1(\Omega)$ and $f_0,\ldots,f_{k-1}\in H^{-1/2}(\pO)$ satisfy
%
%
	\begin{equation}\label{chain4}
		(L-\lambda_0)u_j=u_{j-1},\quad (\gaN^L+\Theta_1\gaD) u_j=f_j
	\end{equation}
	for $0 \leq j \leq k-1$, then
		 \begin{equation}\label{Ruj2}
			\gaD u_j=\sum_{l=0}^j\frac{1}{l!}N^{(l)}_{\Theta_1}(\lambda_0)f_{j-l}
		\end{equation}
for all $0\leq j\leq k-1$.
\end{lemma}

\begin{proof}
Fix $\nu>0$ and let $\Theta_2=\nu\J$. We will apply Lemma \ref{lemma1} with this $\Theta_2$.  From Proposition \ref{propRtoR} we have $R_{1,2}(\lambda_0)= I_{H^{-1/2}}+(\nu\J - \Theta_1) N_{\Theta_1}(\lambda_0)$. Using Lemma \ref{lemma1}, we compute for $0\leq j\leq k-1$
\begin{align*}
	(\gaN^L+\nu\J\gaD) u_j&=\sum_{l=0}^j\frac{1}{l!}R^{(l)}_{12}(\lambda_0)f_{j-l}\\
	&=\big(I_{H^{-1/2}}+(\nu\J - \Theta_1) N_{\Theta_1}(\lambda_0)\big)f_j+\sum_{l=1}^j\frac{1}{l!}(\nu\J - \Theta_1)N^{(l)}_{\Theta_1}(\lambda_0)f_{j-l}.
\end{align*}
Dividing both sides by $\nu$, passing to the limit as $\nu\to\infty$ and using injectivity of $\J$ gives \eqref{Ruj2}.
\end{proof}

As before, the following one-to-one correspondence of the Jordan chains for $\cL^{D}$ and $N_{\Theta_1}$ is an easy consequence of Lemma \ref{lemma2-2}. We omit the proof as it is identical to the proof of Theorem \ref{Rchain} with Lemma \ref{lemma1} replaced by Lemma \ref{lemma2-2}.

\begin{theorem}\label{Nchain}
	Assume that $\Theta_1$ satisfies Hypothesis \ref{hypT}. Let $\lambda_0\in\sigma(\cL^D)\cap\rho(\cL^{\Theta_1})$ and consider the analytic function
		$\lambda\mapsto N_{\Theta_1}(\lambda)$
	from $\rho(\cL^{\Theta_1})$ into $\mathcal{B}(H^{-1/2}(\pO),H^{1/2}(\pO))$.
	\begin{enumerate}
		\item Let $u_0,\ldots,u_{k-1}\in\dom(\cL^{D})$ be a Jordan chain of length $k$ for $\cL^{D}$ at $\lambda_0$, and define 
		\begin{equation*}
			f_j:=(\gaN^L+\Theta_1\gaD) u_j\in H^{-1/2}(\pO)\quad \hbox{for}\,\, 0\leq j\leq k-1.
		\end{equation*}
		Then the vectors $f_0,\ldots,f_{k-1}$ form a Jordan chain of length $k$ for $N_{\Theta_1}(\cdot)$ at $\lambda_0$.
		\item Let  $f_0,\ldots,f_{k-1}\in H^{-1/2}(\pO)$ be a Jordan chain of length $k$ for $N_{\Theta_1}(\cdot)$ at $\lambda_0$. Set $u_{-1}=0$ and for $0 \leq j \leq k-1$ let $u_j\in\cD_L^1(\Omega)$ be the unique solution of the boundary value problem
		\begin{equation}\label{bvpuN}
			(L-\lambda_0)u_j=u_{j-1},\quad (\gaN^L+\Theta_1\gaD) u_j=f_j.
		\end{equation}
		Then the vectors $u_0,\ldots,u_{k-1}$ form a Jordan chain of length $k$ for $\cL^{D}$ at $\lambda_0$. 
	\end{enumerate}
\end{theorem}

We will now relate the Jordan chains for $\cL^{\Theta_2}$ and $M_{\Theta_2}$. Again, we begin with an analogue and a consequence of Lemmas \ref{lemma1} and \ref{lemma2-2}.

\begin{lemma}\label{lemma3}
	Assume that $\Theta_2$ satisfies Hypothesis \ref{hypT} and $\lambda_0\in\rho(\cL^{D})$. If $u_{-1}=0$, $u_0,\ldots,u_{k-1}\in D_L^1(\Omega)$ and $f_0,\ldots,f_{k-1}\in H^{-1/2}(\pO)$ satisfy
%
%
	\begin{equation}\label{chain3}
		(L-\lambda_0)u_j=u_{j-1},\,\,\,\,\gaD u_j=f_j
	\end{equation}
	for $0 \leq j \leq k-1$, then,
	\begin{equation}\label{Muj}
			(\gaN^L+\Theta_2\gaD) u_j=\sum_{l=0}^j\frac{1}{l!}M^{(l)}_{\Theta_2}(\lambda_0)f_{j-l}
		\end{equation}
holds for $0\leq j\leq k-1$.
\end{lemma}

\begin{proof}
	Pick an auxiliary operator $\Theta_1$ satisfying Hypothesis \ref{hypT} and such that $\lambda_0\in\rho(\cL^{\Theta_1})$, using Lemma \ref{T0}. Then the operators $R_{1,2}$, $N_{\Theta_1}$ and $M_{\Theta_1}$ are all well-defined for $\lambda$ near $\lambda_0$, and
	\begin{equation}\label{M1N1}
		M_{\Theta_2}(\lambda)=R_{1,2}(\lambda)M_{\Theta_1},\quad M_{\Theta_1}N_{\Theta_1}=I_{\Hm}
	\end{equation}
by Proposition \ref{propRtoR}. We introduce auxiliary vectors 
\begin{equation}\label{gi}
	g_{j}:=(\gaN^L+\Theta_1\gaD)u_j=\gaN^L u_j+\Theta_1f_j,\quad 0\leq j\leq k-1.
\end{equation}
By Lemma \ref{lemma1} with $f_j$ replaced by $g_j$, we know that
\begin{equation}\label{tgi}
		(\gaN^L+\Theta_2\gaD) u_j=\sum_{l=0}^j\frac{1}{l!}R^{(l)}_{12}(\lambda_0)g_{j-l}, \quad 0\leq j\leq k-1.
\end{equation}
Thus, to verify \eqref{Muj} we have to show that the right-hand sides of  \eqref{Muj} and \eqref{tgi} are equal. Using the first equality in \eqref{M1N1}, the product rule, changing the order of summation, and relabeling yields
\begin{align}
	\begin{split}
		\sum_{l=0}^j\frac{1}{l!}M^{(l)}_{\Theta_2}(\lambda_0)f_{j-l}&=\sum_{l=0}^j\frac{1}{l!}\sum_{n=0}^l\binom{l}{n} R^{(n)}_{1,2}(\lambda_0)M^{(l-n)}_{\Theta_1}(\lambda_0)f_{j-l}\\
		&=\sum_{l=0}^j\frac{1}{l!}R^{(l)}_{1,2}(\lambda_0)\sum_{m=0}^{j-l}\frac{1}{m!} M^{(m)}_{\Theta_1}(\lambda_0)f_{j-l-m},\quad 0\leq j\leq k-1.
	\end{split}
\end{align}
This shows that  \eqref{Muj} holds provided we know that 
\begin{equation}\label{gifor}
	g_j=\sum_{l=0}^{j}\frac{1}{l!} M^{(l)}_{\Theta_1}(\lambda_0)f_{j-l},\quad 0\leq j\leq k-1;
\end{equation}
in other words, by \eqref{gi}, that an analogue of the lemma holds for $M_{\Theta_1}$ in place of $M_{\Theta_2}$. To show \eqref{gifor}, we recall that 
\begin{equation*}
	(L-\lambda_0)u_j=u_{j-1},\quad	(\gaN^L+\Theta_1\gaD) u_j=g_j,\quad 0\leq j\leq k-1,
\end{equation*}
and thus Lemma \ref{lemma2} applies with $f_j$ replaced by $g_j$, giving
\begin{equation}\label{2.5}
	f_j=\gaD u_j=\sum_{m=0}^j\frac{1}{m!}N^{(m)}_{\Theta_1}(\lambda_0)g_{j-m}, \quad1\leq j\leq k-1.
\end{equation}
Plugging \eqref{2.5} with $j-l$ in the right-hand side of \eqref{gifor}, changing the order of summation,  relabeling and using the second equality in \eqref{M1N1} yields
\begin{align*}
		\sum_{l=0}^{j}\frac{1}{l!} M^{(l)}_{\Theta_1}(\lambda_0)f_{j-l}&=\sum_{l=0}^{j}\frac{1}{l!} M^{(l)}_{\Theta_1}(\lambda_0)\sum_{m=0}^{j-l}\frac{1}{m!}N^{(m)}_{\Theta_1}(\lambda_0)g_{j-l-m}\\
		&=\sum_{m=0}^{j}\frac{1}{(j-m)!} \Big(\sum_{l=0}^{j-m}\binom{j-m}{l}M^{(l)}_{\Theta_1}(\lambda_0)N^{(j-m-l)}_{\Theta_1}(\lambda_0)\Big)g_{m}\\
		&=\sum_{m=0}^{j}\frac{1}{(j-m)!}(M_{\Theta_1}(\lambda)N_{\Theta_1}(\lambda))^{(j-m)}\big|_{{}_{\lambda=\lambda_0}}g_{m}=g_j,
\end{align*}
as needed in \eqref{gifor}.
\end{proof}

Our next result shows that the Jordan chains of $\cL^{\Theta_2}$ and $M_{\Theta_2}$ are in one-to-one correspondence.
The proof is again omitted as it is identical to the proof of Theorem \ref{Rchain} with Lemma \ref{lemma1} replaced by Lemma \ref{lemma3}.

\begin{theorem}
	\label{Jordan chainsM}
	Assume that $\Theta_2$ satisfies Hypothesis \ref{hypT}. Let 
	$\lambda_0\in\sigma(\cL^{\Theta_2})\cap\rho(\cL^D)$ and 
	consider the analytic function $\lambda\mapsto M_{\Theta_2}(\lambda)$ from $\rho(\cL^D)$ to $\cB\big(\Hp,\Hm\big)$.
	\begin{enumerate}
		\item Let $u_0,\ldots,u_{k-1}\in\dom(\cL^{\Theta_2})$ be a Jordan chain of length $k$ for $\cL^{\Theta_2}$ at $\lambda_0$, and define 
		\begin{equation*}
		f_j:=\gaD u_j\in H^{1/2}(\pO)\,\,\,\hbox{for}\,\,0\leq j\leq k-1.
		\end{equation*}
		Then the vectors $f_0,\ldots,f_{k-1}$ form a Jordan chain of length $k$ for $M_{\Theta_2}$ at $\lambda_0$.
		\item Let $f_0,\ldots,f_{k-1}\in H^{1/2}(\pO)$ be a Jordan chain of length $k$ for $M_{\Theta_2}$ at $\lambda_0$. Set $u_{-1}=0$ and for $0 \leq j \leq k-1$ let $u_j\in H^1(\Omega)$ be the unique solution of the boundary value problem
		\begin{equation*}
			(L-\lambda_0)u_j=u_{j-1},\quad \gaD u_j=f_j.
		\end{equation*}
		Then the vectors $u_0,\ldots,u_{k-1}$ form a Jordan chain of length $k$ for $\cL^{\Theta_2}$ at $\lambda_0$. 
	\end{enumerate}
\end{theorem}

\begin{proof}[Proof of Theorem \ref{multR}]
	Item (1)  follows from the fact that the geometric multiplicities $\dim \ker(\cL^{\Theta_2}-\lambda_0)$ and $\dim \ker R_{1,2}(\lambda_0)$ are equal by Lemma \ref{LNFred}, while Theorem \ref{Rchain} gives a one-to-one correspondence between the Jordan chains of $\cL^{\Theta_2}$ and $R_{1,2}$.
	Similarly using Lemma \ref{LNFred}, Theorem \ref{Nchain},  and Theorem \ref{Jordan chainsM}, one can prove items (2) and (3).
\end{proof}

\begin{rem}\label{rNFred1} Sometimes it is convenient to use the Riesz map $\R$ in order to replace operator pencils acting between $H^{-1/2}(\pO)$ and $H^{1/2}(\pO)$ by pencils acting in a single space  $H^{-1/2}(\pO)$. In this context we formulate the following easy consequence of Theorem \ref{multR} and Lemma \ref{LNFred}.
If $\Theta_1$ satisfies Hypothesis \ref{hypT}, then the operator  $\R N_{\Theta_1}(\lambda)$ is Fredholm of index $0$ for all $\lambda\in\rho(\cL^{\Theta_1})$. Moreover, for any $\lambda_0\in\rho(\cL^{\Theta_1})$ we have equality of the geometric and algebraic multiplicities: $m_g(\lambda_0,\cL^D) = m_g\big(\lambda_0, \R N_{\Theta_1}(\cdot)\big)$ and $m_a(\lambda_0,\cL^D)=m_a\big(\lambda_0, \R N_{\Theta_1}(\cdot)\big)$.
\end{rem}

\section{The Evans function and operator pencils}
\label{sec:Evans}

We are now ready to construct the promised multi-dimensional Evans function, thus proving Theorem ~\ref{thm:WEvans}. In Section~\ref{sec:prelimBp} we recall definitions and important properties of the Schatten--von Neumann ideals and $p$-modified Fredholm determinants. In Section \ref{sec:Bp} we show that Hypothesis~\ref{hyp:Bp} implies $E(\lambda) - I \in \cB_p\big(\Hp\big)$ for sufficiently large $p$, where $E(\lambda)$ is the operator defined in \eqref{eq:Edef}. This guarantees that the $p$-modified Fredholm determinant $\cE(\lambda) = \det_p E(\lambda)$ is a well-defined analytic function on $\rho(\cL^\Theta) \cap \rho(\widehat\cL^D)$. The proof relies on elliptic estimates, and hence is sensitive to the smoothness of the boundary and the coefficients of $L$; see Remark \ref{rem:betterp} and Proposition \ref{propBp}. Finally, in Section \ref{sec:det} we study the function $E$ and its determinant $\cE$ in more detail, proving that $E$ is completely meromorphic in $\bbC$ and relating the order of zeros, poles and essential singularities of its determinant to the eigenvalues of $\cL^D$, $\widehat\cL^D$, $\cL^\Theta$ and $\widehat\cL^\Theta$, as stated in \eqref{detorder2} and \eqref{detorder}.

\subsection{Schatten--von Neumann ideals and modified determinants}
\label{sec:prelimBp}
We first recall some definitions and basic facts that are needed for the statement and proof of Theorem ~\ref{thm:WEvans}. Let $\cH$ and $\cK$ be separable Hilbert spaces, and $A \colon \cH \to \cK$ a compact linear operator. The singular values of $A$, denoted $s_k(A)$, are the eigenvalues of the compact, selfadjoint operator $\sqrt{A^*A}$. For $p \geq 1$, the Schatten--von Neumann ideal $\cB_p(\cH,\cK)$ is defined to be the set of all compact $A$ for which the $p$th Schatten norm
\[
	\|A\|_p = \left(\sum_{k=1}^\infty (s_k(A))^p\right)^{1/p}
\]
is finite. This is a Banach space with respect to $\| \cdot \|_p$. It is a two-sided ideal, in the sense that if $A \in \cB_p(\cH,\cK)$, $C \in \cB(\cK,\cK_1)$ and $B \in \cB(\cH_1,\cH)$ for some separable Hilbert spaces $\cH_1$ and $\cK_1$, then $CAB \in \cB_p(\cH_1,\cK_1)$. 
If $\cH=\cK$ we abbreviate $\cB_p(\cH)=\cB_p(\cH, \cK)$. We denote by $\cF(\cH)$ the two-sided ideal of finite rank operators in $\cH$, and note that $\cF(\cH) \subset \cB_p(\cH)$ for all $p$.

We next recall some basic notions and facts about modified Fredholm determinants. If $B\in\cB_p(\cH)$ for some integer $p \geq 1$, the $p$-modified Fredholm determinant $\det_p(I+B)$ is defined by the formula
\begin{equation}\label{defpdet}
{\det}_p(I+B)=\prod_{n=1}^\infty\left((1+\mu_n)\exp\left(\sum_{k=1}^{p-1}
k^{-1}(-1)^k \mu_n^k\right)\right),
\end{equation}
where $\{\mu_n\}$ are the eigenvalues of $B$, repeated according to their algebraic multiplicity. We refer to \cite[Section I.7]{Yaf} for the properties of determinants that we will need. In particular, we recall from \cite[(I.7.17)]{Yaf} that if $F\in\cF(\cH)$ is a finite rank operator, then 
\begin{equation}\label{detpdet}
	{\det}_p(I+F)=\det(I+F)\exp\left(\sum_{k=1}^{p-1}k^{-1}(-1)^k \tr \big(F^k\big) \right)
\end{equation}
for any $p \geq 1$, where $\det$ stands for the usual determinant, that is, the product of the respective eigenvalues. We also recall, see \cite[p.~44]{Yaf}, that
\begin{equation}
\label{detprod}
	\operatorname{det}_p(I + A_1A_2) = \operatorname{det}_p(I + A_2A_1)
\end{equation}
for any bounded operators $A_1$ and $A_2$ such that both $A_1
A_2$ and $A_2A_1$ are in $\cB_p(\cH)$. (Note that $A_1$ and $A_2$ are not individually required to be $\cB_p$.)
We will also need a more advanced result from \cite[Lemma~4.1]{GMZ1}, generalizing \cite[(I.7.19)]{Yaf} from $p=2$ and saying that if both $B_1$ and $B_2$ are in $\cB_p(\cH)$, then
\begin{equation}\label{PRODdet}
	{\det}_p\big((I+B_1)(I+B_2)\big)={\det}_p(I+B_1){\det}_p(I+B_2)\exp\big(\tr T_p(B_1,B_2)\big),
\end{equation}
where $T_p(\cdot,\, \cdot)$ is a polynomial function with
$T_p(B_1,B_2)\in\cB_1(\cH)$, so that its trace is well defined.

We also need to take determinants of operators that depend analytically or meromorphically on a parameter. If $\mathcal{V}$ is an open subset of $\bbC$ and $B\colon\mathcal{V}\to \cB(\cH)$ is a $\cB_p(\cH)$-valued analytic function, then the function $\lambda\mapsto\det_p(I+B(\lambda))$ is analytic; see \cite[Section~I.7]{Yaf} and also \cite{GLMZ05,Howland}. Moreover, if $B_1\colon\mathcal{V}\to \cB(\cH)$ and $B_2\colon\mathcal{V}\to \cB(\cH)$ are two $\cB_p(\cH)$-valued analytic functions, then there exists an analytic function $\varphi:\mathcal{V}\to\bbC$  such that 
\begin{equation}\label{prodform}
{\det}_p\big((I+B_1(\lambda))(I+B_2(\lambda))\big) = e^{\varphi(\lambda)} {\det}_p\big(I+B_1(\lambda)\big){\det}_p\big(I+B_2(\lambda)\big) 
\end{equation}
for all $\lambda\in \mathcal V$.
This is a direct consequence of \eqref{PRODdet} with $\varphi(\lambda) = \tr T_p(B_1(\lambda),B_2(\lambda))$.

The situation for meromorphic functions is slightly more complicated.
If $F\colon\mathcal{V}\to\cF(\cH)$ is meromorphic, then the function $\lambda\mapsto\det(I+F(\lambda))$ is also meromorphic. However, if $B\colon\mathcal{V}\to\cB(\cH)$ is a meromorphic function with values in $\cB_p(\cH)$ (but not necessarily $\cF(\cH)$), then the poles of $B(\cdot)$ may produce essential singularities of the function $\det_p(I+B(\cdot))$, as can be seen from \eqref{defpdet}. If $B_1\colon\mathcal{V}\to \cB(\cH)$ and $B_2\colon\mathcal{V}\to \cB(\cH)$ are two $\cB_p(\cH)$-valued meromorphic functions, then there exists a meromorphic function $\varphi\colon\mathcal{V}\to\bbC$ such that 
\begin{equation}\label{prodformMER}
{\det}_p\big((I+B_1(\lambda))(I+B_2(\lambda))\big) = e^{\varphi(\lambda)}{\det}_p\big(I+B_1(\lambda)\big){\det}_p\big(I+B_2(\lambda)\big)
\end{equation}
for all $\lambda\in\mathcal{V}$.
This is a direct consequence of \eqref{PRODdet} with $\varphi(\lambda)=\tr \big(T_p(B_1(\lambda),B_2(\lambda))\big)$.

To prove Theorem \ref{thm:WEvans}, we need to understand the poles and singularities of functions of type $\det_p\big(I+B(\cdot)\big)$, where $B(\cdot)$ is $\cB_p(\cH)$-valued and meromorphic. Following \cite{GLMZ05} and \cite{Howland}, for a function $f\colon\mathcal{V}\to\bbC$ that is analytic except at a discrete set of singularities (which could be either poles or essential singularities) and whose zeros do not accumulate in $\bbC$, we define by \eqref{m:def} the multiplicity function $m(\lambda_0; f)$, and recall the formula \eqref{m:def2} that holds for meromorphic $f$. 
The multiplicity $m(\lambda_0,f)$, however,  is defined even when $\lambda_0$ is an essential singularity of $f$, in which case it can assume any integer value. For instance the function $f(\lambda) = \lambda^k e^{1/\lambda}$ has an essential singularity at the origin with $m(0;f) = k$ for any $k \in \bbZ$. As a result, one must be careful when interpreting the multiplicity. If we know a priori that $f$ is analytic at $\lambda_0$, then $m(\lambda_0;f) > 0$ if and only if $f$ has a zero at $\lambda_0$. Without this a priori knowledge, however, we can only conclude from $m(\lambda_0;f) > 0$ that $f$ has a zero or an essential singularity at $\lambda_0$. 

We finally note that if $f_1$ and $f_2$ are two functions of this type, then $m(\lambda_0; f_1 f_2)=m(\lambda_0; f_1)+m(\lambda_0; f_2)$. In particular, if $\varphi$ is meromorphic near $\lambda_0$ (and hence analytic in a punctured neighborhood of $\lambda_0$), then $m(\lambda_0; e^\varphi f)=m(\lambda_0; f)$, since
\[
	m(\lambda_0; e^\varphi) = \frac{1}{2\pi i} \int_{\p D(\lambda_0;\epsilon)} \varphi'(\lambda)\,d\lambda = 0
\]
by the fundamental theorem of calculus.

\subsection{$\cB_p$ properties of Robin-to-Dirichlet maps}
\label{sec:Bp}

We now return to the elliptic setting described in the Introduction and Section \ref{sec:prelim}. Before stating the main result of this section, we mention an additional set of hypotheses on $\Omega$, $L$ and $\Theta$; cf. Hypothesis~\ref{hyp:Bp}.

\begin{hyp}
\label{hyp:smallerp}
Assume, in addition to Hypothesis~\ref{hyp:L}, that:
\begin{enumerate}
	\item $\pO$ is of class $C^{1,1}$;
	\item $a_{jk} = a_{kj}$ for $1 \leq j,k \leq n$;
	\item $a_{jk}$, $b_j$, $d_j$ and $q$ are Lipschitz;
	\item $\Theta = \J \tilde\Theta$, where $\tilde\Theta \in \cB\big(\Hp\big)$ and $\J \colon \Hp \to \Hm$ is inclusion.

\end{enumerate}
\end{hyp}

These stronger assumptions are not needed for the statement or proof of Theorem~\ref{thm:WEvans}. However, they give improved $\cB_p$ properties for $E(\lambda)-I$, and hence lead to a refinement of the theorem, as explained in Remark~\ref{rem:betterp}.

\begin{prop}
\label{propBp}
Under the  hypotheses of Theorem~\ref{thm:WEvans},
\begin{enumerate}\item if $\lambda \in \rho(\cL^\Theta) \cap \rho(\widehat \cL^D)$, then $E(\lambda) - I \in \cB_p\big(\Hm\big)$,
\item if $\lambda \in \rho(\cL^\Theta) \cap \rho(\widehat \cL^{\widehat\Theta})$, then
$N_\Theta(\lambda)-\widehat{N}_{\widehat\Theta}(\lambda)\in\cB_p\big(H^{-1/2}(\pO),H^{1/2}(\pO)\big)$
\end{enumerate}
 for any $p>2(n-1)$. If Hypothesis~\ref{hyp:smallerp} also holds for $L$, $\Theta$ and $\widehat L$, $\widehat\Theta$, then $p>n-1$ suffices.
\end{prop}

To motivate the proof, let $\lambda \in \rho(\cL^\Theta) \cap \rho(\widehat \cL^D)$, so that $E(\lambda) = \widehat M_{\widehat\Theta} (\lambda) N_\Theta(\lambda) $ is defined. If we additionally assume that $\lambda \in\rho(\widehat\cL^{\widehat\Theta})$, then $\widehat N_{\widehat\Theta}(\lambda)$ is also defined, and we can write
\begin{equation}
\label{Edecomp}
	E(\lambda) - I = \widehat M_{\widehat\Theta}(\lambda) \big(N_\Theta(\lambda) - \widehat N_{\widehat\Theta}(\lambda) \big).
\end{equation}
Therefore, we start by studying the difference of Robin-to-Dirichlet maps. The key observation is that the principal parts of $L$ and $\widehat L$ coincide, so the difference $L - \widehat L$ is first order. As a result, $N_\Theta - \widehat N_{\widehat\Theta}$ has better mapping properties than $N_\Theta$ and $\widehat N_{\widehat\Theta}$ on their own; cf. Lemma \ref{lemRtoD}.

\begin{lemma}\label{lem:RtDBp}
If the hypotheses of Theorem ~\ref{thm:WEvans} are satisfied and $\lambda \in \rho(\cL^\Theta) \cap \rho(\widehat \cL^{\widehat\Theta})$, then $N_\Theta(\lambda) - \widehat N_{\widehat\Theta}(\lambda)$ is in $\cB\big(\Hm,H^1(\pO) \big)$ and depends continuously on $\lambda$. If Hypothesis~\ref{hyp:smallerp} also holds, then $N_\Theta(\lambda) - \widehat N_{\widehat\Theta}(\lambda)$ is in $\cB\big(\Hm,H^{3/2}(\pO) \big)$ and depends continuously on $\lambda$.
\end{lemma}

\begin{proof}
	Fix $g \in \Hm$ and let $u, \widehat u$ be the unique solutions to the boundary value problems
	\begin{equation}
		\label{uBVP}
		Lu = \lambda u, \quad \gaN^L u + \Theta \gaD u = g
	\end{equation}
	and
	\begin{equation}\label{hatuBVP}
		\widehat L \widehat u = \lambda \widehat u, \quad \gaN^{\widehat{L}} \widehat u + \widehat\Theta \gaD \widehat u = g
	\end{equation}
	respectively, so that $N_\Theta(\lambda)g = \gaD u$ and $\widehat N_\Theta(\lambda)g = \gaD \widehat u$. From Proposition \ref{TH4.11} we have $\|u\|_{H^1(\Omega)} \leq c \|g\|_{\Hm}$, and similarly for $\widehat u$.

	The difference $w:= u-\widehat u$ satisfies the boundary value problem
	\begin{equation}\label{eqn:w}
	(L - \lambda) w = (\widehat L - L)\widehat u, \quad \gaN^L w + \Theta \gaD w=(\gaN^{\widehat{L}}-\gaN^{{L}})\widehat u + (\widehat\Theta - \Theta) \gaD \widehat u.
	\end{equation}
	Since $\widehat L - L$ is a first-order differential expression with bounded coefficients, we have
	\begin{equation}\label{est1}
	\| (\widehat L - L)\widehat u \|_{L^2(\Omega)} \leq c \|\widehat u\|_{H^1(\Omega)}.
	\end{equation}
	Moreover, from Lemma~\ref{lem:gaNdiff} we get $(\gaN^{\widehat{L}}-\gaN^{{L}})\widehat u = \sum_{j}\nu_j\gaD\big((\widehat{d}_j-d_j)\widehat{u}\big) \in L^2(\pO)$, and we easily obtain the estimate
	\begin{equation}\label{est2}
		\|(\gaN^{\widehat{L}}-\gaN^{{L}})\widehat u\|_{L^2(\partial\Omega)}
	\leq c \|\widehat u\|_{H^1(\Omega)}.
	\end{equation}
	Finally, the assumptions on $\Theta$ and $\widehat\Theta$ in Hypothesis~\ref{hyp:Bp} imply $\Theta \gaD \widehat u,  \widehat\Theta \gaD \widehat u \in L^2(\pO)$, with
	\begin{equation}\label{est3}
		\big\| (\widehat\Theta - \Theta) \gaD \widehat u \big\|_{L^2(\pO)}
		\leq c \|\widehat u\|_{H^1(\Omega)}.
	\end{equation}
Applying Proposition~\ref{TH4.24} to \eqref{eqn:w} and using \eqref{est1}--\eqref{est3}, it follows that $\gaD w \in H^1(\pO)$, with
\[	
	\| \gaD w\|_{H^1(\pO)} 
	\leq c \| \widehat u\|_{H^1(\Omega)}.
\]	
We therefore obtain
\begin{equation}
	\big\| N_\Theta(\lambda)g - \widehat N_{\widehat\Theta}(\lambda)g \big\|_{H^1(\pO)} = \| \gaD w\|_{H^1(\pO)} 
	\leq c \|\widehat u\|_{H^1(\Omega)} 
	\leq c \|g\|_{\Hm}
\end{equation}
and conclude that $N_\Theta(\lambda) - \widehat N_{\widehat\Theta}(\lambda)$ is bounded from $\Hm$ to $H^1(\pO)$.
	
	To prove continuity, we vary $\lambda$ in a neighborhood of a fixed $\lambda_0 \in \rho(\cL^\Theta) \cap \rho(\widehat \cL^\Theta)$, letting $u(\lambda)$ and $\widehat u(\lambda)$ denote the correspond solutions to \eqref{uBVP} and \eqref{hatuBVP}, respectively. Define
\[
	v(\lambda) := w(\lambda) - w(\lambda_0) = u(\lambda) - u(\lambda_0) + \widehat u(\lambda_0) - \widehat u(\lambda),
\]
so that
\[
	\big(N_\Theta(\lambda) - \widehat N_{\widehat\Theta}(\lambda)\big)g - \big(N_\Theta(\lambda_0) - \widehat N_{\widehat\Theta}(\lambda_0)\big)g = \gaD v(\lambda).
\]
A calculation shows that $v(\lambda)$ solves the boundary value problem
\begin{align*}		
	\big(L - \lambda_0\big)v(\lambda) &= (\lambda - \lambda_0) \big(u(\lambda) - \widehat u(\lambda) \big) + (L - \widehat L) \big(\widehat u(\lambda_0) - \widehat u(\lambda) \big) \\
	\gaN^L v(\lambda) + \Theta \gaD v(\lambda) &= (\gaN^{\widehat{L}}-\gaN^{{L}})(\widehat u(\lambda) - \widehat u(\lambda_0)) + (\widehat\Theta - \Theta) \gaD (\widehat u(\lambda) - \widehat u(\lambda_0)).
\end{align*}
As in the first part of the proof, we use Proposition~\ref{TH4.24} to obtain
\begin{align*}		
	\| \gaD v(\lambda) \|_{H^1(\pO)}
	&\leq c \left( |\lambda - \lambda_0| \big\|u(\lambda) - \widehat u(\lambda) \big\|_{L^2(\Omega)} + 
	\big\| \widehat u(\lambda_0) - \widehat u(\lambda)\big\|_{H^1(\Omega)} \right).
\end{align*}
For the first term we note that $\|u(\lambda) - \widehat u(\lambda) \|_{L^2(\Omega)} \leq \|u(\lambda) - \widehat u(\lambda) \|_{H^1(\Omega)} \leq c \|g\|_{\Hm}$ in a neighborhood of $\lambda_0$. For the second term we write, as in the proof of Lemma~\ref{lemRtoD},
\[
	\widehat u(\lambda) - \widehat u(\lambda_0) = \big(\cL^\Theta - \lambda\big)^{-1}\big((\lambda-\lambda_0) \widehat u(\lambda_0)\big).
\]
It follows that
\[
	\big\|\widehat u(\lambda) - \widehat u(\lambda_0)\big\|_{H^1(\Omega)} 
	\leq c \big\| (\lambda-\lambda_0) \widehat u(\lambda_0) \big\|_{L^2(\Omega)} 
	\leq c |\lambda - \lambda_0| \big\|g\big\|_{\Hm}
\]
and so we obtain
\begin{equation}
	\| \gaD v(\lambda) \|_{H^1(\pO)} \leq c |\lambda - \lambda_0| \big\|g\big\|_{\Hm}
\end{equation}
for all $g \in \Hm$ and all $\lambda$ sufficiently close to $\lambda_0$.

The second statement is proved similarly so we just sketch the argument, defining $u$, $\widehat u$ and $w$ as above. The fact that $\pO$ is $C^{1,1}$ implies each $\nu_j$ is Lipschitz, so Lemma~\ref{lem:gaNdiff} implies $(\gaN^{\widehat{L}}-\gaN^{{L}})\widehat u = \sum_{j}\nu_j\gaD\big((\widehat{d}_j-d_j)\widehat{u}\big) \in \Hp$, and the additional assumptions on $\Theta$ and $\widehat\Theta$ imply $\Theta\gaD w \in \Hp$ and $(\widehat\Theta - \Theta) \gaD \widehat u \in \Hp$.
Then $w$ satisfies \eqref{eqn:w}, with $\gaN^L w \in \Hp$, so \cite[Theorem~4.18]{McL00} implies $w \in H^2(\Omega)$, with the estimate
\[
	\|w\|_{H^2(\Omega)} \leq c \|\gaN^L w\|_{\Hp} \leq c \|g\|_{\Hm}.
\]
It follows from the trace theorem that $\gaD w \in H^{3/2}(\pO)$ and $\| \gaD w \|_{H^{3/2}(\pO)} \leq c \|g\|_{\Hm}$, so $N_\Theta - \widehat N_{\widehat\Theta}$ is bounded from $\Hm$ to $H^{3/2}(\pO)$ as claimed. Similar estimates can be used to prove continuity in $\lambda$, as above.
\end{proof}

The proof of Proposition \ref{propBp} is an easy consequence of Lemma \ref{lem:RtDBp} and the following result.

\begin{lemma}\label{compact}
Let $\Omega \subset \bbR^n$ be a bounded Lipschitz domain and fix $-1 \leq s < t \leq 1$. The inclusion $H^t(\pO) \subset H^s(\pO)$ is of class $\cB_p$ for each $p > (n-1)/(t-s)$.
\end{lemma}

When $\pO$ is smooth this follows from \cite[Lemma 4.7]{BLL}. For the Lipschitz case we use the definition of $H^s(\pO)$, in terms of local coordinates and a partition of unity, to reduce the problem to that of a smooth, compact manifold without boundary, where the result of \cite{BLL} then applies. 

\begin{proof}
By \cite[Definition 3.28]{McL00} there exist finite collections of open sets $\{W_j\}$ and $\{\Omega_j\}$ in $\bbR^n$ with the following properties:
\begin{enumerate}
	\item $\pO \subset \bigcup W_j$,
	\item $W_j \cap \Omega = W_j \cap \Omega_j$ for each $j$,
	\item each $\Omega_j$ is equivalent (via rigid motion) to a Lipschitz hypograph.
\end{enumerate}
By (3) we mean that there exists a rigid motion $\kappa_j$ of $\bbR^n$ and a Lipschitz function $\zeta_j \colon \bbR^{n-1} \to \bbR$ so that
\[
	\kappa_j(\Omega_j) = \{x = (x',x_n) \in \bbR^n : x_n < \zeta_j(x')\}.
\]
Next, let $\{\phi_j\}$ be a partition of unity subordinate to $\{W_j\}$. Given a function $u \colon \pO \to \bbC$, we define functions $u_j \colon \bbR^{n-1} \to \bbC$ by
\[
	u_j(x') = (\phi_j u)(\kappa_j^{-1}(x', \zeta_j(x')))
\]
for each $j$. The $H^s(\pO)$ norm is then defined for $0 \leq s \leq 1$ by
\[
	\|u\|_{H^s(\pO)} = \sum_j \|u_j\|_{H^s(\bbR^{n-1})}.
\]
It can be shown that any other choice of $\{W_j\}$, $\{\Omega_j\}$ and $\{\phi_j\}$ will yield an equivalent norm.

Note that $\operatorname{supp}(\phi_j u) \subset \pO \cap W_j$, hence $u_j$ is supported in $P\kappa_j(W_j \cap \pO)$, where $P \colon \bbR^n \to \bbR^{n-1}$ denotes projection onto the first $n-1$ coordinates. Since there are only finitely many $W_j$, we can find a large hypercube $C = [-R,R]^{n-1}$ so that $\operatorname{supp} \phi_j \subset C$ for each $j$. Identifying opposing faces of $C$, we can thus view each $u_j$ as a function on an $(n-1)$-torus. As a result, we can write the embedding $H^t(\pO) \subset H^s(\pO)$ as a composition
\begin{align}
	H^t(\pO) \longrightarrow \bigoplus_j H^t(\bbT^{n-1}) \longrightarrow \bigoplus_j H^s(\bbT^{n-1}) 
	\longrightarrow H^s(\pO).
\end{align}
The definition of $H^s(\pO)$ and $H^t(\pO)$ implies boundedness of the first and last maps, and \cite[Lemma~4.7]{BLL} says that the embedding $H^t(\bbT^{n-1}) \subset H^s(\bbT^{n-1})$ is of class $\cB_p$ for $p > (n-1)/(t-s)$.
\end{proof}

\begin{rem}
The topology of $\bbT^{n-1}$ is irrelevant to the above argument; it is simply a device to identify a bounded domain in $\bbR^{n-1}$ with a subset of a compact manifold without boundary, allowing us to apply the result of \cite{BLL} without modification.
\end{rem}

We are now ready to prove the main result of this section.

\begin{proof}[Proof of Proposition \ref{propBp}]
If $\lambda \in \rho(\cL^\Theta) \cap \rho(\widehat\cL^{\widehat\Theta}) \cap \rho(\widehat \cL^D)$, then using \eqref{Edecomp} we can write $E(\lambda) - I$ as the composition of three operators,
\[
	\Hm \xrightarrow{N_\Theta(\lambda) - \widehat N_{\widehat\Theta}(\lambda)} H^1(\pO) \xrightarrow{\text{inclusion}} \Hp \xrightarrow{\widehat M_{\widehat\Theta}(\lambda)} \Hm,
\]
where the first and last are bounded (by Lemmas \ref{lem:RtDBp} and \ref{lemDtoR}, respectively), and the second one is $\cB_p$ for any $p > 2(n-1)$, by Lemma \ref{compact} with $s=1/2$ and $t=1$. Similarly, if Hypothesis~\ref{hyp:smallerp} holds we have
\[
	\Hm \xrightarrow{N_\Theta(\lambda) - \widehat N_{\widehat\Theta}(\lambda)} H^{3/2}(\pO) \xrightarrow{\text{inclusion}} \Hp \xrightarrow{\widehat M_{\widehat\Theta}(\lambda)} \Hm,
\]
and the result follows from Lemma \ref{compact} with $s=1/2$ and $t=3/2$.

In either case, for an appropriate choice of $p$ we have $E(\lambda) - I \in \cB_p\big(\Hm\big)$ for all $\lambda \in \rho(\cL^\Theta) \cap \rho(\widehat\cL^{\widehat\Theta}) \cap \rho(\widehat \cL^D)$. Moreover, from Lemma \ref{lem:RtDBp}, we get that $E(\lambda)-I$ depends continuously on $\lambda$ in the $\cB_p\big(\Hm\big)$ norm. Since $\cB_p\big(\Hm\big)$ is a Banach space, and hence is closed, and  $\rho(\cL^\Theta) \cap \rho(\widehat\cL^{\widehat\Theta}) \cap \rho(\widehat \cL^D)$ is dense in $\rho(\cL^\Theta) \cap \rho(\widehat \cL^D)$, we find that $E(\lambda) - I \in \cB_p\big(\Hm\big)$ for all $\lambda \in \rho(\cL^\Theta) \cap \rho(\widehat \cL^D)$, which completes the proof of (1) in the proposition. The proof of (2) is analogous (and simpler).
\end{proof}

\begin{rem}\label{MNNM} Proposition \ref{propBp} says $ \widehat M_{\widehat\Theta}(\lambda) N_\Theta(\lambda)-I\in\cB_p\big(H^{-1/2}(\pO)\big)$
for all $\lambda\in\rho(\widehat{\cL}^D)\cap\rho(\cL^{\Theta})$. Analogous arguments show that $N_\Theta(\lambda)\widehat M_{\widehat\Theta}(\lambda) -I \in \cB_p\big(H^{1/2}(\pO)\big)$ for the same set of $\lambda$.
\end{rem}

Finally, we establish $\cB_p$ properties of the Robin-to-Robin map. The proof is much simpler than the above results, and only depends on the properties of the boundary operators.

\begin{prop}\label{RBp}
If Hypothesis \ref{hyp:Bp} holds and $\lambda \in\rho(\cL^{\Theta_1})$, then $R_{\Theta_1,\Theta_2}(\lambda)-I_{\Hm}$  is of class $\cB_p$ for each $p > 2(n-1)$. Moreover, if there exist $\Theta' \in \cB\big(\Hp\big)$ such that  $\Theta_2 - \Theta_1 = \J\Theta'$, then $R_{\Theta_1,\Theta_2}(\lambda)-I_{\Hm}$  is of class $\cB_p$ for each $p > n-1$.
\end{prop}

The second condition holds if both $\Theta_1$ and $\Theta_2$ satisfy Hypothesis~\ref{hyp:smallerp}(4). This observation suffices for the improved version of Theorem~\ref{thm:WEvans} that was promised in Remark~\ref{rem:betterp}. Stronger conclusions are possible with additional conditions on the boundary operators. For instance, if $\Theta_2$ is a finite rank perturbation of $\Theta_1$, then $R_{\Theta_1,\Theta_2} - I$ is of class $\cB_p$ for every $p$.

\begin{proof}
It follows from Proposition \ref{propRtoR} that  $R_{\Theta_1,\Theta_2}(\lambda) - I =(\Theta_2 - \Theta_1) N_1(\lambda)$. Using Hypothesis~\ref{hyp:Bp}, we have $\Theta_2 - \Theta_1 = \iota^*(\tilde\Theta_2 - \tilde\Theta_1)$ for $\tilde\Theta_1,\tilde\Theta_2 \in \cB(\Hp,L^2(\pO)\big)$, and so $\Theta_2 - \Theta_1$ is $\cB_p$ for $p>2(n-1)$ because $\iota^*$ is. Since $N_{\Theta_1}(\lambda)\in \cB\big(H^{-1/2}(\pO),H^{1/2}(\pO)\big)$, this proves the first statement. The second statement follows immediately since $\J$ is of class $\cB_p$ for any $p > n-1$.
\end{proof}

\subsection{The determinant}
\label{sec:det}

In this section we complete the proof of Theorem \ref{thm:WEvans} by analyzing the zeros and singularities of $\cE = \det_p E$, thus establishing the index formulas~\eqref{detorder2} and \eqref{detorder}.

Our strategy is as follows. Using Lemma \ref{lemRtoD}, Lemma \ref{lemDtoR} and Proposition  \ref{propBp}, we see that the function $\lambda\mapsto E(\lambda)=\widehat{M}_{\widehat{\Theta}}(\lambda)N_\Theta(\lambda)$ can be extended from the set $\rho(\widehat{\cL}^D)\cap\rho(\cL^\Theta)$ (where it is analytic) to the entire complex plane $\bbC$ as a meromorphic function whose values are operators 
of the type $I+B$ with  $B\in\cB_p\big(H^{-1/2}(\partial\Omega)\big)$. Thus, the $p$-modified Fredholm determinant $\cE(\lambda)=\det_p E(\lambda)$ is defined and analytic for all $\lambda$ except the poles of $E$. The poles of $E$ produce essential singularities of $\cE$, which is why Theorem \ref{thm:WEvans} involves the meromorphic function $\varphi$ and the multiplicity $m(\lambda_0; \cE)$. 

Our main task is thus to relate $m(\lambda_0; \cE)$ to the algebraic multiplicity of $\lambda_0$ as an eigenvalue of the operators $\cL^D$, $\widehat{\cL}^D$, $\cL^\Theta$ and $\cL^{\widehat{\Theta}}$. To do this we relate the eigenvalues of these operators to the eigenvalues of the respective Robin-to-Robin and Robin-to-Dirichlet operator pencils using results obtained in Section \ref{sec:multiplicity}, but we must overcome the following obstacle. While $N_\Theta(\cdot)$ fails to be invertible at the eigenvalues of $\cL^D$, it is not defined at the eigenvalues of $\cL^\Theta$. As a result, it is not immediately clear what happens to the determinant in the intersection of these two spectra.
To resolve this,
%
%
%
%
we factor $N_\Theta(\cdot)$ through the Robin-to-Dirichlet map $N_{\Theta_1}(\cdot)$ and Robin-to-Robin map $R_{\Theta_1,\Theta}(\cdot)$ associated with an auxiliary boundary operator $\Theta_1$, so that one factor has zeros but no poles, and vice versa for the other.
%
A similar factorization is used for the pencil $\widehat{M}_{\widehat{\Theta}}(\cdot)$.

As a result, we obtain a local description of the determinant, namely \eqref{detorder2}, around \emph{every} point in the complex plane. The ODE example in Section \ref{sec:interval} illustrates this phenomena, since the Dirichlet-to-Neuman map $M(\lambda)$ from \eqref{defMLAM} has singularities at the points of the Dirichlet spectrum while the determinant of $M(\lambda)$ is equal to $-\lambda$ and thus has removable singularities.

We recall that an operator-valued function $B\colon\mathcal{V}\to\cB(\cH)$ is \emph{meromorphic} at $\lambda_0 \in \mathcal V$ if it has a Laurent expansion $B(\lambda)=\sum_{k=-N}^{\infty}B_k(\lambda-\lambda_0)^k$ near $\lambda_0$, with $B_k\in\cB(\cH)$ for each $k$. It is said to be \emph{completely meromorphic} at $\lambda_0$ if, in addition, the coefficients in the principal part of the Laurent expansion satisfy the finite rank conditions $B_k\in\cF(\cH)$ for $-N \leq k \leq -1$.

We will need the following fact, cf.\ \cite[Lemma 4.2]{Howland} or \cite[Lemma 5.3]{GLMZ05}.

\begin{lemma}\label{lem:lem5.3}
Assume that $B\colon\mathcal{V}\to\cB(\cH)$ is a completely meromorphic $\cB_p(\cH)$-valued function such that
\begin{equation}\label{eq1detdetp}
	I+B(\lambda)=S(\lambda)\big(I+F(\lambda)\big)
\end{equation}
for all $\lambda\in\mathcal{V}$, where $S\colon\mathcal{V}\to\cB(\cH)$ is an analytic function whose values are invertible operators and $F\colon\mathcal{V}\to\cF(\cH)$ is a completely meromorphic function whose values are finite rank operators. Then
\begin{equation}\label{eqdetdetp}
	m\big(\lambda_0; {\det}_p(I+B(\cdot))\big)=m\big(\lambda_0; \det(I+F(\cdot))\big)
\end{equation}
for each $\lambda_0\in\mathcal{V}$, where $m(\lambda_0; \cdot)$ is the multiplicity function defined in \eqref{m:def}. If, in addition, $S(\lambda)-I\in\cB_p(\cH)$ for all $\lambda\in\mathcal{V}$, then there exists a meromorphic function $\varphi\colon\mathcal V\to\bbC$ such that
\begin{equation}\label{eqdetsphi}
	{\det}_p\big(I+B(\lambda)\big) = e^{\varphi(\lambda)} \det\big(I+F(\lambda)\big)
\end{equation}
for all $\lambda\in\mathcal{V}$.
\end{lemma}

The statement \eqref{eqdetdetp} about multiplicities does not require the assumption that $S(\lambda)-I\in\cB_p(\cH)$. With this extra assumption, however, we get \eqref{eqdetsphi}, which tells us more about the structure of the determinant itself, not just its multiplicity.

\begin{proof} We refer to \cite[Lemma 4.2]{Howland} for the proof of \eqref{eqdetdetp}.
Under the additional assumption $S(\lambda)-I\in\cB_p(\cH)$, \eqref{prodformMER} implies
\[
	{\det}_p\big(I+B(\lambda)\big) = e^{\tilde\varphi(\lambda)}{\det}_p\big(S(\lambda)\big){\det}_p\big(I+F(\lambda)\big)
\]
for some meromorphic function $\tilde\varphi$. Since ${\det}_p S(\lambda)$ is nonzero, it can be written as $e^{\psi(\lambda)}$ for some analytic function $\psi$, and the result follows from \eqref{detpdet}.
\end{proof}

\begin{rem}\label{rem-invcl} The sets \begin{equation}\label{BFalg}
I+\cB_p(\cH)=\big\{ I + B : B\in\cB_p(\cH)\big\}, \quad I+\cF(\cH)=\big\{ I + F : F\in\cF(\cH)\big\},
\end{equation} are inverse-closed sub-semigroups of $\cB(\cH)$ by multiplication: If $B_1,B_2\in\cB_p(\cH)$, then $(I+B_1)(I+B_2)-I\in\cB_p(\cH)$, and if $B\in\cB_p(\cH)$ and $I+B$ is invertible in $\cB(\cH)$, then $(I+B)^{-1}-I\in\cB_p(\cH)$, and analogously for $\cF(\cH)$. Indeed, 
\[
	(I+B_1)(I+B_2)-I=B_2+B_1(I+B_2)\in\cB_p(\cH)
\]
and
\[
	(I+B)^{-1}-I=(I+B)^{-1}\big(I - (I+B)\big)=-(I+B)^{-1}B\in\cB_p(\cH)
\]
since $\cB_p(\cH)$ is an ideal. For future reference we remark that if $D=I+F$ for some $F\in\cF(\cH)$ and $S$ is an invertible operator in $\cB(\cH)$, then $S^{-1}DS-I\in\cF(\cH)$ and the identity
\begin{equation}\label{DSid}
	DS=S\widetilde{D}
\end{equation}
holds, where $\widetilde{D}:=S^{-1}DS$ and $\widetilde{D}-I$ is in $\cF(\cH)$. Moreover,  $\widetilde{D}$ has the same determinant as $D$.
\end{rem}

One of the main tools used in what follows is a ``canonical'' representation 
of  nonlinear operator pencils and their local equivalence, a theory that goes back to \cite{GohSig}; see also \cite[Chapter IX]{GGK90}, \cite{BTG20}, \cite{LS10} and the vast literature therein. We now briefly recall some relevant facts.

Let $\cH$ be a Hilbert space and $T\colon\mathcal{V}\rightarrow\mathcal{B}(\mathcal{H})$ be an analytic operator-valued
function on an open subset $\mathcal{V}$ of the complex plane. Assume that $T(\lambda_0)$
is Fredholm of index zero for some
$\lambda_0\in\mathcal{V}$. Then there exists an operator
$F\in\mathcal{B}(\mathcal{H})$ of finite rank such that
$T(\lambda_0)+F$ is invertible. Since $T(\lambda)\in\mathcal{B}(\mathcal{H})$ and $T(\cdot)$ is continuous
in $\lambda$, the operator $G(\lambda):=T(\lambda)+F$ is invertible for all
$\lambda$ in some open neighborhood $\mathcal{U}$ of $\lambda_0$ in $\mathcal V$, and thus we can write
\begin{equation}
	T(\lambda)=G(\lambda)-F=G(\lambda)\big(I-G(\lambda)^{-1}F\big),\quad \lambda\in\mathcal{U}.
\end{equation}
Since  the operator $F$ is of finite rank, $\ker F$ has a finite-dimensional complement $\mathcal{H}_0$ in $\mathcal{H}$. Let $P$ be
the projection in $\mathcal{H}$ along $\ker F$ onto
$\mathcal{H}_0$. 
It follows that
\begin{equation}
	I-G(\lambda)^{-1}F=\big(I-PG(\lambda)^{-1}FP\big)\big(I-(I-P)G(\lambda)^{-1}FP\big).
\end{equation}
We put $G_1(\lambda):=I-(I-P)G(\lambda)^{-1}FP$ and note that $G_1$ is
defined and analytic in 
$\mathcal U$. Furthermore, the values of $G_1$ are
invertible operators on $\mathcal{H}$; in fact
\[
	G_1(\lambda)^{-1}=I+(I-P)G(\lambda)^{-1}FP.
\]
We thus obtain
\begin{equation}\label{eq2}
T(\lambda)=G(\lambda)\big(I-PG(\lambda)^{-1}FP\big)G_1(\lambda), \quad
\lambda \in \mathcal U,
\end{equation}
where $G$ and $G_1$ are analytic operator-valued functions on $\mathcal{U}$ and their values are invertible.

This motivates the following definition.

\begin{define}
\label{def:equivalent}
	Let $\mathcal{V}$ be an open set in $\mathbb{C}$ and let $T\colon \mathcal{V}\to\mathcal{B}(\mathcal{H})$ and $D\colon \mathcal{V}\to\mathcal{B}(\mathcal{H})$ be analytic operator-valued functions. We say that $T$ and $D$ are {\em equivalent} at a point $\lambda_0 \in \mathcal V$
	if there exists an open neighborhood $\mathcal{U}$ of $\lambda_0$ in $\mathcal{V}$ such that
	\begin{equation}
	T(\lambda)=S_1(\lambda)D(\lambda)S_2(\lambda)
	\end{equation}
	for all $\lambda\in\mathcal{U}$, where $S_1(\lambda), S_2(\lambda)\in\mathcal{B}(\mathcal{H})$ are invertible operators that depend analytically on $\lambda$.
\end{define}

Based on formula \eqref{eq2}, one can further decompose the middle term and obtain the following well-known result that goes back to \cite{GohSig}; see also \cite[Theorem XI.8.1]{GGK90}, \cite{LS10}, \cite[Theorem 3.10]{BTG20} and references therein.

\begin{theorem}\label{t1.1}
	Let $T \colon \mathcal{V}\rightarrow\mathcal{B}(\mathcal{H})$ be an analytic operator-valued
	function, and assume that $T(\lambda_0)$ is Fredholm with index zero for some $\lambda_0\in\mathcal{V}$.
	Then $T$ is equivalent at $\lambda_0$ to an analytic operator-valued
	function $D$, defined in a neighborhood $\mathcal{U}$ of $\lambda_0$, of the form
	\begin{equation}\label{indices}
	D(\lambda)=P_0+(\lambda-\lambda_0)^{k_1}P_1+\cdots+(\lambda-\lambda_0)^{k_r}P_r,
	\end{equation}
	where $r\in\mathbb{N}\cup\{0\}$, $P_0, P_1,\ldots,P_r$ are mutually disjoint projections such that $I-P_0$ is of finite rank and $P_1,\ldots,P_r$ have rank
	one, and $0<k_1\leq k_2\ldots\leq
	k_r$. Moreover, there exist operators $G(\lambda), G_1(\lambda)$ and $G_2(\lambda)$  so that
	\begin{equation}\label{eq5}
	T(\lambda)=G(\lambda)G_2(\lambda)D(\lambda)G_1(\lambda)
	\end{equation}
	for $\lambda \in \mathcal U$, where $G$, $G_1$ and $G_2$ are analytic with invertible values, and $G(\lambda)=T(\lambda)+F$, where $F$ is of finite rank. Furthermore, the operators $G_1(\lambda)-I$, $D(\lambda)-I$ and $G_2(\lambda)-I$ in \eqref{eq5} are of finite rank, and hence belong to $\mathcal{B}_p(\cH)$ for every $p$.
\end{theorem}

Note that for $T(\lambda)$ as in \eqref{eq5}, the operator $T(\lambda)-I$ is not necessarily in $\mathcal{B}_p(\cH)$, and $D(\lambda)$ from \eqref{indices} is not necessarily invertible for $\lambda\neq\lambda_0$.

\begin{prop}\label{invertT}
	Let $T \colon \mathcal V \to \cB(\cH)$ satisfy the hypotheses of Theorem \ref{t1.1}.
		\begin{enumerate}
	\item For each $\lambda\in\mathcal{U}$, the operator $T(\lambda)$ in \eqref{eq5} is invertible if and only if $D(\lambda)$ is invertible. 

		\item $T(\lambda)$ is invertible for some (and hence all) $\lambda$ in $\mathcal{U}\setminus\{\lambda_0\}$ if and only if
		\begin{equation}\label{eq1}
			P_0+P_1+\cdots+P_r=I.
		\end{equation}
		
		\item $T(\lambda_0)$ is invertible  if and only if $r=0$ and $P_0 = I$.	\end{enumerate}
\end{prop}
\begin{proof}
	
Since $G, G_1$ and $G_2$ in \eqref{eq5} take invertible values, the required assertions directly follow from Theorem \ref{t1.1} and the fact that 
	\begin{align*}
		\dim\ker D(\lambda)&=\dim\ker(P_0+P_1\cdots +P_r) \ \text{ for } \lambda\neq\lambda_0,\\
		\dim\ker D(\lambda_0)&=\dim\ker P_0,
	\end{align*}
by formula \eqref{indices}.\end{proof}
 As seen in the following simple example, the pencil $D$  in \eqref{indices} is obtained from $T$ by elementary transformations on rows and columns.
\begin{example}
Consider the linear pencil $T(\lambda) = \left[\begin{smallmatrix}\lambda&1\\0&\lambda\end{smallmatrix}\right]$, for which $\lambda=0$ is an eigenvalue with $m_a\big(0,T(\cdot)\big) = 2$. One can show that $T$ is equivalent at $\lambda_0 = 0$ to the nonlinear pencil $D(\lambda) = \left[\begin{smallmatrix}1&0\\0&0\end{smallmatrix}\right] + \lambda^2 \left[\begin{smallmatrix}0&0\\0&1\end{smallmatrix}\right]$. Indeed, using the elementary matrices
\[E_1=\begin{bmatrix}0&1\\1&0\end{bmatrix}, \ \
E_2=\begin{bmatrix}1&0\\-\lambda&1\end{bmatrix},\ \
E_3=\begin{bmatrix}1&-\lambda\\0&1\end{bmatrix}, \ \
E_4=\begin{bmatrix}1&0\\0&-1\end{bmatrix},
\]
one verifies the identity $E_4E_2T(\lambda)E_1E_3=D(\lambda)$ needed for \eqref{eq5}.
As discussed in Example \ref{ex1}, one also has $m_a\big(0,D(\cdot)\big)=2$.
\end{example}

We next explain how the algebraic multiplicity of an eigenvalue for a nonlinear pencil $T(\cdot)$ is encoded in the indices $k_1,\ldots,k_r$ in \eqref{indices}. This require an extra condition on $T(\cdot)$.

\begin{hyp}\label{mainT}
	Assume that
	\begin{enumerate}
		\item $T \colon \mathcal{V}\rightarrow\mathcal{B}(\mathcal{H})$ is an analytic operator-valued function,
		\item $T(\lambda)$ is Fredholm with index zero for each $\lambda\in\mathcal{V}$,
		\item The set of $\lambda\in\mathcal{V}$ for which $T(\lambda)$ is not invertible is discrete.
	\end{enumerate}
\end{hyp}
\noindent We stress that if $T$ satisfies Hypothesis \ref{mainT}, then \eqref{eq1} holds for each $\lambda_0\in\mathcal{V}$.
The following well-known result goes back to \cite{GohSig}; see also \cite[Chapter XI]{GGK90} and \cite[Theorem 3.10]{BTG20}.

\begin{lemma}\label{Fr}
Suppose $T \colon \mathcal V \to \cB(\cH)$ satisfies Hypothesis \ref{mainT} and  let $\lambda_0\in\mathcal{V}$. Then the algebraic multiplicity $m_{a}\big(\lambda_0,T(\cdot)\big)$, in the sense of Definition \ref{JC}, is equal to $k_1+\cdots+k_r$ for the indices in Theorem \ref{t1.1}.
\end{lemma}

Next, we describe how the algebraic multiplicity is encoded in the $p$-modified determinant.

\begin{prop}\label{detcalc}
Suppose $T \colon \mathcal V \to \cB(\cH)$ satisfies Hypothesis \ref{mainT} and let $\lambda_0\in\mathcal{V}$.
	\begin{enumerate}
		\item In a neighborhood $\mathcal{U}$ of $\lambda_0$ the operator $T(\lambda)$  is represented by \eqref{eq5}, with
		\begin{align}\label{detD1}
				\operatorname{det} D(\lambda)&=(\lambda-\lambda_0)^{m_{a}(\lambda_0,T(\cdot))},\\
		\label{detD2}	\operatorname{det}_p D(\lambda)&=(\lambda-\lambda_0)^{m_{a}(\lambda_0,T(\cdot))}\exp\left({\sum_{k=1}^{p-1}k^{-1}(-1)^k\bigg(\sum_{j=1}^r \big((\lambda-\lambda_0)^{k_j}-1\big)\bigg)^k}\right),
		\end{align}
		and $\operatorname{det}_p G_1(\lambda)$, $\operatorname{det}_p G_2(\lambda)$ do not vanish in $\mathcal U$.
		\item If, in addition, the operator $T(\lambda)-I$ is in $\cB_p\left(\mathcal{H}\right)$ for all $\lambda\in\mathcal{U}$, then
			\begin{align*}
				\operatorname{det}_p T(\lambda)&= e^{\varphi(\lambda)} (\lambda-\lambda_0)^{{m_{a}(\lambda_0,T(\cdot))}}
		\end{align*}
		for some analytic function $\varphi$ on $\mathcal U$, hence
		\begin{align*}
				m(\lambda_0; \operatorname{det}_p T)&=m_{a}\big(\lambda_0,T(\cdot)\big),
		\end{align*}
		where $m(\lambda_0; \cdot)$ is the multiplicity function from \eqref{m:def}.
	\end{enumerate}
\end{prop}

\begin{proof}
	 (1)  \,  The operator $T(\lambda)$ is invertible in some punctured neighborhood of $\lambda_0$ by Hypothesis \ref{mainT}, so we can assume it is invertible in $\mathcal{U}\setminus\{\lambda_0\}$, and hence \eqref{eq1} holds by Proposition \ref{invertT}. Therefore, \eqref{indices} yields 
	 $D(\lambda)-I=\sum_{j=1}^r\big((\lambda-\lambda_0)^{k_j}-1\big)P_j$, with mutually disjoint rank one projections $P_j$. Thus the spectrum of the operator $D(\lambda)-I$ consists of $\mu_0=0$ and the eigenvalues $\mu_j=(\lambda-\lambda_0)^{k_j}-1$ for $1\le j\le r$. It follows that 
	  	\begin{align*}
			\operatorname{det} D(\lambda)&=\prod_{j=0}^r(1+\mu_j)=(\lambda-\lambda_0)^{\sum_{j=1}^rk_j}=(\lambda-\lambda_0)^{m_{a}(\lambda_0,T(\cdot))},
	\end{align*}
where in the last equality we applied  Lemma \ref{Fr}.

Next, using \eqref{detD1}, formula \eqref{detpdet} relating $\det_p(I+F)$ and $\det(I+F)$, and the formulas just obtained for the eigenvalues $\mu_j$ of $D(\lambda)-I$, we arrive at  \eqref{detD2}.
The statements for $\operatorname{det}_p G_1(\lambda)$ and $\operatorname{det}_p G_2(\lambda)$ follow from Theorem \ref{t1.1} since both $G_1(\lambda)$ and $G_2(\lambda)$ are invertible, and $\det_p G_1(\lambda)$ and $\det_p G_2(\lambda)$
 are well defined because $G_{1}(\lambda)-I$ and  $G_{2}(\lambda)-I$ are in $\cB_p(\cH)$ for all $\lambda\in\mathcal{U}$.
  
	(2) \, Since $T(\lambda)-G(\lambda)\in\cF(\cH)$ by Theorem \ref{t1.1} and $T(\lambda)-I\in\cB_p\left(\mathcal{H}\right)$ by assumption, we conclude that  $G(\lambda)-I\in\cB_p\left(\mathcal{H}\right)$ for $\lambda\in\mathcal{U}$. Using the fact that $G_1(\lambda)-I$, $G_2(\lambda)-I$ and $D(\lambda)-I$ are in $\cB_p(\cH)$ by Theorem \ref{t1.1},
	applying formula \eqref{prodform} twice and utilizing \eqref{detD1}, we have
	\begin{align*}
			\operatorname{det}_p T(\lambda)&=\operatorname{det}_p \big(G(\lambda)G_2(\lambda)D(\lambda)G_1(\lambda)\big)=\operatorname{det}(D(\lambda)) S(\lambda)=(\lambda-\lambda_0)^{m_{a}(\lambda_0,T(\cdot))} S(\lambda)
	\end{align*}
with some nowhere vanishing analytic function $S$, as required.	
\end{proof}

We now consider what happens to the {\em ratio} of $T(\cdot)$ from Theorem \ref{t1.1} and another analytic operator-valued function $\widehat{T}(\cdot)$, whose representation \eqref{eq5} from the theorem is 
\begin{equation}\label{eq5.bis}
	\widehat{T}(\lambda)=\widehat{G}(\lambda)\widehat{G}_2(\lambda)\widehat{D}(\lambda)\widehat{G}_1(\lambda), \quad \lambda\in\widehat{\mathcal{U}}.
	\end{equation}
	
\begin{prop}\label{propM}
Suppose $\widehat T\colon\widehat{\mathcal{V}}\rightarrow\mathcal{B}(\mathcal{H})$  satisfies Hypothesis \ref{mainT} and  let $\lambda_0\in\widehat{\mathcal{V}}$, so $\widehat T$ is given by \eqref{eq5.bis} in a neighborhood $\widehat{\mathcal U}$ of $\lambda_0$.
	\begin{enumerate}
\item The operator-valued function $\widehat D$ from \eqref{eq5.bis} and the projections $\widehat P_j$ from \eqref{indices} satisfy
		\begin{align}\label{invD}
		\widehat D(\lambda)^{-1}=\widehat P_0+(\lambda-\lambda_0)^{-\widehat k_1}\widehat P_1+\cdots+(\lambda-\lambda_0)^{-\widehat k_{\widehat r}}\widehat P_{\widehat r}
		\end{align}
		for $\lambda\in\mathcal{\widehat U}\setminus\{\lambda_0\}$,
		and the function $\widehat{T}(\cdot)^{-1}$ is completely meromorphic near $\lambda_0$.
Moreover, the functions $\widehat G(\cdot)^{-1}$,  $\widehat G_1(\cdot)^{-1}$,  $\widehat G_2(\cdot)^{-1}$ from \eqref{eq5.bis} are analytic with
		invertible values, and the operators $\widehat G_1(\lambda)^{-1}-I$, $ \widehat G_2(\lambda)^{-1}-I$, respectively $\widehat D(\lambda)^{-1}-I$, are all of finite rank and therefore belong to $\mathcal{B}_p\left(\mathcal{H}\right)$ for every $p$ and all $\lambda\in\mathcal{\widehat U}$,  respectively $\lambda\in\mathcal{\widehat U}\setminus\{\lambda_0\}$. Furthermore, for all 
		$\lambda\in\mathcal{\widehat U}\setminus\{\lambda_0\}$ we have
		\begin{align}\label{detDinv1}
		\operatorname{det} \widehat D(\lambda)^{-1}&=(\lambda-\lambda_0)^{-m_{a}(\lambda_0,\widehat T(\cdot))},\\
		\label{detDinv2}
		\operatorname{det}_p \widehat D(\lambda)^{-1}&=(\lambda-\lambda_0)^{-m_{a}(\lambda_0,\widehat T(\cdot))}\exp\left(\sum_{k=1}^{p-1}k^{-1}(-1)^k\bigg(\sum_{j=1}^{\widehat{r}}\big((\lambda-\lambda_0)^{-\widehat{k}_j}-1\big)\bigg)^k\right),\end{align}
		and $\operatorname{det}_p \big(\widehat{G}_1(\lambda)^{-1}\big)$, $\operatorname{det}_p \big(\widehat{G}_2(\lambda)^{-1}\big)$ do not vanish in $\widehat{\mathcal U}$.
		
		\item In addition, let $T \colon \mathcal V \to \cB(\cH)$ satisfy Hypothesis \ref{mainT} and  let $\lambda_0\in\mathcal{V}\cap\widehat{\mathcal{V}}$. If  $\widehat T(\lambda)^{-1}T(\lambda)-I$ is in $\cB_p\left(\mathcal{H}\right)$ for $\lambda$ in a punctured neighborhood of $\lambda_0$, then the function $\lambda\mapsto\widehat T(\lambda)^{-1}T(\lambda)$ is completely meromorphic in a neighborhood of $\lambda_0$, and there exists a meromorphic function $\varphi$ in a neighbourhood of $\lambda_0$ such that 
\begin{equation}\label{thatt}
{\det}_p\big(\widehat T(\lambda)^{-1}T(\lambda)\big) = e^{\varphi(\lambda)}(\lambda-\lambda_0)^{m_{a}(\lambda_0,T(\cdot))-m_{a}(\lambda_0,\widehat T(\cdot))} ,
\end{equation}
and hence
		\begin{align*}
			m\big(\lambda_0;\operatorname{det}_p \big(\widehat T(\cdot)^{-1}T(\cdot)\big)\big)= m_{a}\big(\lambda_0,T(\cdot)\big)-m_{a}\big(\lambda_0,\widehat T(\cdot)\big),
		\end{align*}
		where $m(\lambda_0; \cdot)$ is the multiplicity function from \eqref{m:def}.
	\end{enumerate}	
\end{prop}
\begin{proof}
	
	(1)\, By Theorem \ref{t1.1}, $\widehat T(\cdot)$ is equivalent at $\lambda_0$ to the analytic operator-valued
	function \begin{equation}\label{indices.bis}
	\widehat D(\lambda)=\widehat P_0+\sum_{j=1}^{\widehat r}(\lambda-\lambda_0)^{\widehat k_j}\widehat P_j
	\end{equation}
	for $\lambda\in\mathcal{\widehat U}$. That is, there exist analytic and
	invertible operators $\widehat G(\lambda), \widehat G_1(\lambda)$ and $\widehat G_2(\lambda)$  so that \eqref{eq5.bis} holds, 
where the operators $\widehat G_1(\lambda)-I, \widehat D(\lambda)-I, \widehat G_2(\lambda)-I$ are of finite rank and $\widehat G(\lambda)=\widehat T(\lambda)+\widehat F$ for some $\widehat F\in\cF(\cH)$.

By Hypothesis \ref{mainT}, $\widehat T(\lambda)$ is invertible for $\lambda$ in a punctured neighborhood of $\lambda_0$, so the same is true of $\widehat D(\lambda)$ and the identity \eqref{eq1} holds for $\widehat{P}_j$ by Proposition~\ref{invertT}.
Inverting \eqref{indices.bis}, we arrive at  \eqref{invD}. Since the coefficients of the singular terms in \eqref{invD} are of rank one, the function
 \begin{equation}\label{invN}
	\widehat T(\lambda)^{-1}= \widehat G_1(\lambda)^{-1}  \widehat D(\lambda)^{-1} \widehat G_2(\lambda)^{-1} \widehat G(\lambda)^{-1}
\end{equation}
is completely meromorphic.
All required assertions regarding the $\cB_p$ and finite rank properties follow from Remark \ref{rem-invcl}, while the determinant calculation for $\widehat D(\cdot)^{-1}$ is similar to that in the proof of Proposition \ref{detcalc}.

	(2)\,  Let $\lambda_0\in\mathcal{V}\cap\widehat{\mathcal{V}}$ and let $\mathcal{ U}$ and $\mathcal{\widehat U}$ be the neighborhoods of $\lambda_0$ as in Theorem \ref{t1.1}.  	
	For the rest of the proof we assume that $\lambda\in(\mathcal{\widehat U}\cap\mathcal{U})\setminus\{\lambda_0\}$ and suppress $\lambda$-dependence in operator-valued functions. Note that $\widehat{T}^{-1}T$ is completely meromorphic because $\widehat{T}^{-1}$ is and $T$ is analytic. Using \eqref{eq5} and \eqref{eq5.bis} we can write
		\begin{equation}\label{TT1}
	\widehat T^{-1}T=S_1\widehat D^{-1}S_2DS_3,
	\end{equation}
where $S_1, S_2, S_3$ are analytic with invertible values. Applying \eqref{DSid} twice, we can rewrite \eqref{TT1} as 
\begin{equation}\label{TT2}
	\widehat T^{-1}T=S_1S_2S_3\widetilde{D}_1\widetilde{D}_2,
\end{equation}
where $\widetilde{D}_1=(S_2S_3)^{-1}\widehat D^{-1}(S_2S_3)$ and $\widetilde{D}_2=S_3^{-1}DS_3$.
We observe that all five operators $\widehat D^{-1}-I$, $D-I$, $\widetilde{D}_1-I$, $\widetilde{D}_2-I$ and $\widetilde{D}_1\widetilde{D}_2-I$ are of finite rank by Remark \ref{rem-invcl}, and that $\det \widetilde{D}_1 = \det(\widehat D^{-1})$ and  $\det D=\det \widetilde{D}_2$. We now apply Lemma \ref{lem:lem5.3} with $I+B=\widehat T^{-1}T$,
$S=S_1S_2S_3$ and $I+F=\widetilde{D}_1\widetilde{D}_2$ to obtain the desired result \eqref{thatt}, 
\begin{align*}
	 {\det}_p(\widehat T^{-1}T) &=e^{\varphi(\lambda)} \det(\widetilde{D}_1)\det(\widetilde{D}_2)\\&=e^{\varphi(\lambda)}
\det(\widehat D^{-1})\det(D) = e^{\varphi(\lambda)} (\lambda-\lambda_0)^{m_{a}(\lambda_0,T(\cdot))-m_{a}(\lambda_0,\widehat T(\cdot))},
\end{align*} 
where in the last equality we used \eqref{detD1} and \eqref{detDinv1}.
\end{proof}

We now return to the elliptic setting described in the Introduction and in Section 4. 
We fix an operator $\Theta=\Theta_2$, let $\lambda_0\in\bbC$, and use Lemma \ref{T0} to choose $\Theta_1$ such that $\lambda_0\in\rho(\cL^{\Theta_1})$. We then consider the Robin-to-Dirichlet map $N_{\Theta_1}$ and Robin-to-Robin map $R_{1,2}$ defined in Section \ref{sec:Robin}.

\begin{prop}\label{multLN} Assume that $\Theta_1$ and $\Theta_2$ satisfy Hypothesis \ref{hypT}, and let $\lambda_0\in\rho(\cL^{\Theta_1})$.
\begin{enumerate}
	\item $\R N_{\Theta_1}\colon\rho(\cL^{\Theta_1})\rightarrow\cB\big(\Hm\big)$ and $R_{1,2}\colon\rho(\cL^{\Theta_1})\rightarrow\cB\big(\Hm\big)$ satisfy Hypothesis \ref{mainT}, and so are equivalent at $\lambda_0$ to analytic operator-valued functions of the form \eqref{indices}.

\item The projections $P_j$ from \eqref{indices} corresponding to the operator family $\R N_{\Theta_1}$ satisfy \eqref{eq1}, with $r=0$ for $\lambda_0 \in \rho(\cL^D)$ and $r>0$ for $\lambda_0 \in \sigma(\cL^D)$.

\item The projections $P_j$ from \eqref{indices} corresponding to the operator family $R_{1,2}$ satisfy \eqref{eq1}, with $r=0$ for $\lambda_0 \in \rho(\cL^{\Theta_2})$ and $r>0$ for $\lambda_0 \in \sigma(\cL^{\Theta_2})$.
\end{enumerate}
\end{prop}

Analogous results hold for the differential expression $\widehat L$. 

\begin{proof}
	(1) \, This follows directly from Lemma \ref{LNFred} and Theorem \ref{t1.1}.
	
	(2) \, If $\lambda \in \rho(\cL^D)$, then Lemma \ref{LNFred} implies that the operator $\R N_{\Theta_1}(\lambda)$ is invertible in a neighborhood of $\lambda_0$, including $\lambda_0$, and so by Proposition \ref{invertT}, we have $r=0$ and $P_0=I$. That is, \eqref{eq1} holds with $r=0$. On the other hand, if $\lambda \in \sigma(\cL^D)$, then by Theorem \ref{multR} (see also Remark \ref{rNFred1}) we have $m_a\big(\lambda_0, \R N_{\Theta_1}(\cdot)\big) = m_a(\lambda_0,\cL^D) > 0$. Thus by Lemma \ref{LNFred} the operator $\R N_{\Theta_1}(\lambda)$ is invertible for $\lambda$ in a punctured neighborhood of $\lambda_0$, but not at $\lambda_0$, so by Proposition \ref{invertT} the identity \eqref{eq1} holds and $r>0$. 
	
	The proof of (3) is analogous to item (2) so we omit it.
\end{proof}


We are ready to present the proof of Theorem~\ref{thm:WEvans}.
For all $\lambda\in\rho(\widehat\cL^{D})\cap\rho(\cL^{\Theta})$ we recall the definition
\begin{equation}
	E(\lambda) = \widehat M_{\widehat\Theta} (\lambda) N_{\Theta}(\lambda) \in \cB\big(\Hm\big),
\end{equation}
where we assume that ${\widehat\Theta}$ and ${\Theta}$ satisfy Hypotheses~\ref{hypT} (this is implied by the assumptions in Theorem \ref{thm:WEvans}).  In the course of proof we will also show that the function $E$ is completely meromorphic.

\begin{proof}[Proof of Theorem \ref{thm:WEvans}]
For convenience, we will use notations ${\widehat\Theta_2}={\widehat\Theta}$
and $\Theta_2=\Theta$. 
It suffices to prove the result locally. Fix $\lambda_0\in  \bbC$. For all $\lambda$ in a punctured neighborhood of $\lambda_0$ we have $\lambda\in\rho(\widehat\cL^{D})\cap\rho(\cL^{\Theta_2})$ 
and thus $E(\lambda)$ is well-defined.  By Proposition \ref{propBp} we know that  $E(\lambda)-I\in\cB_p\left(\Hm\right)$ 
for any $p > 2(n-1)$. 
In particular, $\operatorname{det}_p E(\cdot)=\operatorname{det}_p \big(I+E(\cdot)-I\big)$  is defined in a punctured neighborhood of $\lambda_0$.

We will now factorize $E$ as a product of four operator pencils.  To this end,	
 we use Lemma \ref{T0} to pick $\Theta_1$ and $\widehat\Theta_1$ such that $\lambda_0\in\rho(\cL^{\Theta_1}) \cap \rho(\widehat\cL^{\widehat\Theta_1})$. By Proposition \ref{propRtoR}, $E(\lambda)$ can be written
	\begin{equation}\label{maindecom}
	E(\lambda) = \widehat M_{\widehat\Theta_2} (\lambda) N_{\Theta_2}(\lambda)=\widehat R_{{\widehat\Theta_1,\widehat\Theta_2}}(\lambda)\big(\R \widehat N_{{{\widehat\Theta_1}}}(\lambda)\big)^{-1}\R N_{\Theta_1}(\lambda)\big(R_{{\Theta_1,\Theta_2}}(\lambda)\big)^{-1}.
	\end{equation}
	By Proposition \ref{multLN} the pencils $\widehat R_{{\widehat\Theta_1,\widehat\Theta_2}}$, $\R \widehat N_{{{\widehat\Theta_1}}}$,  $\R N_{\Theta_1}$ and  $R_{{\Theta_1,\Theta_2}}$ all satisfy Hypothesis \ref{mainT} in some neighborhood of $\lambda_0$. Thus, Theorem \ref{t1.1} and Propositions \ref{detcalc} and \ref{propM} apply. In particular, the pencils $\widehat R_{{\widehat\Theta_1,\widehat\Theta_2}}$ and $\R N_{\Theta_1}$ can be represented as in \eqref{eq5}, while the pencils $\big(\R \widehat N_{{{\widehat\Theta_1}}}\big)^{-1}$ and $\big(R_{{\Theta_1,\Theta_2}}\big)^{-1}$ can be represented as in \eqref{invN}.
By \eqref{indices} and \eqref{invD},  we conclude that $E(\cdot)$ is completely meromorphic in some neighborhood of $\lambda_0$, and hence over $\bbC$, as the coefficients of singular terms in \eqref{invD} have rank one while all other factors are analytic. 

It remains to prove formulas  \eqref{detorder2} and \eqref{detorder}. To simplify notation we rewrite \eqref{maindecom}
as 
\begin{equation}\label{maindecom2}
E(\lambda)=S_1(\lambda)D_1(\lambda)S_2(\lambda)D_2(\lambda)S_3(\lambda)D_3(\lambda)S_4(\lambda)D_4(\lambda)S_5(\lambda), \quad \lambda\in\mathcal{U}\setminus\{\lambda_0\},
\end{equation}
where $\mathcal{U}$ is the intersection of the neighborhoods of $\lambda_0$ where the conclusion of Theorem \ref{t1.1} holds for each of the four pencils in \eqref{maindecom}. From now on we will suppress $\lambda$ in the notation. In \eqref{maindecom2} the pencils $D_1$ and $D_3$ are as in \eqref{indices} and correspond to $\widehat R_{{\widehat\Theta_1,\widehat\Theta_2}}$ and $\R N_{\Theta_1}$, respectively, while the pencils $D_2$ and $D_4$ 
are as in \eqref{invD} and correspond to $\big(\R \widehat N_{{{\widehat\Theta_1}}}\big)^{-1}$ and $\big(R_{{\Theta_1,\Theta_2}}\big)^{-1}$, respectively. All five pencils $S_i$ in \eqref{maindecom2} are analytic in $\lambda$ and their  values are invertible operators, since they are obtained as products of various $G$-terms from the representations \eqref{eq5} and \eqref{invN}.

We claim that that $S_j-I$, $1\le j\le5$, and $D_i-I$, $1\le i\le4$, are $\cB_p$. 
For $D_i$ this follows from two formulas of type \eqref{indices} and two formulas of type 
\eqref{invD}. For $S_1$ we temporarily denote $T=\widehat{R}_{\widehat{\Theta}_1,\widehat{\Theta}_2}$ so that $S_1=GG_2$ by \eqref{eq5}. Since $G-T$ and $G_2-I$ are of finite rank by Theorem \ref{t1.1}, and $T-I$ is of class $\cB_p$ by Proposition \ref{RBp}, we may use Remark \ref{rem-invcl} to conclude that $S_1-I$ is of $\cB_p$ class because $G-I$ is. An analogous argument with $\widehat{T}=R_{\Theta_1,\Theta_2}$ works for $S_5$. To deal with $S_2$ we temporarily let $T=\widehat{R}_{\widehat{\Theta}_1,\widehat{\Theta}_2}$ and $\widehat{T}=\R\widehat{N}_{\widehat{\Theta}_1}$. Using \eqref{eq5} and \eqref{eq5.bis} we have $S_2=G_1\widehat{G}^{-1}_1$ and so $S_2-I$ is of $\cB_p$-class by Remark \ref{rem-invcl} because both $G_1-I$ and $\widehat{G}_1-I$ are by Theorem \ref{t1.1}. The argument for $S_4$ is analogous. It remains to deal with $S_3=\widehat{G}^{-1}_2\widehat{G}^{-1}GG_2$ with the $G$-terms from \eqref{eq5} and \eqref{eq5.bis}, where this time we denote $T=\R N_{\Theta_1}$ and $\widehat{T}=\R \widehat{N}_{\widehat{\Theta}_1}$. The operator
\[
	\widehat{T}^{-1}T-I = \big(\widehat{N}_{\widehat{\Theta}_1}\big)^{-1}N_{\Theta_1}-I
	= \big(\widehat{N}_{\widehat{\Theta}_1}\big)^{-1}\big(N_{\Theta_1}-\widehat{N}_{\widehat{\Theta}_1}\big)
\]
is $\cB_p$ because $N_{\Theta_1}-\widehat{N}_{\widehat{\Theta}_1}$ is by Proposition \ref{propBp}(2). Since $\widehat{T}^{-1}T=\widehat{G}^{-1}_1\widehat{D}^{-1}S_3DG_1$, using Theorem \ref{t1.1} and Remark \ref{rem-invcl} we conclude that $S_3-I$ is $\cB_p$, thus finishing the proof of the claim.

We now apply formula \eqref{DSid} to \eqref{maindecom2} four times to obtain
\[
	E=\big(S_1S_2S_3S_4S_5\big)\widetilde{D}_1\widetilde{D}_2\widetilde{D}_3\widetilde{D}_4,
\]
where we have defined
\begin{align*}
\widetilde{D}_1&=\big(S_2S_3S_4S_5\big)^{-1}D_1\big(S_2S_3S_4S_5\big), \qquad
\widetilde{D}_2=\big(S_3S_4S_5\big)^{-1}D_2\big(S_3S_4S_5\big)),\\
\widetilde{D}_3&=\big(S_4S_5\big)^{-1}D_3\big(S_4S_5\big),\qquad
\widetilde{D}_4=(S_5)^{-1}D_4S_5.
\end{align*}
Remark \ref{rem-invcl} tells us that $F:=\widetilde{D}_1\widetilde{D}_2\widetilde{D}_3\widetilde{D}_4-I$ is of finite rank. Moreover, $F$ is completely meromorphic and
\begin{equation}\label{prfDD}
\det(I+F)=\det\big(\widetilde{D}_1\widetilde{D}_2\widetilde{D}_3\widetilde{D}_4\big)=\prod_{i=1}^4\det D_i
\end{equation} as $\widetilde{D_i}$ and $D_i$ are similar and thus have equal determinants. The function $B=E-I$ is completely meromorphic because $E$ is. We can now apply \eqref{eqdetsphi} from Lemma \ref{lem:lem5.3} with $S=S_1S_2S_3S_4S_5$, because $S-I$ is of $\cB_p$-class since each $S_j-I$ is by the claim proved above. This and \eqref{prfDD} yield 
\begin{align*}
{\det}_p E &= e^\varphi \det(I+F)=e^\varphi \prod_{i=1}^4\det(D_i)\\
&=e^\varphi(\lambda-\lambda_0)^{m_a(\lambda_0, \widehat R_{{\widehat\Theta_1\widehat\Theta_2}}(\cdot)) - m_a(\lambda_0, \R \widehat N_{{{\widehat\Theta_1}}}(\cdot))+m_a(\lambda_0, \R N_{\Theta_1}(\cdot))-m_a(\lambda_0, R_{{\Theta_1\Theta_2}}(\cdot))},
\end{align*}
where for $\det D_i(\cdot)$ we used formula \eqref{detD1} when $i=1, 3$ and formula \eqref{detDinv1} when $i=2,4$. Now Theorem \ref{multR} implies the desired formulas \eqref{detorder2} and \eqref{detorder}, recalling that $\Theta=\Theta_2$ while  $\widehat\Theta=\widehat\Theta_2$.
\end{proof}

\begin{rem}\label{MNNM2} We now address the issue raised in Remark \ref{MNNM0} regarding switching the order of factors in $E$. As mentioned in Remark \ref{MNNM}, both inclusions 
\[
	\widehat M_{\widehat\Theta}(\lambda) N_\Theta(\lambda)-I \in \cB_p\big(H^{-1/2}(\pO)\big), \qquad N_\Theta(\lambda)\widehat M_{\widehat\Theta}(\lambda) -I \in \cB_p\big(H^{1/2}(\pO)\big)
\]
hold provided $\lambda\in\rho(\widehat{\cL}^D)\cap\rho({\cL}^\Theta)$. Therefore, both determinants
\[
	{\det}_p\big(\widehat M_{\widehat\Theta}(\lambda) N_\Theta(\lambda)\big), \qquad
{\det}_p\big(N_\Theta(\lambda)\widehat M_{\widehat\Theta}(\lambda)\big)
\]
are well-defined. It is easy to see that the two determinants are equal. Indeed, temporarily assuming that $\lambda\in\rho(\widehat{\cL}^{\widehat\Theta})$ also holds, so that $\widehat{N}_{\widehat{\Theta}}(\lambda)=\widehat M_{\widehat\Theta}(\lambda)^{-1}$ is bounded, and letting $A_1=\widehat M_{\widehat\Theta}(\lambda)$
and $A_2= N_\Theta(\lambda) - \widehat M_{\widehat\Theta}(\lambda)^{-1}$, we see that $\widehat M_{\widehat\Theta}(\lambda) N_\Theta(\lambda)-I=A_1A_2$ while $ N_\Theta(\lambda)\widehat M_{\widehat\Theta}(\lambda) -I=A_2A_1$, and thus formula \eqref{detprod} yields the result for all $\lambda\in\rho(\widehat{\cL}^D)\cap\rho({\cL}^\Theta)$ by continuity.
\end{rem}


\section{The complex Souriau map}
\label{sec:Sou}

In Section~\ref{sec:SA} we will explore the connection between our generalized Evans function and the Maslov index, assuming $L$ and $\widehat L$ are symmetric. This amounts to a detailed study of the so-called \emph{Souriau map} $W$, which was defined in \eqref{Souriaudef1}. Many relevant properties of $W$ are also satisfied by a more general object, which we call the \emph{complex Souriau map} and denote by $\gW$. The definition of $\gW$ does not require $L$ to be symmetric, so it generalizes $W$ to many other situations of interest. We thus devote this section to a detailed study of $\gW$. 

\subsection{Definition and continuation}
Assume that $L$ satisfies Hypothesis~\ref{hyp:L} and let $\lambda \in \rho(\cL^N)$, so that the Neumann-to-Dirichlet map $N(\lambda)$ is defined. Letting $\cL^{i\R}$ denote the realization of $L$ with the Robin boundary condition $\gaN^L u + i \R \gaD u = 0$, we will see in Theorem~\ref{thm:cont} that $N(\lambda)\R - i$ is invertible if and only if $\lambda \in \rho(\cL^{i\R})$. Therefore, the map
\begin{equation}
\label{genWdef}
	\gW_0(\lambda) := \big(N(\lambda)\R + i \big)\big(N(\lambda)\R - i\big)^{-1} \in \cB\big(\Hp\big)
\end{equation}
is analytic on $\rho(\cL^N) \cap \rho(\cL^{i\R})$. We claim that points in $\rho(\cL^N) \setminus \rho(\cL^{i\R})$ are removable singularities, so this extends to an analytic function on $\rho(\cL^{i\R})$.
We recall notation $N_\Theta$ and $R_{\Theta_1,\Theta_2}$ for the Robin-to-Dirichlet and Robin-to-Robin maps from Section~\ref{sec:Robin}.

\begin{theorem}
\label{thm:cont}
The function $\gW_0$ defined in  \eqref{genWdef} is given by $\gW_0(\lambda)=2iN_{i\R}(\lambda)\R-I_{\Hp}$ and hence can be analytically continued to $\rho(\cL^{i\R})$. For any $\Theta$ satisfying Hypothesis~\ref{hypT}, the continuation, which we denote by $\gW$, is given by
\begin{align}
\label{eq:WJmu'}
	\gW(\lambda) & = -\R^{-1} R_{\Theta,-i\R}(\lambda) R_{i\R,\Theta}(\lambda) \R
	\\ & = 2iN_\Theta(\lambda)R_{i\R,\Theta}(\lambda)\R-I_{\Hp}\label{eq:WJmu'2}
\end{align}
for all $\lambda \in \rho(\cL^{i\R}) \cap \rho(\cL^\Theta)$.
\end{theorem}

\begin{proof}
Proposition \ref{propRtoR}(i) yields
\begin{equation}\label{eq2.10}
i\R^{-1}R_{0,\pm i\R}(\lambda)\R=i\R^{-1}\big(I\pm i\R N(\lambda)\big)\R=i\mp N(\lambda)\R
\end{equation}
for $\lambda\in\rho(\cL^N)$, and so $i-N(\lambda)\R$ is invertible if and only if $R_{0, i\R}(\lambda)$ is, which is the case if and only if $\lambda\in\rho(\cL^N)\cap\rho(\cL^{i\R})$. This shows that 
$\gW_0$ is well defined and analytic on $\rho(\cL^N)\cap\rho(\cL^{i\R})$. 

For $\lambda\in\rho(\cL^N)\cap\rho(\cL^{i\R})$ equation \eqref{eq2.10} yields
\begin{equation}\label{eq2.11}
	\gW_0(\lambda) = \big(N(\lambda)\R + i \big)\big(N(\lambda)\R - i\big)^{-1} =
	-\R^{-1}R_{0,-i\R}(\lambda)R_{i\R,0}(\lambda)\R.
\end{equation}
By Proposition \ref{propRtoR}(v) with $\Theta_1=0$ and $\Theta_2=i\R$ we infer that
\begin{equation}\label{eq2.13}
R_{0,-i\R}(\lambda)R_{i\R,0}(\lambda)=I-2i\R N_{i\R}(\lambda),
\end{equation}
and so
\begin{equation}\label{eq2.14}
	\gW_0(\lambda) = -\R^{-1}\big(I-2i\R N_{i\R}(\lambda)\big)\R=2iN_{i\R}(\lambda)\R-I
\end{equation}
as claimed. In particular, $\gW_0$ is analytic in $\rho(\cL^{i\R})$ by Lemma \ref{lemRtoD}(ii). We now fix any $\Theta$ satisfying Hypothesis \ref{hypT}. Applying Proposition \ref{propRtoR}(v) with $\Theta_1=\Theta$ and $\Theta_2=i\R$ yields
\[R_{\Theta, - i\R}(\lambda)R_{i\R, \Theta}(\lambda)=I-2i\R N_{i\R}(\lambda)=R_{0,-i\R}(\lambda)R_{i\R, 0}(\lambda),\]
where in the last equality we used \eqref{eq2.13}. This and \eqref{eq2.11} show  \eqref{eq:WJmu'}. By Proposition \ref{propRtoR}(iii) with $\Theta_1=i\R$ and $\Theta_2=\Theta$ we have $N_\Theta(\lambda)R_{i\R,\Theta}(\lambda)\R=N_{i\R}(\lambda)\R$, and  \eqref{eq:WJmu'2} follows from \eqref{eq2.14}.
\end{proof}

Given an auxiliary differential expression $\widehat L$ also satisfying Hypothesis~\ref{hyp:L}, we can define $\widehat\gW$ analogously on $\rho(\widehat\cL^{i\R})$.

\subsection{The determinant}
We next relate the operators $E(\lambda)$, $\gW(\lambda)$ and $\widehat \gW(\lambda)$, and the corresponding determinants.

\begin{theorem}
\label{thm:WEequal}
Assuming the hypotheses of Theorem~\ref{thm:WEvans} are satisfied and letting
\[
	\mathcal U := \rho(\cL^{i\R}) \cap \rho(\widehat\cL^{i\R}) \cap \rho(\cL^\Theta) \cap \rho(\widehat \cL^{\widehat\Theta}),
\]
we have
\begin{equation}
\label{eq:WEequiv}
	\big(I + \widehat \gW(\lambda)\big)^{-1} \big(I + \gW(\lambda)\big) = S_1(\lambda) E(\lambda) S_2(\lambda)
\end{equation}
for $\lambda \in \mathcal U \cap \rho(\widehat \cL^D)$, where $S_1(\lambda) \in \cB\big(\Hm,\Hp\big)$ and  $S_2(\lambda) \in \cB\big(\Hp,\Hm\big)$ are invertible and depend analytically on $\lambda \in \mathcal U$. Moreover, $\big(I + \widehat \gW(\lambda)\big)^{-1} \big(I + \gW(\lambda)\big)-I$ is in $\cB_p\big(H^{1/2}(\pO)\big)$ for any integer  $p>2(n-1)$, and for each $\lambda_0 \in \bbC$ there is a meromorphic function $\varphi:\mathcal{U}_0\to\bbC$ in a neighborhood $\mathcal{U}_0$ of $\lambda_0$ such that the determinant
\begin{equation}
\label{def:omega}
	\mathcal{W}(\lambda) := \operatorname{det}_p \big( \big(I + \widehat \gW(\lambda)\big)^{-1} (I + \gW(\lambda))\big)
\end{equation}
satisfies
\begin{equation}
\label{WEequal}
	\mathcal{W}(\lambda) = e^{\varphi(\lambda)} \cE(\lambda)
\end{equation}
for $\lambda \in \mathcal U_0 \setminus \{\lambda_0\}$ and hence
\begin{equation}
\label{WEequal2}
	m(\lambda_0;\mathcal{W}) = m(\lambda_0;\cE).
\end{equation}

\end{theorem}

The formula \eqref{eq:WEequiv} says that $\big(I + \widehat \gW(\lambda)\big)^{-1} \big(I + \gW(\lambda)\big)$ is equivalent to $E(\lambda)$ in the sense of Definition~\ref{def:equivalent}.

\begin{proof}
Applying  \eqref{eq:WJmu'2} for $\gW$ and $\widehat{\gW}$ yields
\begin{align}
\begin{split}\label{S1S2def}
	\big(I + \widehat\gW\big)^{-1}\big(I + \gW\big)
	&= \R^{-1} \widehat R_{\widehat\Theta,i\R} \widehat M_{\widehat\Theta} N_\Theta R_{i\R,\Theta} \R \\
	&= \underbrace{\R^{-1} \widehat R_{\widehat\Theta,i\R}}_{S_1} E \underbrace{R_{i\R,\Theta} \R}_{S_2},
\end{split}
\end{align}
which verifies \eqref{eq:WEequiv}.

To complete the proof we must relate $\det_p (S_1 E S_2)$ to $\det_p E$. This is complicated by the fact that $\det_p S_1$ and $\det_p S_2$ are not defined, so we cannot directly relate $\det_p (S_1 E S_2)$ to a product of three determinants. We claim, however, that
\begin{equation}
\label{S1ES2}
	\operatorname{det}_p (S_1 E S_2) = \operatorname{det}_p (S_2 S_1 E).
\end{equation}
This is useful because the determinant of $S_2S_1$ is well defined, as we will see below.

To prove \eqref{S1ES2} we use the identity \eqref{detprod} with 
with $A_1 = S_1 E - S_2^{-1}$ and $A_2 = S_2$, for which we easily compute $I + A_1A_2 = S_1 E S_2$ and $I + A_2A_1 = S_2 S_1 E$. Therefore, the assertion $\big(I + \widehat \gW(\lambda)\big)^{-1} \big(I + \gW(\lambda)\big)-I\in\cB_p(H^{1/2}(\pO))$ will be proved and 
\eqref{S1ES2} will follow from \eqref{detprod} once we verify that $A_1A_2$ and $A_2A_1$ are $\cB_p$. In turn, this and Remark \ref{rem-invcl} then imply that $S_2S_1-I$ is of class $\cB_p$ since $E-I$ is by Proposition \ref{propBp}(1).

Under the additional assumption that $\lambda \in \rho(\cL^D)$, so that $M_\Theta(\lambda)$ is defined, we can write
\begin{equation}
\label{AB}
	A_1A_2 = S_1 \widehat M_{\widehat\Theta} N_\Theta S_2 - I 
	= \big(S_1 \widehat M_{\widehat\Theta}  - S_2^{-1} M_{\Theta}\big) N_\Theta S_2
\end{equation}
and
\begin{equation}
\label{BA}
	A_2A_1 = S_2 S_1 \widehat M_{\widehat\Theta} N_\Theta - I 
	= S_2 \big(S_1 \widehat M_{\widehat\Theta}  - S_2^{-1} M_{\Theta}\big) N_\Theta.
\end{equation}
Next, we observe that
\begin{equation}
\label{Sdiff}
	S_1 \widehat M_{\widehat\Theta}  - S_2^{-1} M_{\Theta} = S_1 \widehat M_{\widehat\Theta}
	\big(N_\Theta - \widehat N_{\widehat\Theta} \big)M_{\Theta} + \big(S_1 - S_2^{-1}\big) M_\Theta.
\end{equation}
Recalling the definitions of $S_1$ and $S_2$ in \eqref{S1S2def} and using Proposition~\ref{propRtoR}, we have
\[
	S_1 = \R^{-1} \widehat R_{\widehat\Theta,i\R} = \R^{-1}\big(I + (i\R - \widehat\Theta)\widehat N_{\widehat\Theta}\big)
\]
and
\[
	S_2^{-1} = \R^{-1} R_{\Theta,i\R} = \R^{-1} \big(I + (i\R - \Theta) N_{\Theta}\big)
\]
so that
\[
	S_1 - S_2^{-1} = (i\R - \widehat\Theta)\big(\widehat N_{\widehat\Theta} - N_{\Theta}\big)
	+ (\Theta - \widehat\Theta)N_{\Theta}.
\]
It follows from Proposition \ref{propBp}(2) and our assumptions on $\Theta$ and $\widehat\Theta$ in Hypothesis~\ref{hyp:Bp} that $S_1 - S_2^{-1}$ is $\cB_p$ for any $p>2(n-1)$. Substituting this in \eqref{Sdiff}, we get that $S_1 \widehat M_{\widehat\Theta}  - S_2^{-1} M_{\Theta}$ is $\cB_p$, so it follows from \eqref{AB} and \eqref{BA} that $A_1A_2$ and $A_2A_1$ are $\cB_p$ for all $\lambda \in \mathcal U \cap \rho(\cL^D) \cap \rho(\widehat\cL^D)$. The continuity argument in the proof of Proposition~\ref{propBp} shows that this in fact holds for all $\lambda \in \mathcal U \cap \rho(\widehat\cL^D)$.

We can thus use \eqref{detprod} to conclude \eqref{S1ES2} 
for all $\lambda \in \mathcal U \cap \rho(\widehat\cL^D)$, as claimed above. To complete the proof we use \eqref{prodformMER} to get
\[
	\operatorname{det}_p (S_2 S_1 E) = e^\varphi \operatorname{det}_p(S_2 S_1)\operatorname{det}_p E,
\]
where $\varphi$ is a meromorphic function near $\lambda_0$ that can be computed explicitly in terms of $E$ and $S_2 S_1$. Using \eqref{eq:WEequiv} and \eqref{S1ES2} and changing  $\varphi$ by incorporating $\det_p(S_1S_2)\neq0$ into the exponential factor  yields \eqref{WEequal}.
\end{proof}

We next relate the winding of $\cE(\lambda)$ to the eigenvalues of $I + \gW(\lambda)$ and $I + \widehat \gW(\lambda)$; cf. Corollary~\ref{cor:argument}.

\begin{cor}
\label{cor:winding}
Let $K \subset \rho(\cL^{i\R}) \cap \rho(\widehat\cL^{i\R})$ be a compact set with a rectifiable boundary. If $\p K$ is disjoint from  $\sigma(\cL^D)  \cup \sigma(\widehat\cL^D) \cup \sigma(\cL^\Theta) \cup \sigma(\widehat\cL^{\widehat\Theta})$, then
\begin{equation}
\label{corWind}
	\frac{1}{2\pi}\int_{\p K} \frac{\cE'(\lambda)}{\cE(\lambda)} \,d\lambda = \sum_{\lambda \in K} m_a\big(\lambda,I+\gW(\cdot)\big) - m_a\big(\lambda,I + \widehat\gW(\cdot)\big).
\end{equation}
\end{cor}

\begin{proof}
The hypotheses on $K$ imply that $\cE$ has no zeros or singularities on $\p K$. Defining $\mathcal{W}$ as in \eqref{def:omega}, we conclude from \eqref{WEequal} that
\[
	\int_{\p K} \frac{\cE'(\lambda)}{\cE(\lambda)} \,d\lambda 
	= \sum_{\lambda \in K} m(\lambda;\mathcal{W}).
\]
To complete the proof we will evaluate the right-hand side using Proposition~\ref{propM}, so we need to verify its hypotheses.

Theorem~\ref{thm:cont} and Lemma~\ref{lemRtoR} imply that $\gW(\lambda)$ is analytic in a neighborhood of $K$. Lemma~\ref{LNFred} implies that $N_\Theta(\lambda)$ and $R_{i\R,\Theta}(\lambda)$ are Fredholm of index zero, so it follows from \eqref{eq:WJmu'2} that $I + \gW(\lambda)$ is also Fredholm of index zero. The same is true of $\widehat \gW(\lambda)$. Finally, we have from Theorem~\ref{thm:WEequal} that $\big(I + \widehat \gW\big)^{-1}(I + \gW) - I$ is of class $\cB_p$. We can thus use Proposition~\ref{propM}(2) to conclude that 
\begin{align*}
	m(\lambda_0;\mathcal{W})=m_a\big(\lambda_0,I+\gW(\cdot)\big) - m_a\big(\lambda_0,I + \widehat\gW(\cdot)\big)
\end{align*}
for any point $\lambda_0 \in K$, and the result follows.
\end{proof}

\section{The Maslov index}
\label{sec:SA}

In this final section we explain how our multi-dimensional Evans function is related to the Maslov index when $L$ and $\widehat L$ are symmetric. Our main goal is to prove the equality \eqref{eq:MasEvans1} directly from the definitions of $\cE$ and $W$, without relating either side to the eigenvalue counting functions for $\cL^D$ and $\widehat \cL^D$.

In Section~\ref{sec:Maslovpre} we review the definition of the Maslov index. For simplicity we only describe the spaces and constructions arising in our particular setting; a general survey of the Maslov index in infinite dimensions is given in \cite{F04}, and a shorter summary with an emphasis on applications to boundary value problems can be found in \cite[Appendix B]{CJM15}. In Section~\ref{sec:Wcomp} we explicitly compute the Souriau map and relate it to the operator $\gW$ defined in Section~\ref{sec:Sou}, thus obtaining its analytic continuation from Theorem~\ref{thm:cont}. Finally, in Section \ref{ssec:Maslov} we use the symmetry of $L$ to prove a spectral mapping property for the analytic continuation, and use this to obtain a crucial monotonicity property for the Maslov index.

\subsection{Preliminaries}
\label{sec:Maslovpre}

The Maslov index is only defined for self-adjoint boundary value problems. We therefore assume that $L$ is symmetric for the remainder of the section. This implies $\cL^D$ is self-adjoint, hence $\sigma(\cL^D) \subset \bbR$ and all eigenfunctions can be assumed to be real-valued.

Abusing notation slightly, we let $\R$ denote both the real and complex Riesz maps, $\HpR \to \HmR$ and $\Hp \to \Hm$, respectively, where we recall our standing convention that all function spaces are complex-valued unless explicitly stated otherwise. The real Riesz map is the restriction of the complex one, so there is no ambiguity in our notation. Similarly, if $\lambda$ is real and $L$ is symmetric, then $N(\lambda) \in \cB\big(\Hm,\Hp\big)$ 
restricts to an operator in $\cB\big(\HmR,\HpR\big)$ which we again denote by $N(\lambda)$. 

We next define the real Hilbert space
\begin{equation}
	\cH = \HpR \oplus \HmR,
\end{equation}
with the symplectic form
	\[\omega\big((f_1,g_1), (f_2,g_2)\big) = \llangle g_2,f_1\rrangle - \llangle g_1,f_2\rrangle= \left<J(f_1,g_1), (f_2,g_2)\right>_\cH,\] where $J \colon \cH \to \cH$ is the almost complex structure\footnote{Cf. \cite[Example 2.1]{F04}, which gives the correct $\omega$ but has a sign error in $J$.}
\begin{equation}
	J(f,g) = (-\R^{-1} g, \R f).
\end{equation}
Using this, we can give $\cH$ the structure of a complex vector space, which we denote $\cH_J$, by defining the scalar multiplication $(a + ib)v = av + i bJv$ for all $a,b \in \bbR$ and $v \in \cH$. Note that $\cH_J$ and $\cH$ are isomorphic as real vector spaces; $\cH_J$ is not the complexification of $\cH$.

For $\lambda \in \bbR$, consider the subspaces
\begin{equation}
	\cG(\lambda) = \big\{ (\gaD u, \gaN^L u) : u \in D^1_L(\Omega,\bbR), \ Lu = \lambda u \big\}
\end{equation}
and
\begin{equation}
	D = \big\{(0,g) : g \in \HmR \}.
\end{equation}
These are both Lagrangian subspaces of $\cH$, and $\cG(\lambda)$ depends continuously (in fact analytically) on $\lambda$; see \cite{CJLS16} and \cite{CJM15} for details. Letting $P_\bullet$ denote the corresponding $\cH$-orthogonal projections, we define the \emph{Souriau map}
\begin{equation}
\label{Souriaudef}
	W(\lambda) = - (I-2P_{\cG(\lambda)})(I - 2P_D).
\end{equation}
The fact that both $D$ and $\cG$ are Lagrangian implies $WJ = JW$, hence $W$ defines a complex linear operator on the space $\cH_J$. Moreover, it is unitary, and hence has spectrum only on the unit circle. Up to a minus sign, this is simply reflection about $D$ followed by reflection about $\cG(\lambda)$. It follows that $\cG(\lambda)$ intersects $D$ nontrivially if and only if $-1 \in \sigma(W(\lambda))$, and
\begin{equation}
\label{eq:mult}
	\dim_\bbC \ker \big(I + W(\lambda) \big) = \dim_\bbR \big(D \cap \cG(\lambda)\big) = \dim_\bbR \ker\big(\cL^D - \lambda \big).
\end{equation}
That is, the multiplicity of $-1$ as an eigenvalue of $W(\lambda)$ equals the multiplicity of $\lambda$ as a Dirichlet eigenvalue. (Since $\cL^D$ is selfadjoint, the geometric and algebraic multiplicity coincide, so we can use the term ``multiplicity" unambiguously.)

 Finally, the fact that $D$ and $\cG$ form a Fredholm pair implies that the spectrum of $W$ cannot accumulate at $-1$. The \emph{Maslov index of $\cG(\lambda)$ with respect to $D$} is then defined to be the spectral flow of $W(\lambda)$ through the point $-1$ on the unit circle, in a counterclockwise direction, as $\lambda$ varies. 

From \eqref{eq:mult} we see that this gives a signed count of the eigenvalues of $\cL^D$, where the sign depends on whether the corresponding eigenvalues of $W$ passes through $-1$ in the clockwise or counterclockwise direction. However, it is well known that this index satisfies a monotonicity property, in the sense that eigenvalues of $W(\lambda)$ always pass though $-1$ in the negative (clockwise) direction as $\lambda$ increases. As a result, for any real numbers $\lambda_1 < \lambda_2$ both contained in $\rho(\cL^D)$ we have
\begin{equation}
	\text{number of eigenvalues of } \cL^D \text{ in } (\lambda_1,\lambda_2) = -\sflow \left(W \big|_{\lambda_1}^{\lambda_2} , -1 \right),
\end{equation}
where the eigenvalues are counted with multiplicity. Below we will give an independent proof of this monotonicity, using complex analytic methods.

\subsection{Computing $W_J$}
\label{sec:Wcomp}
To understand the connection between the Evans function and the Maslov index, we must relate the Souriau map $W(\lambda)$ to the operator $E(\lambda)$ defined in \eqref{eq:Edef}. We do this by computing an equivalent version of $W(\lambda)$ which is defined on $\Hp$ instead of $\cH_J$. We first describe the relationship between these spaces.

\begin{lemma}
\label{lem:psi}
The map $\psi(f,g) = f + i \R^{-1} g$ gives a (complex) linear isomorphism from $\cH_J$ to $\Hp$.
\end{lemma}

\begin{proof}
The map is obviously real linear, so we just need to check multiplication by $i$. We have $i(f,g)  = (-\R^{-1} g, \R f)$ and hence
\[
	\psi\big(i(f,g) \big) = \psi\big(  (-\R^{-1} g, \R f) \big) = -\R^{-1} g + i f = i(f + i \R^{-1} g) = i \psi(f,g).
\]
This completes the proof.
\end{proof}

A real linear map $T \colon \cH \to \cH$ is complex linear if and only if it commutes with $J$, in which case it defines a map on $\cH_J$. Using the isomorphism $\psi$ from Lemma \ref{lem:psi}, we obtain a complex linear map $T_J := \psi T \psi^{-1}$ on $\Hp$.

\begin{lemma}
\label{lem:block}
The real linear operator
\[
	T = \begin{bmatrix} A & B \\ C & D \end{bmatrix} \colon \cH \longrightarrow \cH
\]
commutes with $J$ if and only if $C = - \R B \R$ and $D = \R A \R^{-1}$.
In this case, the corresponding complex linear operator $T_J$ on $\Hp$ is
\begin{equation}
	T_J = A - iB \R.
\end{equation}
\end{lemma}

\begin{proof}
The first claim follows from a direct computation, using the definition of $J$. Now assume this holds, so that
\begin{equation}
	T = \begin{bmatrix} A & B \\ - \R B \R & \R A \R^{-1} \end{bmatrix}.
\end{equation}
Let $x + iy \in \Hp$, so that $x,y \in \HpR$. From the definition of $\psi$ (in Lemma \ref{lem:psi}) have $\psi^{-1}(x + iy) = (x,\R y)$, so that
\[
	T\psi^{-1}(x + iy) = \begin{bmatrix} A & B \\ - \R B \R & \R A \R^{-1} \end{bmatrix} \begin{bmatrix} x \\ \R y \end{bmatrix} 
	= \begin{bmatrix} Ax + B\R y \\ - \R B \R x + \R Ay \end{bmatrix}
\]
and hence
\begin{align*}
	\psi T\psi^{-1}(x + iy) &= (Ax + B\R y) + i \R^{-1} (- \R B \R x + \R Ay) \\
	&= (Ax + B\R y) + i(-B \R x + Ay) = \big(A - iB \R \big)(x + iy),
\end{align*}
as was to be shown.
\end{proof}

We now apply this lemma to the Souriau map.

\begin{prop}
\label{prop:WN}
The map $W(\lambda)$ defined in \eqref{Souriaudef} commutes with $J$ for all real $\lambda$, and the corresponding map $W_J(\lambda)$ on $\Hp$ is given by
\begin{equation}
\label{WJdef}
	W_J(\lambda) = \big(N(\lambda)\R + i \big)\big(N(\lambda)\R - i\big)^{-1}
\end{equation}
for all $\lambda \in \rho(\cL^N) \cap \bbR$.
\end{prop}

In other words, $W_J(\lambda)$ is the Cayley transform of $N(\lambda)\R$.

\begin{proof}
Fixing $\lambda \in \rho(\cL^N) \cap \bbR$, we abbreviate $\cG = \cG(\lambda)$ and $N = N(\lambda)$ for convenience. From \cite[Proposition~4.11]{CJLS16} we know that the orthogonal projection onto $\cG$ is given by\footnote{The operators $N$ and $\R$ here correspond to $-M_s$ and $J_\bbR^{-1}$ in \cite{CJLS16}, cf.\ also Remark \ref{signconv}.}
\[
	P_\cG = \begin{bmatrix} (N\R)^2 & N \\ \R N \R & I \end{bmatrix} \begin{bmatrix} \big(I + (N\R)^2 \big)^{-1} & 0 \\ 0 &  \big(I + (\R N)^2 \big)^{-1} \end{bmatrix},
\]
from which we obtain
\begin{align*}
	2 P_\cG - I 
	&=  \begin{bmatrix} (N\R)^2 - I  & 2N \\ 2\R N \R & I - (\R N)^2 \end{bmatrix} \begin{bmatrix} \big(I + (N\R)^2 \big)^{-1} & 0 \\ 0 &  \big(I + (\R N)^2 \big)^{-1} \end{bmatrix}.
\end{align*}
For the projection onto $D$ we have
\[
	P_D = \begin{bmatrix} 0 & 0 \\ 0 & I \end{bmatrix}, \qquad I - 2 P_D = \begin{bmatrix} I & 0 \\ 0 & -I \end{bmatrix},
\]
and so
\[
	W = (2 P_\cG - I)(I - 2 P_D) = \begin{bmatrix} (N\R)^2 - I  & 2N \\ 2\R N \R & I - (\R N)^2 \end{bmatrix} \begin{bmatrix} \big(I + (N\R)^2 \big)^{-1} & 0 \\ 0 & -\big(I + (\R N)^2 \big)^{-1} \end{bmatrix}.
\]
This is of the form given in Lemma \ref{lem:block}, with
\[
	A = \big((N\R)^2 - I\big) \big(I + (N\R)^2 \big)^{-1}, \qquad B = -2 N \big(I + (\R N)^2 \big)^{-1}.
\]
Using the easily verified identity
\[
	\big(I + (\R N)^2 \big)^{-1} \R = \R \big(I + (N\R)^2 \big)^{-1}
\]
we compute
\begin{align*}
	A  - i B \R 
	&= \big( (N\R)^2 - I + 2i N \R \big) \big(I + (N\R)^2 \big)^{-1} \\
	&= (N\R + i)^2 \big((N\R - i)(N\R + i)\big)^{-1} = (N\R + i)(N\R - i)^{-1}
\end{align*}
as claimed.
\end{proof}

For $\lambda \in \rho(\cL^N) \cap \bbR$ the right-hand side of \eqref{WJdef} is precisely the operator $\gW(\lambda)$ defined in \eqref{genWdef}. Therefore, by Theorem~\ref{thm:cont}, $W_J$ can be continued to an analytic function on $\rho(\cL^{i\R})$, which we will continue to denote by $W_J$.  Given $\Theta$ satisfying Hypothesis~\ref{hypT}, we use \eqref{eq:WJmu'2} to obtain
\begin{equation}
\label{eq:WJmu''}
	I + W_J(\lambda) = 2i N_\Theta(\lambda) R_{i\R,\Theta}(\lambda)\R
\end{equation}
for all $\lambda \in \rho(\cL^{i\R}) \cap \rho(\cL^\Theta)$, 
To apply this result in practice, we therefore need to understand the resolvent sets of $\cL^\Theta$ and $\cL^{i\R}$.

\begin{lemma}
\label{resolventLTheta}
If $\Theta$ is non-real, then the resolvent set $\rho(\cL^\Theta)$ contains an open neighborhood of the real axis. In particular, $\rho(\cL^{i\R})$ contains an open neighborhood of the real axis.
\end{lemma}

\begin{proof}
Since the resolvent set is open, it suffices to prove that it contains the real axis. If $\cL^\Theta u = \lambda u$, then
\begin{equation}
\label{evalquad}
	\lambda \|u\|_{L^2(\Omega)}^2 = \Phi_\Theta[u] = \Phi[u] + \llangle\Theta\gaD u, \gaD u\rrangle.
\end{equation}
Since $L$ is symmetric, the quadratic form $\Phi[u]$ is real for any $u \in H^1(\Omega)$. Assuming $\lambda$ is real and taking the imaginary part of \eqref{evalquad}, we obtain $\Im \llangle\Theta\gaD u, \gaD u\rrangle = 0$ and hence $\gaD u = 0$. It follows from the unique continuation principle (see \eqref{UCP2}) that $u=0$, and so $\cL^\Theta$ has no real eigenvalues. To prove the second claim, we observe that $\Im \llangle i\R f,f \rrangle = \|f\|_{\Hp}^2$, so $i\R$ is non-real.
\end{proof}

\subsection{Connection to the Maslov index}
\label{ssec:Maslov}

We are now ready to prove Theorem \ref{thm:MorseEvans}, relating the Maslov index to the winding of the Evans function for a selfadjoint boundary value problem. Given real numbers $\lambda_1 < \lambda_2$ in $\rho(\cL^D) \cap \rho(\widehat\cL^D)$, Lemma~\ref{resolventLTheta} guarantees that $K = [\lambda_1,\lambda_2] \times [-\delta,\delta]$ is contained in $\rho(\cL^{i\R}) \cap \rho(\widehat\cL^{i\R}) \cap \rho(\cL^\Theta) \cap \rho(\widehat\cL^\Theta)$ as long as $\delta>0$ is sufficiently small.

As described in the introduction, we want to show that
\begin{equation}
\label{eq:MasEvans2}
	\frac{1}{2\pi i} \int_{\p K} \frac{\cE'(\lambda)}{\cE(\lambda)} \,d\lambda = 
	\sflow \left(\widehat W_J \big|_{\lambda_1}^{\lambda_2} , -1 \right)
	- \sflow \left(W_J \big|_{\lambda_1}^{\lambda_2} , -1 \right),
\end{equation}
where we have used the fact that $W$ and $W_J$ are unitarily equivalent and hence have the same spectral flow, and likewise for $\widehat W$ and $\widehat W_J$. In light of Corollary~\ref{cor:winding}, it is enough to show that
\begin{equation}
\label{eq:sfW}
	\sum_{\lambda \in K} m_a\big(\lambda,I + W_J(\cdot)\big) = -\sflow \left(W_J \big|_{\lambda_1}^{\lambda_2} , -1 \right),
\end{equation}
and similarly for $\widehat W_J$. 

We first deal with the left-hand side of \eqref{eq:sfW}, proving that the algebraic multiplicity of $\lambda$ as an eigenvalue of the nonlinear pencil $I + W_J(\cdot)$ coincides with the geometric multiplicity of $-1$ as an eigenvalue of the linear operator $W_J(\lambda)$.

\begin{prop}
\label{prop:Wmult}
If $L$ is symmetric, then
\begin{equation}
\label{W:ma}
	m_a\big(\lambda_0, I + W_J(\cdot)\big) = \dim \ker \big(I + W_J(\lambda_0)\big)
\end{equation}
for all $\lambda_0 \in \rho(\cL^{i\R})$. In particular, both quantities are zero if $\lambda_0$ is not real.
\end{prop}

\begin{proof}
Fix $\lambda_0 \in \rho(\cL^{i\R})$ and choose $\Theta_0$ such that $\lambda_0 \in \rho(\cL^{\Theta_0})$, by Lemma~\ref{T0}. It follows from \eqref{eq:WJmu''} and Theorem~\ref{multR}(2) that
\begin{align*}
	\text{$I + W_J(\lambda_0)$ is invertible } \ \Longleftrightarrow \ \text{$N_{\Theta_0}(\lambda_0)$ is invertible }
	 \ \Longleftrightarrow \ \lambda_0 \in \rho(\cL^D).
\end{align*}
Since $\cL^D$ is self-adjoint, we conclude that $I + W_J(\lambda_0)$ is invertible if $\lambda_0$ is not real, in which case both sides of \eqref{W:ma} vanish.

It remains to prove \eqref{W:ma} when $\lambda_0$ is real. In this case we can use \cite[Theorem~3.2]{MR3206692} (cf. Remark~\ref{rem:Rohleder}) to find a real number $\mu$ such that $\lambda_0 \in \rho(\cL^\Theta)$, where $\Theta = \mu \J$. This means $N_\Theta(\lambda)$ is defined for $\lambda$ in a neighborhood of $\lambda_0$, and \eqref{eq:WJmu''} says that $I+W_J(\lambda)$ is equivalent to $N_\Theta(\lambda)$, in the sense of Definition~\ref{def:equivalent}.  It follows that
\begin{equation}
\label{W:ma1}
	m_a\big(\lambda_0, I + W_J(\cdot)\big) = m_a\big(\lambda_0, N_\Theta(\cdot)\big)
\end{equation}
and
\begin{equation}
\label{W:ma2}
	\dim \ker \big(I + W_J(\lambda_0)\big) = \dim \ker N_\Theta(\lambda_0).
\end{equation}
To prove \eqref{W:ma}, we therefore need to show that
\begin{equation}
\label{W:ma3}
	m_a\big(\lambda_0, N_{\Theta}(\cdot)\big) = \dim \ker N_\Theta(\lambda_0).
\end{equation}
This is precisely the statement that every Jordan chain for the nonlinear pencil $N_\Theta(\cdot)$ has length one. The fact that $N_{\Theta}(\lambda_0)$ is self-adjoint for this particular choice of $\Theta$ and $\lambda_0\in\bbR$ is not enough to imply \eqref{W:ma3}, however, since for nonlinear pencils $m_a\big(\lambda_0, N_{\Theta}(\cdot)\big)$ does not necessarily coincide with the algebraic multiplicity of $0$ as an eigenvalue of the linear operator $N_{\Theta}(\lambda_0)$; see Remark~\ref{rem:mult_notation} and Example~\ref{ex1}.

We can deduce \eqref{W:ma3} immediately from Theorem~\ref{Nchain} and the fact that $\cL^D$ is self-adjoint, since this implies that every Jordan chain for $\cL^D$ has length one. For convenience we also give a more elementary proof that does not rely on the technical results of Section~\ref{sec:multiplicity}.

Suppose that $N_\Theta(\cdot)$ has a Jordan chain of length greater than one at $\lambda_0$, so there exist $f_0, f_1 \in \Hm$ satisfying the equations
\[
	N_\Theta(\lambda_0) f_0 = 0, \qquad N_\Theta'(\lambda_0) f_0 + N_\Theta(\lambda_0) f_1 = 0,
\]
with $f_0 \neq 0$. Since $\Theta = \mu \J$ with $\mu \in \bbR$, and $\lambda_0 \in \bbR$, the Robin-to-Dirichlet map $N_\Theta(\lambda_0)$ is self-adjoint, and we get
\begin{equation}
\label{J:contradiction}
	\llangle f_0, N_\Theta'(\lambda_0) f_0 \rrangle = -\llangle f_0, N_\Theta(\lambda_0) f_1 \rrangle
	= -\llangle f_1, N_\Theta(\lambda_0)f_0 \rrangle = 0.
\end{equation}
It follows from Lemma~\ref{lem:Nprime} that  $f_0 = 0$, a contradiction.
\end{proof}

\begin{lemma}
\label{lem:Nprime}
Suppose $L$ is symmetric, $\Theta$ satisfies Hypothesis~\ref{hypT} and $\lambda_0 \in \rho(\cL^\Theta) \cap \bbR$. If $f \in \Hm$ satisfies $N_\Theta(\lambda_0)f = 0$ and $\llangle f, N_\Theta'(\lambda_0) f\rrangle = 0$, then $f = 0$.
\end{lemma}

\begin{proof}
Given such an $f$, let $u(\lambda)$ denote the unique solution to
\[
	Lu = \lambda u, \quad \gaN^L u + \Theta \gaD u = f,
\]
so that $N_\Theta(\lambda_0)f = \gaD u(\lambda_0) = 0$, and hence $\gaN^L u(\lambda_0) = f$. It follows from the proof of Lemma~\ref{lemRtoD}  that $\lambda \mapsto u(\lambda) \in H^1(\Omega)$ is analytic in a neighborhood of $\lambda_0$, so we have $N_\Theta'(\lambda_0)f = \gaD u'(\lambda_0)$, and hence
\begin{equation}
\label{fN'f}
	\llangle \gaN^L u(\lambda_0), \gaD u'(\lambda_0)\rrangle
	= \llangle f, N_\Theta'(\lambda_0) f \rrangle = 0.
\end{equation}

From Green's first identity \eqref{e14} we have
\begin{equation}
\label{GreenSA}
	\Phi\big(u(\lambda),v\big) = \lambda \left<u(\lambda),v\right> + \llangle \gaN^L u(\lambda),\gaD v \rrangle
\end{equation}
for all $v \in H^1(\Omega)$ and $\lambda$ in a neighborhood of $\lambda_0$. Differentiating \eqref{GreenSA}, evaluating at $\lambda_0$ and choosing $v = u(\lambda_0)$, we obtain
\begin{equation}
\label{GreenSA1}
	\Phi\big(u'(\lambda_0),u(\lambda_0) \big) = \| u(\lambda_0)\|^2 + \lambda_0 \left<u'(\lambda_0), u(\lambda_0)\right>,
\end{equation}
where we have used the fact that $\gaD u(\lambda_0) = 0$.
On the other hand, evaluating \eqref{GreenSA} at $\lambda_0$ and choosing $v = u'(\lambda_0)$ gives
\begin{equation}
\label{GreenSA2}
	\Phi\big(u(\lambda_0),u'(\lambda_0) \big) = \lambda_0 \left<u(\lambda_0), u'(\lambda_0) \right> + \llangle \gaN^L u(\lambda_0),\gaD u'(\lambda_0) \rrangle.
\end{equation}
Since $\lambda_0$ is real and $L$ is symmetric, \eqref{GreenSA1} and \eqref{GreenSA2} together yield
\[
	\| u(\lambda_0)\|^2 = \llangle \gaN^L u(\lambda_0),\gaD u'(\lambda_0) \rrangle.
\]
Combined with \eqref{fN'f}, this gives $u(\lambda_0) = 0$ and hence $f = \gaN^L u(\lambda_0) = 0$, as claimed.
\end{proof}

As a result of Proposition~\ref{prop:Wmult}, equation \eqref{eq:sfW} can be rewritten as
\begin{equation}
\label{eq:sfW2}
	\sum_{\lambda \in K} \dim \ker \big(I + W(\lambda)\big) = -\sflow \left(W_J \big|_{\lambda_1}^{\lambda_2} , -1 \right).
\end{equation}
The left-hand side counts, with multiplicity, the points $\lambda \in (\lambda_1,\lambda_2)$ for which $-1$ is an eigenvalue of $W_J(\lambda)$, while the spectral flow on the right-hand side depends on the \emph{direction} in which these eigenvalues passes through $-1$ as $\lambda$ increases. Therefore, \eqref{eq:sfW2} holds provided the  eigenvalues of $W_J(\lambda)$ only pass through $-1$ in the negative (clockwise) direction.

We give a sufficient condition for this to happen in terms of the location of the spectrum of $W_J(\lambda)$ relative to the unit circle.

\begin{theorem}\label{thm:abstractSF}
Let $X$ be a Hilbert space and suppose that  $\lambda\mapsto U(\lambda)$ is an analytic $\cB(X)$-valued function, defined on an open neighborhood $\mathcal U \subset \bbC$ containing the segment $[\lambda_1, \lambda_2] \subset \bbR$. If $I + U(\lambda_1)$ and $I + U(\lambda_2)$ are invertible and the following conditions hold for all $\lambda \in \mathcal U$,
\begin{enumerate}
	\item[(i)] $I + U(\lambda)$ is Fredholm,
	\item[(ii)] $U(\lambda)$ is unitary for $\lambda \in \bbR$,
	\item[(iii)] $\sigma(U(\lambda)) \subset \{ z : |z| > 1 \}$ for $\Im \lambda > 0$,
	\item[(iv)] $\sigma(U(\lambda)) \subset \{ z : |z| < 1 \}$ for $\Im \lambda < 0$,
\end{enumerate}
then
\begin{equation}
\label{eq:sfU}
	\operatorname{sf} \left(U \big|_{\lambda_1}^{\lambda_2} , -1 \right) =
	-\sum_{\lambda_1 < \lambda < \lambda_2} \dim\ker \big(I + U(\lambda)\big).
\end{equation}
\end{theorem}

Given this abstract result (which we prove later), the equality \eqref{eq:sfW2}, and hence Theorem \ref{thm:MorseEvans}, is an immediate consequence of the following.

\begin{prop}\label{prop8.8}
If $L$ is symmetric and $\lambda_1,\lambda_2 \in \rho(\cL^D)$, then $W_J(\lambda)$ satisfies the hypotheses of Theorem \ref{thm:abstractSF}.
\end{prop}

Indeed, as indicated in Remark \ref{signconv}, $N(\cdot)\R$ is a Nevanlinna function, and so $\Im\lambda>0$ yields $\Im\sigma(N(\lambda)\R)>0$; therefore, $\sigma(W_J(\lambda))$ is outside of the unite circle since $W_J(\lambda)$ is the Cayley transform of $N(\lambda)\R$. We provide, however, a direct proof of the proposition.

\begin{proof}
It was already observed in the proof of Corollary~\ref{cor:winding} that $W_J(\lambda)$ is analytic in a neighborhood of the real axis, $I + W_J(\lambda_1)$ and $I + W_J(\lambda_2)$ are invertible, and $I + W_J(\lambda)$ is Fredholm of index zero. Moreover, when $\lambda$ is real we have that $W_J(\lambda)$ is a unitary operator on $\Hp$, since $W(\lambda)$ is unitary on $\cH_J$.  
Therefore, it only remains to verify the spectral inclusions (iii) and (iv).

Let $z \in \sigma(W_J(\lambda))$ for some $\lambda$ with $\Im\lambda \neq 0$. This guarantees $\lambda \in \rho(\cL^N)$, so \eqref{WJdef} gives
$z = \frac{\nu + i}{\nu - i}$
for some $\nu \in \sigma(N(\lambda)\R)$, and hence $|z| > 1$ is equivalvent to $ \Im \nu > 0$,  while $|z| = 1$ is equivalent to  $\Im \nu = 0$ and 	$|z| < 1$ is equivalent to $ \Im \nu < 0$.

Let $f \in \Hp$ be an eigenfunction corresponding to $\nu \in \sigma(N(\lambda)\R)$. By definition we have $N(\lambda)\R f = \gaD u$, where $u \in H^1(\Omega)$ solves $Lu = \lambda u$ and $\gaN^L u = \R f$. The eigenvalue equation is $\gaD u = \nu f$, so it follows from \eqref{e14} that
\begin{align*}
	\Phi(u,u) &= \lambda \|u\|_{L^2(\Omega)}^2 +\llangle {\gaN^L}u,\gaD u\rrangle \\
&= \lambda \|u\|_{L^2(\Omega)}^2 + \llangle \R f,\nu f \rrangle = \lambda \|u\|_{L^2(\Omega)}^2 + \bar\nu \| f \|_{\Hp}^2.
\end{align*}
The symmetry of $L$ implies that $\Phi(u,u)$ is real, so we get
$(\Im\lambda) \|u\|_{L^2(\Omega)}^2 = (\Im\nu) \| f \|_{\Hp}^2$.
We conclude that $|z| > 1 \,\Leftrightarrow\, \Im \nu > 0 \,\Leftrightarrow\, \Im \lambda > 0$, and similarly when $|z| = 1$ and $|z| < 1$, which completes the proof.
\end{proof}

We conclude by proving the abstract result in Theorem \ref{thm:abstractSF}.

\begin{proof}[Proof of Theorem \ref{thm:abstractSF}]
The set of $\lambda$ such that $I + U(\lambda)$ is not invertible does not contain $\lambda_1$ or $\lambda_2$, and hence is discrete, by the analytic Fredholm theorem. Therefore, it suffices to prove the result when there is a single point $\lambda_* \in (\lambda_1, \lambda_2)$ for which $\dim \ker\big(I + U(\lambda_*)\big) = m > 0$.

Since $U(\lambda)$ is unitary for $\lambda$ in the set $(\lambda_* - \epsilon, \lambda_* + \epsilon)$, which has $\lambda_*$ as a limit point, \cite[Theorem II-1.10~(p.71)]{K76} gives the existence of analytic eigenvalue curves $\mu_1(\lambda), \ldots, \mu_m(\lambda)$ in $\sigma(U(\lambda))$, with $\mu_1(\lambda_*) = \cdots = \mu_m(\lambda_*) = -1$. The right-hand side of \eqref{eq:sfU} is equal to $-m$, so it suffices to show that as $\lambda$ increases from $\lambda_* - \epsilon$ to $\lambda_* + \epsilon$, the curves $\mu_k(\lambda)$ all pass through $-1$ in a negative (clockwise) direction, since this implies
\[
	\sflow \left(U \big|_{\lambda_* - \epsilon}^{\lambda_* + \epsilon}, -1\right) = -m.
\]
It is therefore enough to consider one of these curves, which we denote by $\mu(\lambda)$, for $|\lambda - \lambda_*| \leq \epsilon$.

Consider the curve $\lambda(t) = \lambda_* + \epsilon e^{it}$, $0 \leq t \leq 2\pi$, which parameterizes the circle $C_\epsilon(\lambda_*) = \{\lambda : |\lambda - \lambda_*| = \epsilon\}$. It follows that $t \mapsto \mu(\lambda(t))$ defines a positively oriented curve in the complex plane, with $|\mu(\lambda(0))| = |\mu(\lambda(\pi))| = 1$ because $\lambda(0) = \lambda_* + \epsilon$ and $\lambda(\pi) = \lambda_* - \epsilon$ are real. For $t \in (0,\pi)$, $\lambda(t)$ is in the upper half plane, so item (iii) of the hypotheses implies $|\mu(\lambda(t))| > 1$. Similarly, by (iv) we have $|\mu(\lambda(t))| < 1$ for $t \in (\pi,2\pi)$.

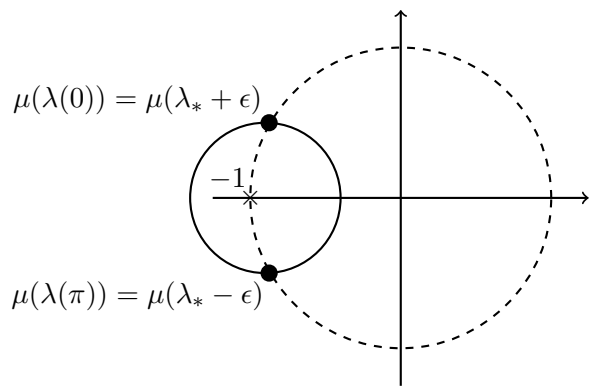
\begin{figure}
\begin{tikzpicture}
	\draw[->,thick] (0,-2.5) to (0,2.5);
	\draw[->,thick] (-2.5,0) to (2.5,0);
	\node at (-3.5, 1.3) {$\mu(\lambda(0)) = \mu(\lambda_* + \epsilon)$};
	\node at (-3.5, -1.3) {$\mu(\lambda(\pi)) = \mu(\lambda_* - \epsilon)$};
	\node at (-2.3, 0.25) {$-1$};
	\node at (-2,0) {$\times$};
	\draw[thick,dashed] (0,0) circle [radius=2];
	\draw[thick] (-1.8,0) circle [radius=1];
	\filldraw[thick] (-1.75,1) circle [radius=0.1];
	\filldraw[thick] (-1.75,-1) circle [radius=0.1];
\end{tikzpicture}
\caption{The curve $\mu(\lambda(t))$ is positively oriented, lies outside the unit circle for $0 < t < \pi$, and inside the circle for $\pi < t < 2\pi$. This implies that the lower point of intersection with the circle is $\mu(\lambda(0)) = \mu(\lambda_* - \epsilon)$, and the upper one is $\mu(\lambda(\pi)) = \mu(\lambda_* + \epsilon)$.}
\label{fig:mucurve}
\end{figure}

In other words, $t \mapsto \mu(\lambda(t))$ is a positively oriented closed curve that intersects the unit circle at $t = 0,\pi$, is outside the circle for $t \in (0,\pi)$, and inside for $t \in (\pi, 2\pi)$. Therefore, it must be as shown in Figure \ref{fig:mucurve}. This implies that $\mu(\lambda)$ passes through $-1$ in a negative (clockwise) direction as $\lambda$ increases through $\lambda_*$, which completes the proof.
%
\end{proof}


\subsection*{Acknowledgements}
The authors acknowledge the support of the BIRS FRG program \emph{Stability Indices for Nonlinear Waves and Patterns in Many Space Dimensions}, where some of this work was done. G.C. acknowledges the support of NSERC grant RGPIN-2017-04259. Y.L. was supported by NSF grants  DMS-1710989 and DMS-2106157, and would like to thank the Courant Institute of Mathematical Sciences and especially Prof.\ Lai-Sang Young for the opportunity to visit CIMS. A.S. was supported by NSF grant DMS-1910820.

\bibliographystyle{amsplain}
\bibliography{MaslovEvans}

\end{document}